\documentclass{amsart}
\usepackage[english]{babel}
\usepackage[utf8]{inputenc}
\usepackage[T1]{fontenc}    
\usepackage{hyperref}
\usepackage{amsthm}
\usepackage{amsmath}
\usepackage{amssymb}
\usepackage{mathrsfs}
\usepackage{graphicx}
\usepackage{subcaption}
\usepackage{stmaryrd}
\usepackage{dsfont}
\usepackage{yfonts}
\usepackage[margin = 2.5cm]{geometry}
\usepackage{caption}
\usepackage{ulem}
\usepackage{caption}
\usepackage{blkarray}
\usepackage{amsmath,amsfonts}
\usepackage[all]{xy}
\usepackage{tikz}
\usepackage{enumerate}
\usepackage[Algorithme]{algorithm}
\usepackage{tikz-cd}
\usepackage{varioref}
\usepackage{nicematrix}
\usepackage{bm}
\usepackage[toc,page]{appendix}
\usepackage{calligra}

\newcommand{\Tr}{\mathrm{Tr\ }}

\numberwithin{equation}{subsection}

\newtheorem{theorem}{Theorem}[section]
\newtheorem{definition}{Definition}[section]
\newtheorem{prop}{Proposition}[section]
\newtheorem{corollaire}{Corollary}[section]
\newtheorem{lemma}{Lemma}[section]

\newtheorem*{assumption}{Assumption}

\newtheorem{mainthm}{Theorem}

\newtheorem{mainconj}{Conjecture}

\labelformat{mainthm}{Theorem~#1}
\labelformat{maincor}{Corollary~#1}
\labelformat{mainconj}{Conjecture~#1}

\labelformat{lemma}{Lemma~#1}
\labelformat{theorem}{Theorem~#1}
\labelformat{prop}{Proposition~#1}
\labelformat{corollaire}{Corollary~#1}

\newcommand{\N}{\mathbb{N}}
\newcommand{\Z}{\mathbb{Z}}
\newcommand{\R}{\mathbb{R}}
\newcommand{\Q}{\mathbb{Q}}

\newcommand{\C}{\mathbb{C}}
\newcommand{\A}{\mathbb{A}}
\def\FF{\mathbb{F}}
\def\TT{\mathrm{\mathbf{T}}}
\def\SS{\mathrm{\mathbf{S}}}

\newcommand{\T}{\mathcal{T}}

\def\W{\mathcal{W}}
\def\O{\mathcal{O}}

\def\L{\mathcal{L}}
\def\H{\mathcal{H}}
\def\P{\mathcal{P}}
\def\K{\mathcal{K}}

\newcommand{\toeq}{\buildrel\sim\over\rightarrow}
\newcommand{\bs}{\backslash}

\newcommand{\inj}{\hookrightarrow}

\renewcommand{\a}{\alpha}
\renewcommand{\b}{\beta}
\renewcommand{\d}{\delta}
\newcommand{\g}{\mathfrak{g}}
\newcommand{\e}{\varepsilon}
\newcommand{\w}{\omega}
\newcommand{\ph}{\varphi}
\newcommand{\s}{\sigma}
\newcommand{\vs}{\varsigma}

\def\Om{\Omega}
\def\Si{\Sigma}
\def\t{\tau}
\def\l{\lambda}
\def\La{\Lambda}
\def\i{\iota}

\def\G{\Gamma}

\def\k{\mathfrak{k}}
\def\p{\mathfrak{p}}

\def\n{\mathfrak{n}}
\def\m{\mathfrak{m}}

\def\inf{\infty}
\def\x{\times}

\def\Ker{\ensuremath{\mathrm{Ker }}}

\def\Tr{\ensuremath{\mathrm{Tr}}}
\def\vol{\ensuremath{\mathrm{vol}}}
\def\diag{\ensuremath{\mathrm{diag}}}

\def\Ad{\ensuremath{\mathrm{Ad}}}
\def\Hom{\ensuremath{\mathrm{Hom}}}
\def\End{\ensuremath{\mathrm{End}}}

\def\GL{\ensuremath{\mathrm{GL}_3}}

\def\GLn{\ensuremath{\mathrm{GL}_n}}

\def\Gal{\ensuremath{\mathrm{Gal}}}
\def\Frob{\ensuremath{\mathrm{Frob}}}

\def\glr{\ensuremath{\mathfrak{gl}_3(\R)}}

\def\SO{\ensuremath{\mathrm{SO}(3)}}

\def\Ld{\ensuremath{\mathrm{L}}}
\def\Pd{\ensuremath{\mathrm{P}}}

\title{A divisibility of automorphic periods for the real quadratic base change of $\mathrm{GL}_3$}

\author{Tristan Ricoul}
\date{}

\begin{document}

\maketitle

\begin{abstract}
We prove a  $p$-adic divisibility between the automorphic periods of a cuspidal automorphic representation $\pi$ of $\mathrm{GL}_3(\Q)$ and the periods of its Arthur--Clozel base change to a real quadratic field $E$. As a corollary, we establish a divisibility predicted by the Bloch--Kato conjecture for the adjoint motive of $\pi$ twisted by an even quadratic character. This generalizes earlier works of Tilouine--Urban and Hida in the case of classical modular forms. The divisibility we prove involves a new kind of automorphic periods for conjugate self-dual cuspidal automorphic representations of $\mathrm{GL}_3(E)$, defined within the middle degree of the cuspidal cohomology instead of the top or bottom degrees. Moreover, we prove an  \textit{à la Hida} adjoint $L$-value formula for $\mathrm{GL}_3(E)$ that relates these newly defined middle-degree periods to the usual top and bottom automorphic periods.
\end{abstract}

\tableofcontents

\section{Introduction}

A natural question in the arithmetic theory of the Langlands program is to understand the behavior of automorphic periods under Langlands functoriality. For a given functorial transfer, the Bloch--Kato conjectures predict the existence of an integral relation between the periods attached to an automorphic representation and those of its transfer. Such a relation was first established by Tilouine--Urban \cite{TU22} in the case of the quadratic base change of a classical modular form. However, no analogous results were previously known for $\GLn$, when $n\geq 3$.

The purpose of this paper\footnote{ This paper, together with~\cite{TR_SBC}, originates from a preprint~\cite{R24} posted on arXiv in November~2024 and later split into two papers. This preprint was withdrawn and subsequently expanded into the author’s PhD thesis~\cite{thesis}.} is to investigate these period relations in the case of the real quadratic base change for $\GL$. More precisely, we introduce new automorphic periods arising from the middle-degree cuspidal cohomology of $\GL(E)$, where $E$ is a real quadratic field, and formulate a conjectural integral relation between these periods and the classical periods attached to a self-dual cuspidal representation of $\GL(\Q)$. We then prove one divisibility in this conjectural relation, providing the first result of this nature in higher rank.

\subsection{A conjectural integral period relation for the real quadratic base change of $\GL$ } Let $\pi$ be a cohomological cuspidal automorphic representation of $\GL(\Q)$, which is self-dual. Let $p$ be an odd prime number and let $\O$ be the integer ring of some sufficiently large $p$-adic field $\K$. Since the cuspidal cohomology of $\GL(\Q)$ is concentrated in degrees $2$ and $3$, the representation $\pi$ is associated with two $p$-integral automorphic periods $\Om_2(\pi)$ and $\Om_3(\pi)$, respectively called the bottom and top degree periods of $\pi$. These two periods are non-zero complex numbers, which compare the $\O$-integral structure of the cohomology of $\GL(\Q)$ given by algebraic topology to the $\O$-integral structure on the space of automorphic forms given by Whittaker models.

Let $E$ be a real quadratic number field and let $\Pi= \mathrm{BC}(\pi)$ be its strong base change to $\GL(E)$, whose existence, predicted by the Langlands principle of functoriality,  has been established by Arthur-Clozel \cite{SABC}. Then $\Pi$ is a cohomological automophic representation of $\GL(\A_E)$. Let $\chi_E$ be the quadratic character of $\A_\Q^\x/\Q^\x$ associated with the quadratic field $E$ by class field theory. Assume that $\pi \not\simeq \pi \otimes \chi_E$, so that $\Pi$ is cuspidal. The cuspidal cohomology of $\GL(E)$ is concentrated in degrees $q=4,5,6$. Thus, $\Pi$ is associated with two $p$-integral automorphic periods $\Om_4(\Pi)$ (the bottom degree period) and $\Om_6(\Pi)$ (the top degree period), defined within the bottom $b=4$ and top $t=6$ degrees of the cuspidal cohomology.

In general, as for any cuspidal automorphic representation of $\GLn$, one can only attach these two top and bottom degree periods to $\Pi$. The reason for that is that $\Pi$ appears with multiplicity $1$ in extremal (top $t=6$ and bottom $b=4$) degree cuspidal cohomology groups, but not in middle degree $m=5$. The first contribution of this paper is the introduction of a new kind of automorphic periods, $\Om_5(\Pi,\e,+)$ and $\Om_5(\Pi,\e,-)$, associated with any cohomological cuspidal representation $\Pi$ of $\GL(E)$ which is conjugate self-dual, i.e. such that $\Pi^\s = \Pi^\vee$. Note that the base change $\Pi = \mathrm{BC}(\pi)$ of a self-dual representation $\pi$ is conjugate self-dual. These two periods are called the $\e$-periods of $\Pi$ and are defined within the middle degree group of the cuspidal cohomology of $\GL(E)$. The reason we introduce these middle-degree periods is that we conjecture that they are related to the periods of $\pi$, whereas, for dimensional reasons which we will explain later, it is not clear that the extremal periods are. \\

According to various famous conjectures (see Deligne \cite{De79}, Beilinson \cite{beilinson}, and Bloch-Kato \cite{BK07}), motivic periods are expected to be closely related to special values of $L$-functions. Thus, by the automorphic/motivic analogy, automorphic periods should be related  to special $L$-values too, and decompositions of $L$-functions should in turn correspond to relations of automorphic periods. This leads us to formulate the following conjectural integral relation between the periods of $\pi$ and the newly defined middle-degree periods of $\Pi$ (see \cite{thesis} for a precise motivation):

\begin{mainconj}
\label{main_conj}
Let $\n \subset \O_E$ be the mirahoric level of $\Pi$. Assume that $p$ is prime to $6N_{E/\Q}(\mathfrak{n}) h_E(\mathfrak{n})D_E$, then:
$$
\Om_5(\Pi,\e,-) \sim \Om_2(\pi) \cdot \Om_3(\pi) \cdot \nu_\pi
$$
where $ \nu_\pi := \eta_\pi \cdot \eta_\pi(H^5_\TT)[+]^{-1} \in \O$ and we write $z_1 \sim z_2$, for $z_1,z_2 \in \C^\x$ if there exists $a \in \O^\x$ such that $z_1 = a \cdot z_2$.
\end{mainconj}

In the above formula, the term $\nu_\pi \in \O$ measures the defect between the Hecke congruence number of $\pi$ and the cohomological congruence number of $\pi$ on the degree-5 cuspidal cohomology of $\GL(E)$ (see \S\ref{congruence_numbers} for a precise definition of these congruence numbers). The congruence number $\eta_\pi$ has a deep number-theoretic meaning, as it can sometimes be related to the adjoint Selmer group of (the Galois representation associated with) $\pi$, and thus be can though as an automorphic analog of the latter (see \cite{thesis}[\S1.1.5]). Moreover, in some precise sense, the above conjecture is equivalent to (an automorphic analog of) the Bloch-Kato Tamagawa number conjecture for the adjoint motive of $\pi$ twisted by the quadratic character $\chi_E$ (see the discussion below \ref{thmB}). For $\mathrm{GL}_2$, Tilouine-Urban \cite{TU22} have proven, building on some previous work by Hida \cite{Hi99}, a similar period relation for both real and imaginary quadratic base changes of a classical modular form $f$. Their method ultimately relies on some numerical coincidence between the top degree of cuspidal cohomology of $\mathrm{GL}_2(E)$ (for $E$ real or imaginary quadratic) and the real dimension of some locally symmetric manifold (in their case, the modular curve). However this numerical coincidence no longer holds in general for $\mathrm{GL}_n$ when $n\geq 3$, and the situation is more delicate,  as the periods expected to appear in such a period relation are therefore no longer the extremal periods, but rather some hypothetical intermediate degree periods which are not even defined. This is the reason why, in order to formulate \ref{main_conj}  for $n=3$, we needed to introduce the new middle-degree automorphic periods $\Om_5(\Pi,\e,\pm)$ attached to conjugate self-dual automorphic representations of $\GL(E)$. As already said, these periods are defined within the degree-5 cuspidal cohomology group, and $5$ is precisely the real dimension of the $\GL(\Q)$-locally symmetric manifold inside the $10$-dimensional $\GL(E)$-locally symmetric manifold. \\

In this paper, we prove one divisibility of \ref{main_conj}. More precisely, we establish one divisibility of (an automorphic version of) the Bloch-Kato conjecture for the adjoint motive of $\pi$, twisted by the even quadratic character $\chi_E$. In doing so, we also prove an \textit{à la Hida} adjoint $L$-value formula for $\GL(E)$, which relates the adjoint $L$-value $L(\Pi,\Ad,1)$ to the newly defined middle-degree periods $\Om_5(\Pi,\e,\pm)$ of $\Pi$.

\subsection{Main result} In this paper we prove one divisibility of (an automorphic version of) the Bloch-Kato conjecture for the twisted adjoint motive $\Ad(\pi) \otimes \chi_E$, which yields one $\O$-divisibility of periods in \ref{main_conj}. Let us now present our results in more details.

Let $E$ still denote a real quadratic field, and let $\pi$ and $\Pi= \mathrm{BC}(\pi)$ be as above. Moreover, assume that $\pi$ is only ramified at primes which are split in $E$. As mentioned in the previous pararagraph, using the fact that $\Pi$ is conjugate self-dual, we define a new kind of middle-degree automorphic periods $\Om_5(\Pi,\e,+)$ and $\Om_5(\Pi,\e,-)$, called the $\e$-periods of $\Pi$ (see \S\ref{e_periods} for a precise definition). Finally, we assume that the Galois representation $\rho_\Pi$ attached to $\Pi$ is residually irreducible. Our main result is the following theorem:
 
 \begin{mainthm}[\ref{period_divisibility}]
\label{main_div_BC}
Let $\Pi = \mathrm{BC}(\pi)$ as above. Assume that $p$ is unramified in $E$, that the cohomological weight $\mu$ of $\pi$ is $p$-small, and that $p \nmid 6N_{E/\Q}(\mathfrak{n}) h_E(\mathfrak{n})D_E$ and $p \notin S_u$. Then, under the assumption $(H = \TT_\Q)$, the following divisibility holds:
$$
\nu_\pi\cdot \Om_2(\pi) \cdot \Om_3(\pi)\, \mid \, \Om_5(\Pi,\e,-)
$$
where $\nu_\pi := \eta_\pi \cdot \eta_\pi(H^5_{\TT})[+]^{-1} \in \O$, and we write $z_1 \;  \mid \; z_2$ for $z_1,z_2 \in \C^\x$ if there exists $a \in \O$ such that $z_2 = a \cdot z_1$.
\end{mainthm}

The set $S_u$ is a finite set of primes introduced in \ref{thmB} below. It is expected to be empty. As said earlier, $\nu_\pi \in \O$ measures the defect between the Hecke congruence number $\eta_\pi$ of $\pi$ and the congruence number $\eta_\pi(H^5_{\TT})[+]$ on some submodule $H^5_{\TT}$ of the degree-5 cuspidal cohomology of $\GL(E)$ (see \S\ref{congruence_numbers} for a precise definition of these congruence numbers). Note that its presence in the above divisibility makes the divisibility stronger. $D_E$ is the discriminant of $E$, $\n \subset \O_E$ is the (mirahoric) level of $\Pi$ and $h_E(\mathfrak{n})$ is the cardinal of the ray class group of level $\mathfrak{n}$. The condition of $p$-smallness excludes a finite number of small primes with respect to $\mu$. Overall, the conditions on $p$ exclude only finitely many prime numbers. Finally, assumption $(H = \TT_\Q)$ concerns the freeness of the (localized) top-degree cuspidal cohomology group of $\GL(\Q)$ over the localized Hecke algebra. Conditionally on some local-global compatibilities conjecture, we give sufficient conditions on $p$ and $\pi$ (i.e. that $\pi$ is $\n$-minimal, Fontaine-Lafaille at $p$ and has enormous residual Galois image) so that this assumption holds. \\

We also prove a similar divisibility of periods when the representation $\pi$ is no longer assumed to be self-dual. However, in this case, rather restrictive conditions must be assumed on the ramification of the representation $\pi$ in order to obtain the same result. Moreover, the periods appearing in the formula are no longer the $\e$-periods of $\Pi$ (as $\Pi$ is no longer conjugate self-dual in this situation), but rather another type of middle-degree periods, called the $\s$-periods, which are defined for self-conjugate representations. For the sake of clarity, we do not include this result in the  introduction, and we refer the reader to \S4.1 for its presentation. \\

As mentioned above, the integral period relation in \ref{main_conj} is  related to the Bloch-Kato Tamagawa number conjectures for the three motives $\Ad(\pi)$, $\Ad(\pi) \otimes \chi_E$ and $\Ad(\Pi)$. Let us now briefly explain why, in order to show how \ref{main_div_BC} follows from the other main results in this article. The automorphic base change corresponds, at the level of $L$-functions, to the following decomposition of adjoint $L$-functions:

\begin{equation}
\label{L_func_dec}
L(\Pi,\Ad,1) = L(\pi,\Ad,1) \cdot L(\pi,\Ad \otimes \chi_E,1) 
\end{equation}

Generalizing a formula first established by Hida \cite{Hi81a,Hi81b} in the case of modular forms, Balasubramanyam-Raghuram \cite{BR17} and Chen \cite{Che22} have established an adjoint $L$-value formula for $\pi$, which relates the adjoint $L$-value $L(\pi,\Ad,1)$ to the cohomological congruence number of $\pi$, and the product $\Om_2(\pi) \cdot \Om_3(\pi)$ of the periods of $\pi$. Their formula is also valid for $\Pi$, and relates the adjoint $L$-value $L(\Pi,\Ad,1)$ to the product $\Om_4(\pi) \cdot \Om_6(\pi)$ of the extremal periods of $\Pi$. Under some reasonable conditions on $\pi$ and $\Pi$, one may interpret there results (see \cite{thesis}[\S 1.1] for details, in particular Proposition 1.2)  as an automorphic analog of the Bloch-Kato Tamagawa number conjecture for the adjoint motives of $\pi$ and $\Pi$, involving the product of the two extremal periods of $\pi$ and $\Pi$ instead of the motivic periods appearing in the original formulation of the Bloch-Kato conjecture \cite{BK07}[Conjecture 5.15]. Consequently, given the decomposition (\ref{L_func_dec}), one needs to study the Bloch-Kato conjecture for the twisted adjoint motive of $\pi$ in order to relate the periods of $\pi$ to the periods of $\Pi$. \\

The main result of this paper is that we prove one divisibility of (an automorphic version of) the Bloch-Kato conjecture for the twisted motive $\Ad(\pi) \otimes \chi_E$ (see \ref{thmB} below). As our formula involves the middle-degree period $\Om_5(\Pi,\e,+)$ of $\Pi$, we need to establish an adjoint $L$-value formula involving the middle-degree periods of $\Pi$ as well, instead of the top and bottom periods of $\Pi$ appearing in Balasubramanyam-Raghuram and Chen formula. This is \ref{thmC} below.  Modulo some congruence module algebra, the periods divisibility in \ref{main_div_BC} then easily follows from these results, which we now present in more details. \\

\subsection{The Bloch-Kato conjecture for the twisted adjoint motive of $\pi$}

Let $\pi$ and $\Pi = \mathrm{BC}(\pi)$ be as above. Let $\mu \in X^+(T_3)$ be the cohomological weight of $\pi$, so that $\mu_E = (\mu,\mu)$ is the cohomological weight of $\Pi$. Let $\mathfrak{n}$ be the mirahoric level of $\Pi$ and $K_f = K^*_1(\n)$ be the mixed mirahoric subgroup of level $\mathfrak{n}$. We denote by $\TT_E$ be the spherical Hecke algebra of level $K_f$, localized at the maximal ideal $\m_\Pi$ corresponding to $\Pi$. Let $M$ denote the degree $5$ cuspidal cohomology group localized at $\m_\Pi$:
$$
M = H_{cusp}^5(Y_E(K_f), \L_{\mu_E}(\O))_{\m_\Pi}
$$
where $Y_E(K_f)$ is the adelic variety of level $K_f$ for $\GL(E)$, and $\L_{\mu_E}(\O)$ is a locally constant sheaf associated to the irreducible algebraic representation $L_{\mu_E}$ of $\GL(E)$ of highest weight $\mu_E$. Then, $M$ is a $\TT_E$-module endowed with a semi-linear action of some involution denoted $\e$. Let $\eta_{\Pi}^\#(M^*)$ be the relative (for the base change) congruence number of $\Pi$ on the dual module $M^* = \Hom_\O(M,\O)$. See \S\ref{relative_congruence_numbers} for its precise definition. We prove the following result:

\begin{mainthm}[\ref{CBC_divisibility}]
\label{thmB}
Let $\pi$ and $\Pi = BC(\pi)$ as above. Assume that the Galois representation associated to $\Pi$ is residually absolutely irreducible. Assume that the cohomological weight $\mu_E$ of $\Pi$ is $p$-small and that $p$ does not divide $6N_{E/\Q}(\mathfrak{n}) h_E(\mathfrak{n})D_E$ and $p \notin S_u$. Then:
$$
\eta_\Pi^\#(M^*)[+] \quad | \quad \frac{\Lambda^{imp}(\pi,\Ad \otimes \chi_{E},1)}{\Om_5(\Pi,\e,+)}
$$
where $\eta_\Pi^\#(M^*)[+]$ is the $+$-part of $\eta_\Pi^\#(M^*)$ for the action of $\e$.
\end{mainthm}

In the above theorem, $\Lambda^{imp}(\pi,\Ad \otimes \chi_{E},s)$ is the imprimitive completed twisted adjoint $L$-function of $\pi$, defined as an Euler product including the $\G$-factor at the archimedean places, and imprimitive local factors at places of ramification for $\Pi$. The set $S_u$ is defined to be the finite set of primes which divides a certain nonzero algebraic number $u \in \bar{\Q}^\x$ defined in \ref{linear_formula}. The constant $u$ is a product of local factors at $2$ and at primes which are ramified in $E$ and only depends on the local components of $\Pi$ above these primes. The constant $u$ is expected to be $1$ (and hence $S_u$ is expected to be empty) but is hard to compute because it involves local linear forms which are defined indirectly and have no explicit formula. Ultimately, the calculability of this constant depends on an analog of the Fundamental Lemma of Jacquet and Ye (see \cite[Theorem 3.1]{FLO12}) for ramified quadratic extension of non-Archimedean local fields (of characteristic zero) and for 2-adic fields. However, we prove that $u$ is an algebraic number (in fact, our result is a little bit more precise, see \ref{algebraicity_u2ram} in \S\ref{paragraph_algebraicity}). Following the same method, it should even be possible to prove that $u$ belongs to the rationality field $\Q(\Pi_f)$ of $\Pi_f$. Finally, note that this theorem implies in particular that the normalized twisted adjoint $L$-value appearing on the right-hand side belongs to $\O$.

We also prove a similar divisibility when the representation $\pi$ is no longer assumed to be self-dual. However, in this case, rather restrictive conditions must be imposed on the ramification of the representation $\pi$ in order to obtain the same result, and the periods appearing at the denominator of the $L$-value is different. For the sake of clarity, we do not include this result in the  introduction, and we refer the reader to \S4.1 of the paper for its presentation. 
 \\

Let us now briefly explain how \ref{thmB} is related to the Bloch-Kato conjectures. Let $M_\pi$ be the pure irreducible Grothendieck motive over $\Q$ conjecturally associated to $\pi$ according the Langlands-Clozel conjecture, and let $\rho_\pi$ be the correponding $p$-adic Galois representation, whose existence has been established by  \cite{HLTT16} and \cite{Scholze15}. Under some Calegari-Geraghty \cite{CG18} hypothesis for $\Pi$, one can show (as for the adjoint motive $\Ad(M_\pi)$ of $\pi$ in \cite{thesis}[\S1.1], see also \cite{TU22}[\S5.4]) that the Selmer group $\mathrm{Sel}(\Ad(\rho_\pi) \otimes \chi_E) :=  H_f^1(\Q,\Ad(\rho_\pi) \otimes \chi_E \otimes \K/\O)$ is finite, and that the Bloch-Kato Tamagawa number conjecture for the motive $W_\pi = \Ad(M_\pi) \otimes \chi_E$ is equivalent to the following formula:
$$
\eta_\pi^\# \sim \frac{\La(\pi, \Ad \otimes \chi_E,1)}{\Om(W_\pi) \cdot \d_\pi^\#}
$$
where $\eta_\pi^\#$ is the relative congruence number of $\pi$ with respect to the automorphic base change from $\Q$ to $E$, and $\d_\pi^\# := \# \mathrm{Sel}(\Ad(\rho_\pi) \otimes \chi_E) \cdot (\eta_\pi^\#)^{-1} \in \O$ is the so-called relative Wiles defect. Note that in general one has that $\d_\pi^\#$ divides the ratio $\d_\Pi \cdot \d_\pi^{-1}$ of the Wiles defects of $\pi$ and $\Pi$ (see \cite[\S1.1.5]{thesis} for their definition). Thus, up to the defect $\nu_\pi^\# := \eta_\pi^\# \cdot \eta_\Pi^\#(M^*)[+]^{-1}$ (which could be precisely studied) between the two relative congruence numbers $\eta_\pi^\#$ and $\eta_\Pi^\#(M^*)[+]$, \ref{thmB} can be though as proving one divisibility of an automorphic analog of the Bloch-Kato Tamagawa number conjecture, involving the automorphic period $\Om_5(\Pi,\e,+)$, instead of the motivic period $\Om(W_\pi)$. \\

\subsection{Definition of the $\e$-periods and proof of \ref{thmB}} Since this is a general method for establishing period relation, we now outline the proof of \ref{thmB}. In doing so, we explain how the $\e$-periods are defined. The key point is to construct an $\O$-linear form on the cuspidal cohomology of $Y_E(K_f)$ with coefficients in $\O$, which vanishes on non-base-changes and takes the expected twisted adjoint $L$-value on (the differential form associated to) some newform $\phi_f^\circ$ in $\Pi_f$. The idea is to construct this linear form by giving a cohomological interpretation of the Jacquet-Ye integral period:
$$
\P_{U}: \phi\in \Pi \mapsto \int_{U(\Q) \backslash U(\mathbb{A})} \phi(g) dg 
$$
where $U$ is some quasi-split unitary group embedded in $\GL(E)$. The vanishing property on non-base-changes is ensured by a result of Feigon, Lapid and Offen \cite{FLO12}, refining some earlier results of Jacquet \cite{J05a} \cite{J10}. They have shown that this period is non-zero for some $\phi \in \Pi$ if and only if $\Pi$ is the base change of some cuspidal automorphic representation of $\GL(\Q)$.

The cuspidal cohomology of $\GL(E)$ is concentrated in degree $q=4,5,6$. Since the  modular submanifold associated to $U$ is $5$-dimensional, the integral period $\P_{U}(\phi)$ should be cohomologically interpreted as integrals involving the differential $5$-form $\d(\phi)$ associated to $\phi$ through some Eichler-Shimura map $\d$. Here we face two problems. The first problem concerns the normalisation of the Eichler-Shimura map $\d$. In fact, this map is only defined up to a complex scalar since its definition depends on the choice of a generator of certain $(\g,K_\inf)$-cohomology groups. Since we want to compare the $p$-adic integral stucture given by the Whittaker model on $\Pi_f$ to the $p$-adic integral structure on the cuspidal cohomology, we need to make this choice canonical. For $\mathrm{GL}_2$, this has been done by Hida (see \cite[Section 4]{Hi94}), and the maps are normalized so that they preserve the completed $L$-functions on both sides. To the best of our knowledge, this has not been done for $\GLn$ yet. However, in the special case of $\GL$ over a totally real field, Chen \cite{Che22} has explicited some canonical choices for these generators, and shown that they give the correct adjoint $\G$-factors. The second problem is that the Eichler-Shimura map $\d$ lands in:
$$
H^5_{cusp}(Y_E(K_f), \L_{\mu_E}(\C))[\Pi_f] 
$$
which is a $2$-dimensional $\C$-vector space. Hence it may not be possible to normalize $\d$ by some complex scalar (a period) so that it lands in the canonical $\O$-structure of the above $\C$-vector space. This is why the Betti-Whittaker periods of $\Pi$ are only defined using the top degree $t=6$ and the bottom degree $b=4$ cuspidal cohomology groups, whose $\Pi_f$-part are $1$-dimensional. However, we remark that there exists an involution $\e$ acting on $\GL(E)(\A)$, which induces an involution on the cuspidal cohomology of $\GL(E)$ of level $K_f$, when $K_f$ is $\e$-invariant. Moreover if $\Pi$ is conjugate self-dual, the $\Pi$-part of the cuspidal cohomology is endowed with a non-trivial action of $\e$. Finally, this action preserves the $\O$-integral structure. We can then define two canonically normalized Eichler-Shimura maps:
$$
\d_\e^\pm: \Pi_f^{K_f} \to H^5_{cusp}(Y_E(K_f), \L_{\mu_E}(\C))[\Pi,\e = \pm]
$$
going into the $\pm$-eigenspace for $\e$. The right-hand side is a complex $1$-dimensional vector space admitting a canonical integral $\O$-structure, so we can attach two periods $\Om(\Pi,\e,\pm)$ to $\Pi$ provided it is conjugate self-dual (see paragraph~\ref{e_periods} for details). We note that in order to define $\Om(\Pi,\e,\pm)$, we need to assume that the ramification of $\Pi$ is located at places of $E$ above split primes. This condition ensures that there exists a newform theory for a family of $\e$-invariant open compact subgroups $K_f$.

We thus obtain a linear form on the middle-degree cuspidal cohomology, with the desired vanishing property. The last step consists in computing the explicit value of this linear form on the $5$-differential form $\d^{+}_\e(\phi_f)$ associated to some normalized $K_f$-newform $\phi_f \in \Pi_f$ used to define the periods $\Om_5(\Pi,\e,\pm)$ of $\Pi$. This is done using a formula due to Jacquet \cite{J01}, which expresses the Jacquet-Ye period as the twisted adjoint $L$-value $L(\pi,\Ad \otimes \chi_E,1)$. Jacquet's formula involve uncomputed ramified and archimedean local factors, which we thus have to compute. The archimedean factor is computed by adapting some computations of Chen \cite{Che22}, and give the (twisted) adjoint $\G$-factors. The ramification local factors are computed by using some explicit expressions for the essential vector given separately by Miyauchi and Matringe (see paragraph~\ref{mirahoric_theory}), and are equal to the imprimitive local $L$-factors. Finally, the divisibility in \ref{thmB} is obtained by applying the key \ref{lf_lemma}, which is an adaptation of \cite[Lemma 2.9]{TU22} to the context of Hecke-modules with a semi-linear involution. \\

\subsection{The Bloch-Kato conjecture for the adjoint motive of $\Pi$} As already explained, in order to deduce \ref{main_div_BC} from \ref{thmB}, we must relate the adjoint $L$-value $L(\Pi,\Ad,1)$ appearing in the decomposition (\ref{L_func_dec}) to the automorphic periods of $\Pi$. We have mentioned that there already exists such a formula, first established by Balasubramanyam and Raghuram \cite{BR17} up to some uncomputed archimedan factor, which has latter been computed by Chen \cite{Che22}. However, this formula relates the adjoint $L$-value $L(\Pi,\Ad,1)$ to the top and bottom periods of $\Pi$. As \ref{thmB} involves the new middle-degree periods $\Om(\Pi,\e,\pm)$, we thus need to prove a similar formula involving these periods. We keep the notation from \ref{main_div_BC}. Let $\eta_{\Pi}(M)$ be the congruence number of $\Pi$ on the degree-5 inner cohomology group $M$ (see \S\ref{congruence_numbers} for its precise definition). Our result is the following theorem: 
\begin{mainthm}[\ref{adjoint_L_value}]
\label{thmC}
Let $\Pi$ be a cohomological automorphic cuspidal representation of $\GL(\A_E)$ which is conjugate self-dual. Assume that $\Pi$ is only ramified above split primes. Under the same hypothesis on $\Pi$ and $p$ as in \ref{main_div_BC}, one has:
$$
\eta_{\Pi}^\pm(M)[\pm] \quad \sim \quad \frac{\Lambda^{imp}(\Pi,\Ad,1)}{\Om_5(\Pi,\e,\pm) \cdot \Om_5(\Pi^\vee,\e,\mp)}
$$
where $\eta_{\Pi}(M)[\pm]$ is the $\pm$-part for the action of $\e$ on $\eta_{\Pi}(M)$.
\end{mainthm}

In the above theorem, $\Lambda^{imp}(\Pi,\Ad,s)$ is the imprimitive completed twisted adjoint $L$-function of $\Pi$. In particular, \ref{adjoint_L_value} implies that the normalized adjoint $L$-values appearing on the right-hand sides belongs to $\O$. Since $\eta_{\l_\Pi}(H^5)[\pm]$ divides the Hecke congruence number $\eta_{\l_\Pi}$ of $\Pi$, the above equality implies that if the $\wp$-adic valuation of the right hand side is non-zero  (where $\wp$ is the prime ideal of $\O$), then there exists a cuspidal automorphic representation $\Pi' \in \mathrm{Coh}(G_E,\mu_E,K_f)$ distinct from $\Pi$ which is congruent to $\Pi$ modulo $\wp$. However, we stress that $\eta_{\l_\Pi}(H^5)[\pm]$ may not be equal to $\eta_{\l_\Pi}$, since $H^5$ is not free as a $\TT$-module. \\

Combining \ref{adjoint_L_value} with Balasubramanyam-Raghuram's formula \cite[Theorem A]{BR17}, one can deduce an integral relation between the middle-degree periods and the top-bottom periods of $\Pi$, under Calegari-Geraghty setting. Assume that conjecture $\mathrm{(Gal_\m)}$ holds so that the representation $\rho_\m : \Gal(\overline{E}/E) \to\GLn(\TT_E)$ associated with $\Pi$ exists. We refer to paragraph \ref{middle_vs_extremal} for the statement of this conjecture and of the subsequent ones, as well as for the definition of condition~(CG) appearing in the corollary below. Assuming conjectures $\mathrm{(LGC_\m)}$ and $\mathrm{(Van_\m)}$, we obtain the following corollary:

\begin{corollaire}[\ref{relation_middle_top_bottom}]
\label{relation_middle_top_bottom}
Assume $p >2$. Let $\Pi$ be a cohomological cuspidal automorphic representation of $\GL(\A_E)$ which is conjugate self-dual. Assume that $\Pi$ is only ramified above split primes. Assume that $\rho_\m$ satisfies $\mathrm{(CG)}$, and that the cohomological weight $\mu_E \in X^+(T_E)$ of $\pi$ is $p$-small. Then:
$$
\Om_5(\Pi,\e,\pm) \cdot \Om_5(\Pi^\vee,\e,\mp) \sim \Om_4(\Pi) \cdot \Om_6(\Pi^\vee) \cdot  \nu_{\Pi}^\pm
$$
where $\nu_{\Pi}^\pm := \eta_{\Pi} \cdot \eta_{\Pi}(M)[\e = \pm]^{-1} \in \O$.
\end{corollaire}

It is worth noting that there is no need to exclude specific primes (except for $\mu$-small primes), as the constants in the two adjoint $L$-value formulas cancel each other out. \\

The proof of \ref{thmC} (like the proof of all similar formulas) follows the line of the proof in Hida's pioneering paper \cite{Hi81a}. The first step is to express the Petersson product as the value at $s=1$ of the adjoint $L$-function of $\Pi$. In the case of a classical modular form, this formula is du to Shimura \cite{Shi76}. For a cuspidal automorphic representation of $\GLn$, it is a formula of Jacquet and Shalika \cite[Section 4]{J-S81-I}. However, Jacquet-Shalika's formula is only true up to ramified and archimedean local factors. Thus we need to compute these ramified local factors for the newforms $\phi_f \in \Pi_f$ and $\phi_f' \in \Pi^\vee_f$ used to define the $\e$-periods of $\Pi$ and $\Pi^\vee$.

The next step is to interpret cohomologically the Petersson product as a Poincaré pairing $[\cdot,\cdot]$ on the middle-degree-5 cuspidal cohomology groups, using the Eichler-Shimura maps. As explained earlier, the definition of these maps depends on the explicit choice of a generator $[\Pi_\inf]$ of the $(\g,K_\inf)$-cohomology groups of the archimedean part $\Pi_\inf$ of $\Pi$ twisted by the coefficients module $L_{\mu_E}(\C)$. The obtained formula, relating the adjoint $L$-value to the Poincaré pairing, involves an archimedean factor depending on the choice of this generator. In general, one can just show that this factor is non-zero. However, in the case of $\GL$ over a totally real field, Chen \cite{Che22} has computed this factor for some choice of $[\Pi_\inf]$ and shown that it is equal to the appropriate archimedean $\G$-factor. \\

When the $\Pi$-part of the cohomology groups is a 1-dimensional $\C$-vector space, one can relate the Poincaré pairing to the cohomological congruence number of $\Pi$ as follows. By normalizing the cohomological classes associated with $\phi_f$ and $\phi_f'$ by the periods of $\Pi$ and $\Pi^\vee$, we get, by definition of the periods, two $\O$-generators $\d$ of the $\Pi$-part and $\d'$ of the $\Pi^\vee$-part of the cohomology with coefficients in $\O$. The pairing $[\cdot,\cdot]$ is perfect on the full cohomology groups. However, its restriction to the cartesian product of the $\Pi$-part and the $\Pi^\vee$-part is no longer a perfect pairing. Its discriminant, which is $[\d,\d']$, is in fact precisely equal to the cohomological congruence number of $\Pi$.

Nevertheless, in our case, the $\Pi$-isotypic part of the cohomology group $M = H^5_{cusp}(Y_E(K_f),\L_{\mu_E}(\C))$ is of dimension $2$. Consequently the above reasoning does not apply as it stands, because one may not be able to normalize integrally a cohomology class with a scalar. However, when the representation  $\Pi$ is conjugate self-dual, we have explained that the $\Pi$-part of $M$ is equipped with a non-trivial involution. This involution can be used to decompose $M[\Pi]$ has a direct sum of two $1$-dimensional vector spaces $M[\Pi;\pm]$, endowed with an integral $\O$-structure. We choose two generators $[\Pi_\inf]_\pm$ of the $(\g,K_\inf)$-cohomology so that the Eichler-Shimura maps take values in these subspaces. We then normalize by appropriate periods the cohomological classes associated with $\phi_f$ and $\phi_f'$ through these maps, and we thus obtain two generators $\d_\pm$ et $\d'_\pm$ of these integral structures. Unfortunately, the two submodules $M[\pm]$ are not Hecke modules so we still need to adapt the above method. To do this, we remark that $\e$ also induces an involution on the congruence module $C_{\l_\Pi}(M)$. This involution is non-trivial and can be used to decompose the congruence module into two submodules $C_{\l_\Pi}(M)[\pm]$. Then one can prove, using some equivariance property of the Poincaré pairing with respect to the involution $\e$, that the pairing $[\d_\pm,\d'_\mp]$ is equal to the Fitting ideal $\eta_{\l_\Pi}(M)[\pm]$ of these submodules. \\

\subsection{Comments}

We conclude this introduction with a few comments. A natural question is how to prove the reciprocal divisibility in \ref{main_div_BC} (or equivalently in \ref{thmB}). In the few case where such a divisibility is established, the proof generally involves some non-vanishing modulo $p$ result. For instance, in the case of the Yoshida lift of two cuspidal modular forms, Liu and Hsieh (in some work in preparation) have been able to establish such a divisibility by using the non-vanishing modulo $p$ of the Yoshida lift's Fourier coefficients. In \cite{TU22} the reciprocal divisibility is proven by using a theorem of Cornut-Vatsal \cite{Cornut02,Vat02} on the non-vanishing modulo $p$ of twisted standard $L$-values for weight-$2$ modular forms, generalized by Chida-Hsieh \cite{CH18} to $p$-small weights. The method of Tilouine and Urban should generalize to our case and yield the reciprocal divisibility, provided that some generalization of Cornut-Vastal is established for $\GL$. However, to the knowledge of the author, no such result is known beyond $\mathrm{GL}_2$.  \\

We now discuss to what extent the results of the present paper could be generalized to larger $n$ and general extensions $E/F$ of number fields. The problem as $n$ grows and $E$ gets bigger is that the size $\ell_E$ of the cuspidal range for $\GL(E)$ becomes greater than $2$. In this situation, this is not clear how to define intermediate degree periods as the dimension of the intermediate cohomology groups can be strictly bigger than $2$, whereas the number of involutions available for cutting out $1$-dimensional subspaces remains the same. However, the methods presented in this paper may apply to other cases where $\ell_E=2$, which include:
\begin{itemize}
\item  Case 1: $n=2$ and $E$ is an imaginary bi-quadratic number field;
\item  Case 2: $n=4$ and $E$ is real quadratic;
\item  Case 3: $n=3$ and $E$ is imaginary quadratic.
\end{itemize}
The first two cases are similar to the situation in this paper, in the sense that the archimedean places of $F$ are split in $E$. Case 1 has been studied by Hida \cite[Section 7]{Hi99}. Our results for $\GL$ should easily be extended to Case 2, provided that one is able to extend Chen’s archimedean computations to $\mathrm{GL}_4$. Case 3 is investigated by Balasubramanyam and Tilouine \cite{BT}. \\

Finally, for $n \geq 3$, one can also consider the stable base change to $\GLn(E)$ from the $n$-variables quasi-split unitary group $U_{E/F}$ associated with the quadratic extension $E/F$. In a separate paper \cite{TR_SBC}, we study this situation when $n=3$ and $E$ is real quadratic. We prove a periods divisibility result for this automorphic transfer using methods similar to those developed in the present paper. The situation for the stable base change is somehow symmetric to that of the classical base change considered here, and we also refer to the introduction of the author’s thesis \cite{thesis} for a unified presentation of these two settings.\\

\subsubsection{Toward the construction of a $p$-adic $L$-function} To conclude this introduction, we mention that the results of this paper can be used to construct $p$-adic $L$-functions interpolating the values of the twisted adjoint $L$-function $L(\pi,\Ad \otimes \chi_E,1)$ as $\pi$ varies in a Hida family of self-dual cuspidal automorphic representations of $\GL(\Q)$. Let us briefly explain how. As explained above, \ref{thmB} is proven by constructing an $\O$-linear form $\Ld_\mu : H_{cusp}^5(Y_E(K_f), \mathcal{L}_{\mu}(\O)) \to \O$ on the cuspidal cohomology of the adelic variety of $\GL(E)$ of weight $\mu$. This linear form has the particular proprety that:
$$
\Ld_\mu(\d(\phi^\circ)) = L(\pi,\Ad \otimes \chi_E,1)
$$ 
where $\pi$ is a self-dual cuspidal automorphic representation of $\GL(\A_\Q)$ of weight $\mu$ and $\d(\phi^\circ)$ is some $5$-differential form associated, by some Eichler-Shimura map $\d$, to a normalized newform $\phi^\circ \in \Pi := \mathrm{BC}(\pi)$. Thus, one should be able to construct a $p$-adic $L$-function interpolating the special $L$-value of the right-hand side by constructing a linear form $\Ld_{\Lambda}$ on some module over the Iwasawa algebra $\Lambda$ of $\GL(E)$, which interpolates the linear form $\Ld_\mu= H_{cusp}^5(Y_E, \mathcal{L}_{\mu}(\O)) \to \O$ as the algebraic weight $\mu$ varies $p$-adically over the weight space. Such a construction was carried out for Bianchi modular forms in \cite{Lee21}. As elsewhere in this thesis, the main difficulty in such a construction in the case of $\GL$ lies in the need to construct such a linear form on the middle-degree cohomology group. It is natural to try to contruct this $\La$-adic linear form on the ordinary Hida cohomology $H^5_{ord}(K_f)$, which is naturally a module over the Iwasawa algebra $\La$. However, because of the so-called \textit{non-abelian Leopoldt conjecture}, formulated by Hida \cite{Hi98} and Urban, the ordinary Hida cohomology $H^*_{ord}(K_f)$ is expected to be concentrated in top degree $t=6$ (see \cite[Proposition 9]{TU22}). The idea to overcome this problem is to replace the Iwasawa algebra $\La$ with a quotient algebra $\La^{\circ}$, which we called the \textit{restricted} Iwasawa algebra and that only captures pure algebraic weights $\mu$. One can then define a \textit{restricted} Hida cohomology $H^*_{ord}(K_f)^\circ$, which is a module over the restricted Iwasawa algebra $\La^{\circ}$. Building on \cite{KT16}, one can prove a control theorem for $H^5_{ord}(K_f)^\circ$ and show in particular that it is non-zero. It is then possible to construct a $\La^\circ$-linear form on $H^5_{ord}(K_f)^\circ$ which have the expected interpolation property and from which it is possible to construct the twisted adjoint $p$-adic $L$-function for $\GL$.

\subsection{Outline} This paper is organized as follows. Section~\ref{section_periods} introduces the main objects and defines the middle-degree $\e$-periods associated with conjugate self-dual automorphic representations of $\GL(E)$. In Section~\ref{part_adjoint}, we prove \ref{thmC}. Section~\ref{part_CBC} recalls basic definitions and results about  base change for $\GL$, proves \ref{thmB}, and establishes the main period divisibility of \ref{main_div_BC}. The final paragraph of this section is devoted to the presentation of results for non-self-dual representations.

\subsection{Acknowledgements}
The author is deeply indebted to his advisor, J. Tilouine, for his constant help and support during the preparation of this paper. The author also would like to thank B. Balasubramanyam, S.-Y. Chen, E. Ghate, G.Grossi, H. Hida, M.-L.Hsieh, E. Lapid, Z. Liu, N. Matringe, M. Moakher, D. Prasad, K. Prasanna and E. Urban for stimulating conversations, comments and suggestions regarding this work. This paper was partly completed during two visits at IISER Pune on an invitation by B. Balasubramanyam and a visit at IIT Bombay on an invitation by D. Prasad, as well as during a visit at the NCTS on an invitation by M.-L. Hsieh. The author would like to thank them heartily for their hospitality. I would also like to thank Eknath Ghate and Sandeep Varma, organizers of the \textit{p-adic Methods in Number Theory conference} at TIFR in Mumbai, for giving me the opportunity to present my work there.

\subsection{Notation and conventions}
\label{notations}

Throughout this paper, we consider an odd prime number $p$. We fix an embedding $j: \overline{\Q} \inj \C$ and an isomorphism $ j_p: \overline{\Q}_p \simeq \C$. Let $\K \subset \overline{\Q}_p$ be some sufficiently large finite extension of $\Q_p$. We denote by $\O$ its integers ring, and $\wp$ its prime ideal.

Using $j_p$, we can see complex numbers as elements of $\overline{\Q}_p$. Thus, if $z_1$ and $z_2$ are two non-zero complex numbers, we write $z_1 \sim z_2$ (resp. $z_1 \mid z_2$) if the quotient $z_2/z_1$ lies in $\O^\x$ (resp. in $\O$).

Unless otherwise specified, $E$ will always denote a real quadratic field, and $\s \in \mathrm{Gal}(E/\Q)$ is the non-trivial element in $\mathrm{Gal}(E/\Q)$. We fix an embedding $\tau: E \inj \overline{\Q}$, so that via $j$ we can identify the set $\{\tau,\s\tau\}$ with the set of archimedean places of $E$. Moreover, we denote by $\chi_E$ the quadratic Dirichlet character associated with $E$.

\subsubsection{Number fields, groups and additive characters}
Let $F$ be some number field. We denote by $\O_F$ its ring of integer, $\mathfrak{d}_F$ its different and $D_F$ its discriminant. If $v$ is any place of $F$, we denote by $F_v$ the completion at $v$. When $v$ is finite, we denote by $\O_{F_v}$ (or simply $\O_v$ when the context is clear) its ring of integers, $\wp_v$ its prime ideal and $\varpi_v$ an uniformizer of $F_v$. Let $ S_\inf(F)$ be the set of archimedean places of $F$, $\A_F$ be the adeles of $F$ and $\A_{F,f}$ its finite part. When $F=\Q$, we simply write $\A$ and $\A_f$. \\

Let $n\geq 1$ be an integer. Let $\GLn$ denote the general linear group of dimension $n$, $B_n$ its standard Borel, $N_n$ and $T_n$ respectively the maximal unipotent and the maximal torus of $B_n$. Let $Z_n$ denote the center of $\GLn$. We also denote by $G_F$ the $\Q$-algebraic group $\mathrm{Res}_{F/\Q}(\mathrm{GL}_{n/F})$, whose $A$-points are given, for any $\Q$-algebra $A$, by:
$$
G_{F}(A) = \mathrm{GL}_n(F\otimes_\Q A).
$$
Let $K_n$ be the identity component of the maximal compact subgroup modulo center of $\GLn(\R)$:
$$
K_n= \mathrm{SO}(n) \R_+^\x
$$
where $\R^\x$ is seen as the center of $\GLn(\R)$. Let $\g_{n}$ and $\k_{n}$ be the Lie algebras of $\GLn(\R)$ and $K_n$, and let $\p_{n} := \g_{n}/\k_{n}$. If $F$ is a totally real field, let $K_{\inf} = \prod_{\tau \in  S_\inf(F)} K_n
$ be the identity component of the maximal compact subgroup modulo center of $G_F(\R)$ and let $\g = \oplus_\tau \g_{n}$ and $\k =\oplus_\tau \k_{n}$ be the Lie algebras of $G_F(\R)$ and $K_\inf$, and let $\p := \g /\k$. Let $\g_\C$, $\k_\C$ and $\p_\C = \g_\C /\k_\C$ denote their complexification, and $\p_\C^*$ the dual space of $\p_\C$. \\

We denote by $\psi_\Q = \bigotimes_v \psi_{\Q_v} : \Q\bs\A \to \C^\x$ the additive unramified character defined by:
$$
\begin{aligned}
\psi_{\Q_p}(x) &= \exp(-2\sqrt{-1}\pi[x]_p), \quad x \in \Q_p \\
\psi_\R(x) &= \exp(2\sqrt{-1}\pi x), \quad x \in \R \\
\end{aligned}
$$
Here $[x]_p := \sum_{k=-m}^{-1} c_k p^k$ is the $p$-fraction part of $x = \sum_{k=-m}^\inf c_kp^k$. Let $\psi_F : F \bs \A_F\to \C^\x$ be the standard non-trivial character  on $\A_F$, defined by $\psi_F = \psi_\Q \circ \Tr_{F/\Q}$. If $\psi : F \bs \A_F\to \C^\x$ is any additive character, it defines a character of $N_n(F)\bs N_n(\A_F)$, also denoted by the same symbol $\psi$, and defined by:
$$
\psi(n) = \psi(n_{1,2} + \dots + n_{n-1,n}), \quad n \in N_n(\A_F)
$$ 

\subsubsection{Haar measures}
\label{measures}
Write $\psi_F = \bigotimes_v \psi_v$ for the standard additive character on $\A_F$. The additive Haar measures on $\A_F$ are fixed as follows. Let $v$ be a place of $F$. If $v$ is non-archimedean, the Haar measure $dx_v$ is the self-dual measure on $F_v$ with respect to $\psi_v$. The volume of $\O_v$ with respect to $dx_v$ is:
$$
\vol(\O_v,dx_v) = N(\mathfrak{d}_v)^{-1/2}
$$
where $\mathfrak{d}_v$ is the different of $F_v$. If $v$ is a real archimedean place, the Haar measure $dx_v$ on $F_v$ is the Lesbesgue measure. The global Haar measure on $\A_F$ is then $dx = \prod_v dx_v$. Note that:
$$
\vol(\A_F/F,dx) = 1.
$$

The local Haar measures on the idèles $\A_F^\x$ are fixed as follows. Let $v$ be a place of $F$. If $v$ is non-archimedean, corresponding to some prime ideal $\wp$, we set the Haar measure $d^\x x_v$ on $F_v^\x$ to be:
$$
d^\x x_v = (1-N\wp^{-1})^{-1} \frac{dx_v}{|x_v|}
$$
The volume of $\O_v^\x$ with respect to $d^\x x_v$ is  $\vol(\O_v^\x,d^\x x_v) = N(\mathfrak{d}_v)^{-1/2}$. If $v$ is a real archimedean place, we normalize the Haar measure $d^\x x_v$ on $F_v^\x = \R^\x$ by $d^\x x_v = |x_v|^{-1} dx_v$. The global Haar measure on $\A_F^\x$ is then $d^\x x = \prod_v d^\x x_v$. Note that this is not the Tamagawa measure on $\A_F^\x$. In particular, one has:  \\
$$
\vol(\A_F^1/F^\x,d^\x x) = \mathrm{Res}_{s=1} \zeta_E(s).
$$
where $\A_F^1$ denotes the subgroup of idèles of norm $1$ and $\zeta_E$ is the Dedekind zeta function of $E$. \\

\section{Eichler-Shimura maps and middle-degree automorphic periods}
\label{section_periods}

In this section, we first present a new vector theory for a family of $\e$-invariant open compact subgroups of $\GLn$. We then introduce the cohomological framework and the automorphic objects needed throughout the paper. Finally, we construct the middle-degree Eichler–Shimura maps for $\GL(E)$, and define the middle-degree $\e$-periods attached to a conjugate self-dual cohomological cupsidal representation of $\GL(E)$.

\subsection{Mixed mirahoric new vector theory for $\GLn$}
\label{mirahoric_theory}

\subsubsection{Local theory}

Let $L$ be a non-archimedean local field of characteristic zero, with valuation ring $\O$ and prime ideal $\wp$. Let $q = \#(\O/\wp)$ and let $\varpi$ be an uniformizer of $L$. For an integer $c \geq 0$, let $K_1(\wp^c)$ be the mirahoric subgroup of level $\wp^c$, which is the open compact subgroup of $\GLn(L)$ formed by matrices in $\GLn(\O)$ whose last row is congruent to:
$$
e_n:= (0, \dots,0,1)
$$
modulo $\wp^c$. Let $\pi$ be an irreducible admissible representation of $\GLn(L)$ which is generic. Let $\psi$ be a non-trivial unramified additive character of $L$, i.e. such that $\psi(\O_L) = 1$ and $\psi(\varpi^{-1}) \neq 1$. We denote by $\W(\pi,\psi)$ the Whittaker model of $\pi$ with respect to $\psi$, and by $\phi \mapsto W_\phi$ the isomorphism $\pi \toeq \W(\pi,\psi)$. For any non-negative integer $c$, we denote by $V(c)$ the space of $K_1(\wp^c)$-fixed vector in $\pi$. The following theorem is due to Jacquet, Piatetski-Shapiro, and Shalika (see~\cite[Section 5]{J-PS-S81}):

\begin{theorem}
\label{essential_vector}
Let $\pi$ be an irreducible admissible representation of $\GLn(L)$ which is generic. Then, there exists a non-negative integer $c$ such that $V(c) \neq 0$. Moreover, if $c(\pi) \geq 0$ is the minimal integer with this property, then: 
$$
\dim V(c(\pi)) = 1
$$
and $c(\pi)$ coincides with the analytic conductor of $\pi$, i.e the power of $q^{-s}$ in the $\e$-factor of $\pi$ with respect to an unramified additive character $\psi$ of $L$.
\end{theorem}

The integer $c(\pi)$ of the above theorem is called the \textit{mirahoric conductor} of $\pi$. We will say that $\wp^c$ is the \textit{mirahoric level} of $\pi$. A non-zero form $\phi$ in $V(c(\pi))$ will sometimes be called a newvector or a newform. The values of the Whittaker function $W_\phi$ associated with a newform $\phi$ on (a part of) the diagonal torus have been explicitly computed independently by Miyauchi~\cite[Theorem 4.1]{Miyauchi2012} and Matringe~\cite[Formula (1)]{Matringe13}. The proof of Miyauchi is a generalisation of Shintani's method for unramified representation and assume \ref{essential_vector}, while the proof of Matringe deduce these formulas from a new constructive proof of the results of Jacquet, Piatetski-Shapiro, and Shalika.

In order to state their formula, we need to introduce a few notation. First, for any $f = (f_1,\dots,f_{n-1}) \in \Z^{n-1}$, we note $\varpi^f = \mathrm{diag}(\varpi^{f_1},\dots, \varpi^{f_{n-1}}) \in \mathrm{GL}_{n-1}(L)$. A tuple $f \in \Z^{n-1}$ can also be seen as the element $(f_1,\dots,f_{n-1},0)$ in $\Z^n$. When $f$ is dominant as a weight for $\GLn$, i.e if $f_1 \geq \dots \geq f_{n-1} \geq 0$, we denote by $s_{f}(X_1,\dots,X_n)$ the Schur polynomial associated to $f$. Next, we recall that the standard $L$-function of $\pi$ is of degree $r \leq n$ with $r<n$ if and only if $\pi$ is ramified, i.e. if $c(\pi) >0$ (see \cite[Section 3]{Jacquet79}). Thus, we can write it as:
$$
L(\pi,s) = \prod_{i=1}^r (1-\a_i q^{-s})^{-1}
$$
with $\a_i \in \C^\x$. These notations having been presented, we have the following theorem:
\begin{theorem}
\label{explicit_essential_values}
Let $\pi$ be an irreducible admissible generic representation of $\GLn(L)$, and $\phi$ be a newvector in $\Pi$, and $W_\phi \in \W(\pi,\psi)$ its Whittaker function. Then, for all $f \in \Z^{n-1}$, we have:
$$
W_\phi \left(\begin{array}{ll}
\varpi^f & \\
& 1
\end{array}\right) = \left\{
    \begin{array}{ll}
        \d^{1/2}_{B_n}(\varpi^f)s_f(\a)W_\phi(1) & \mathrm{if} \,\,  f_1 \geq \dots \geq f_{n-1} \geq 0 \\
        0 & \mathrm{otherwise.}
    \end{array}
\right.
$$
where $s_f(\a) = s_f(\a_1, \dots,\a_r,0,\dots,0)$ and $\d_{B_n}$ is the modulus character of ${B_n}(L)$ whose value on $\varpi^f$ is given by $\d_B(\varpi^f) = q^{-\sum_{i=1}^{n-1}(n+1-2i)f_i}$.
\end{theorem}

The above result implies that $W_\phi(I_n) \neq 0$ for any non-zero newform $\phi$ in $\Pi$ (see \cite[Corollary 4.4]{Miyauchi2012}). The only newform $\phi$ such that  $W_\phi(I_n) = 1$ will be called the \textit{essential vector} of $\pi$ with respect to $\psi$, and is denoted $\phi_\pi^\circ$. Its Whittaker function is denoted $W_\pi^\circ$. When $c(\pi) = 0$, i.e. when $\pi$ is unramified, $\phi_\pi^\circ$ is rather called the \textit{spherical vector}. \\

We now study the dependance on the choice of the unramified additive character $\psi$. We recall that all additive characters of $L$ are of the form $\psi^a : x \mapsto \psi(ax)$, for $a \in L^\x$. Moreover, the map $t_a :W \mapsto W^a$, where
$$
W^a : g \mapsto W(\mathrm{diag}(a^{n-1},a^{n-2}, \dots,a,1)g),
$$
gives a $\GLn(L)$-equivariant isomorphism $\W(\pi,\psi) \toeq \W(\pi,\psi^a)$. We also denote by $t_a : \phi \in \pi \mapsto \phi^a \in \pi$ the $\GLn(L)$-equivariant automorphism of $\pi$ defined so that the following diagram commutes:
$$
\begin{tikzcd}
 \pi \arrow[r] \arrow[d, "t_a"] & \W(\pi,\psi)  \arrow[d, "t_a"] \\
 \pi \arrow[r] &  \W(\pi,\psi^a) \\
\end{tikzcd}
$$

Let $\ph$ be an other non-trivial unramified character of $L$. We can thus write it as $\ph = \psi^a$, with $a \in \O^\x$. Let $W_\pi^\circ \in \W(\pi,\psi)$ be the essential vector of $\pi$ with respect to $\psi$. Since $\mathrm{diag}(a^{n-1},a^{n-2}, \dots,a,1) \in K_1(\p^{c(\pi)})$, we have:
$$
(W_\pi^\circ)^a(1) = W_\pi^\circ(1) = 1
$$
Consequently, $(W_\pi^\circ)^a \in \W(\pi,\ph)$ is the essential vector of $\pi$ with respect to $\ph$. \\

There is a similar newvector theory for the transposed mirahoric subgroups ${}^t K_1(\wp^c)$. In fact, if $\varpi_{n-1}=\mathrm{diag}(\varpi, \dots, \varpi, 1) \in \GLn(L)$, one checks that for all $c\geq 0$ one has $\varpi_{n-1}^{-c}K_1(\wp^c) \varpi_{n-1}^c = {}^t K_1(\wp^c)$. Consequently, the action of $\varpi_{n-1}^c$ through $\pi$ induces an isomorphism between the space $V(c)$ of $K_1(\wp^{c})$-fixed vectors and the space ${}^tV(c)$ of ${}^t K_1(\wp^{c})$-fixed vectors in $\pi$. In particular, the space $^tV(c(\pi))$ is of dimension 1. We would like to choose a normalized vector in this $1$-dimensional vector space, as we did with the essential vector for the mirahoric subgroup. From what we have just explained, one knows that the vector corresponding to the Whittaker function $W=\pi(\varpi_{n-1}^c)(W_{\pi}^\circ)$ belongs to $^tV(c(\pi))$. From \ref{explicit_essential_values}, we see that $W(I_n)$ may be equal to zero if $s_{(c,\dots,c)}(\a) = 0$. Hence in general, we can't normalize the choice of a vector $\phi \in {}^tV(c(\pi))$ by the value $W_\phi(I_n)$. Instead, we will normalize the choice of $\phi \in {}^tV(c(\pi))$ by the value $W_\phi(w_n)$, where $w_n := \mathrm{antidiag}(1,\dots,1) \in \mathrm{GL}_n(L)$. We thus define ${}^\vee\phi_\pi^\circ$ to be the only vector $\phi \in {}^tV(c(\pi))$ such that $W_\phi \in \W(\pi,\psi)$ satisfies $W_\phi(w_n) = 1$. Let's check that such a vector exists. To do this, let us consider the dual representation $\pi^\vee$ of $\pi$. The Whittaker model $\mathcal{W}(\pi^\vee,\psi^{-1})$ of $\pi^\vee$ is the space of functions of the form:
$$
W^\vee(g) := W(w_n({}^tg^{-1}))
$$
for $W \in \mathcal{W}(\pi,\psi)$. Thus one sees that ${}^\vee\phi_\pi^\circ$ is the form $\phi \in \pi$ whose Whittaker function $W_\phi \in \W(\pi,\psi)$ is given by:
$$
W_\phi := (W_{\pi^\vee}^\circ)^\vee
$$
We will call ${}^\vee\phi_\pi^\circ$ it the \textit{transposed essential vector} of $\pi$ with respect to $\psi$ (note that ${}^\vee\phi_\pi^\circ$ is called the \textit{first} transposed essential vector of $\pi$ in \cite{thesis}). If $\psi$ and $\psi^a$ (for $\a \in \O^\x$) are two non-trivial unramified additive characters of $L$, then one checks that:
$$
(t_aW_{\pi^\vee}^\circ)^\vee = \w_\pi(a)^{n-1} \cdot t_a((W_{\pi^\vee}^\circ)^\vee)
$$
where $\w_\pi$ is the central character of $\pi$. Thus, if ${}^\vee\phi_\pi^\circ$ is the transposed essential vector with respect to $\psi$, then $t_a({}^\vee\phi_\pi^\circ)$ differs form the transposed essential vector with respect to $\psi^a$ by a factor $\w_\pi(a)^{n-1}$. When $\pi$ is self-dual, this factor is a sign, and is trivial if moreover $n$ is odd.   \\

\subsubsection{Global theory}
\label{global_mirahoric_theory}

Let $E$ be a number fields and let $\psi_f := \psi_{E,f}$ be the finite part of the additive character of $\A_E$ defined in \S\ref{notations}. Let $\mathfrak{n} = \prod_{\wp \mid \mathfrak{n}} \wp^{c_\wp}$ be an ideal of $\O_E$. We define the mirahoric subgroup $K_1(\mathfrak{n})$ of level $\mathfrak{n}$ as the following open compact subgroup of $\GLn(\A_{E,f})$:
$$
K_1(\mathfrak{n}) := \prod_{\wp \nmid \mathfrak{n}} \mathrm{GL}_{n}(\O_\wp) \x \prod_{\wp \mid \mathfrak{n}} K_1(\wp^{c_\wp})
$$

Let $\Pi$ be a cuspidal automorphic representation of $\GLn(\A_E)$ and let $\Pi_f$ denote its finite part. Since $\Pi$ is cuspidal, $\Pi$ and (hence) $\Pi_f$ are generic. We denote by $\W(\Pi_f,\psi_f)$ the Whittaker model of $\Pi_f$ with respect to $\psi_f$ of $\A_{E,f}$, and by $\phi_f \mapsto W_{\phi_f}$ the isomorphism $\Pi_f \toeq \W(\Pi_f,\psi_f)$. We can decompose $\Pi_f$ as a tensor product of local representations $\Pi_f = \otimes_{w} \Pi_w$ over the set ${S_f(E)}$ of finite places of $E$ (see \cite[Theorem 3]{Flath}). For any ${w \in S_f(E)}$, the representation $\Pi_w$ is an irreducible admissible generic representation of $\GLn(E_w)$. Let $c_w = c(\Pi_w)$ be its mirahoric conductor. Consider $\mathfrak{n}(\Pi_f) = \prod_{w \in S_f(E)} \wp_w^{c_w}$, where $\wp_w$ is the prime ideal of $\O_E$ corresponding to $w$. Since all but finitely many $\Pi_w$ are unramified, $\mathfrak{n}(\Pi)$ is a well defined ideal of $\O_E$, called the \textit{mirahoric level} of $\Pi$. From the local theory, we know that the space of $K_1(\mathfrak{n}(\Pi))$-fixed vectors in $\Pi_f$ is one dimensional. \\

We now assume that $E$ is a quadratic number field, and $\s$ denotes the non-trivial Galois involution of $E$. Let $\Pi$ be a cuspidal automorphic representation of $\GLn(\A_E)$ and let $\n := \mathfrak{n}(\Pi)$ be its mirahoric level. Assume that $\mathfrak{n}$ is located above primes of $\Q$ that are split in $E$ and that $\s(\mathfrak{n}) = \mathfrak{n}$. Fix a subset $R$ of the set of prime ideals dividing $\mathfrak{n}$ containing exactly one representative in each orbit for the action of $\s$. Thus, the prime ideal decomposition of $\mathfrak{n}$ can be written:
$$
\mathfrak{n} = \prod_{\wp \in R} \wp^{c_\wp} \s(\wp)^{c_\wp}.
$$ 
For such an ideal $\mathfrak{n}$ and such a choice of a subset $R$ of representatives, we can define the \textit{mixed mirahoric subgroup} $K_1^*(\mathfrak{n})$ of level $\mathfrak{n}$ and type $R$ as the open compact subgroup of $\GLn(\A_{E,f})$ defined by:
$$
K_1^*(\mathfrak{n}) = \prod_{\wp \nmid \mathfrak{n}} \mathrm{GL}_{n}(\O_\wp)  \x \prod_{\wp \in R}  K_1(\wp^{c_\wp}) \x {}^t K_1(\s(\wp)^{c_\wp})
$$

It follows from the local theory that the space of $K_1^*(\mathfrak{n})$-fixed vectors in $\Pi_f$ is one dimensional. We then consider the form $\phi^{*}_\Pi$ in $\Pi_f$ defined as the tensor product $\phi^{*}_\Pi := \otimes_{w} \phi_w$ with:

\begin{itemize}
\item if $w \nmid \n$, then $\phi_w$ is defined as for $\phi^{\circ}_\Pi$, depending whether $w$ divides $\mathfrak{d}$ or not;
\item if $w \mid \n$ and $w \in R$, then $\phi_w := \phi_{\Pi_w}^\circ$ is the essential vector of $\Pi_w$ with respect to $\psi_w$;
\item if $w \mid \n$ and $w \notin R$, then $\phi_w := {}^\vee\phi_{\Pi_{w}}^\circ$ is the transposed essential vector of $\Pi_{w}$ with respect to $\psi_w$.
\end{itemize}
Then $\phi^{*}_\Pi$ is a $K_1^*(\mathfrak{n})$-fixed vector in $\Pi_f$ which we call the \textit{mixed essential vector} of $\Pi$ (note that $\phi^{*}_\Pi$ is called the \textit{first} mixed essential vector of $\Pi$ in \cite{thesis}).  \\

\subsection{Cohomology groups and Hecke algebras}

\subsubsection{Irreducible algebraic representations of $\SO$}
\label{alg_irrep_SO3}
The irreducible algebraic representations of $\mathrm{SO}(3)$ are indexed by $\ell \in \Z_{\geq 0}$. For such a $\ell$, we denote by $(\tau_\ell,V_\ell)$ the corresponding irreducible representation of $\mathrm{SO}(3)$. It is a representation of dimension $2\ell + 1$. More precisely, following \cite[2.4.1]{Che22}, we fix a model as follows. $V_\ell$ is the quotient of the space of homogeneous polynomials over $\C$ of degree $\ell$ in variable $X_1,X_2,X_3$, by the subspace generated by $X_1^2+X_2^2+X_3^2$. The action of $\SO$ on this space is given by:
$\tau_\ell(g) \cdot P(\overline{X_1},\overline{X_2},\overline{X_3}) = P((\overline{X_1},\overline{X_2},\overline{X_3})g)$, for $g \in \SO$ and $P \in V_\ell$. A basis of $V_{\ell}$ is given by $\{ \mathbf{v}_i, \, -\ell \leq i \leq \ell \}$, where:
$$
\mathbf{v}_i = (\mathrm{sgn}(i)\overline{X_1} + \sqrt{-1}\cdot\overline{X_2})^{|i|}\overline{X_3}^{\ell-|i|}
$$
For each triplet $(j_1,j_2,j_3) \in \Z_{\geq0}^3$ such that $j_1+j_2+j_3 = \ell$, we also define:
$\mathbf{v}_{\left(\ell ;\left(j_{1}, j_{2}, j_{3}\right)\right)} = \overline{X_1}^{j_1}\overline{X_2}^{j_2}\overline{X_3}^{j_3} \in V_{\ell}
$. \\

\subsubsection{Irreducible algebraic representations of $\GL(E)$}
\label{alg_irrep}
\label{coefficients}

 Let $\K$ be a field and let $X^+(T_3)$ be the set of triplets $X^+(T_3)= \{ \mu = (m^+,m^-,v) \in \Z^3,\quad m^+\geq 0, \,m^-\geq 0\}$. For $\mu~=~(m^+,m^-,v) \in X^+(T_3)$, let $\mathcal{P}_{\mu}(\K) = \K[X,Y,Z ; A,B,C]_{m^+,m^-}$ be the space of 6-variables polynomials homogeneous of degree $m^+$ in $X,Y,Z$ and $m^-$ in $A,B,C$ with coefficients in $\K$, on which $g \in \GL(\K)$ acts by:
$$
\rho_{\mu}(g)(P(X,Y,Z ; A,B,C)) = (\det g)^{v} P((X,Y,Z)g ; (A,B,C){}^\top g^{-1}).
$$ 
We now consider the following differential operator:
$$
i_{m^+,m^-} = \frac{\partial^2}{\partial X \partial A} + \frac{\partial^2}{\partial Y \partial B} + \frac{\partial^2}{\partial X \partial A}:  \mathcal{P}_{\mu}(\K) \to \mathcal{P}_{\mu-(1,1,0)}(\K)
$$
It is equivariant with respect to $\rho_{\mu}$ and $\rho_{\mu-(1,1,0)}$. Let $L_{\mu}(\K)$ be the kernel of $i_{m^+,m^-}$. We thus obtain a representation $(\rho_{\mu},L_{\mu}(\K))$ of $\GL(\K)$. All the irreducible algebraic representations of $\GL(\K)$ are of the form $(\rho_{\mu},L_{\mu})$ where $\mu$ runs over $X^+(T_3)$ (see \cite[Theorem 13.1 and Claim 13.4]{RT:FC}). Note that the highest weight of $(\rho_{\mu},L_{\mu})$ is $\mu = \left( v+m^+, v, v-m^-\right)$, and that a vector of highest weight is given by $P_{\mu}^+ = X^{m^+}C^{m^-}$. Since $X^+(T_3)$ is in bijection with the set of dominant weights for $\GL$, by abuse we will sometimes refer to $\mu$ as the highest weight of $(\rho_{\mu},L_{\mu})$. \\

The dual weight of $\mu \in X^+(T_3)$ is defined to be $\mu^\vee := (m^-,m^+,-v) \in X^+(T_3)$. We then consider the pairing $\langle \cdot, \cdot \rangle_{\mu}: L_{\mu}(\K) \x L_{\mu^\vee}(\K) \to \K$ by:

\begin{equation}
\label{pairing_coefficients}
\langle \sum_{\underline{i}^+,\underline{i}^-} a_{\underline{i}^+,\underline{i}^-} X^{\underline{i}^+} Y^{\underline{i}^-} , \sum_{\underline{j}^+,\underline{j}^-} b_{\underline{j}^+,\underline{j}^-} X^{\underline{j}^+} Y^{\underline{j}^-}\rangle_{\mu} = \sum_{\underline{i}^+,\underline{i}^-} \binom{m^+}{i_1^+,i_2^+,i_3^+}^{-1} \binom{m^-}{i_1^-,i_2^-,i_3^-}^{-1}a_{\underline{i}^+,\underline{i}^-}  b_{\underline{i}^-,\underline{i}^+} 
\end{equation}
where $X^{\underline{i}^+}:= X_1^{i_1^+}X_2^{i_2^+}X_3^{i_3^+}$ and $Y^{\underline{i}^-}:=Y_1^{i_1^-}Y_2^{i_2^-}Y_3^{i_3^-}$. This paring is perfect and equivariant for the action of $\GL(\K)$:
$$
\langle \rho_{\mu}(g)(P),\rho_{\mu^\vee}(g)(Q)\rangle_{\mu} =\langle P,Q\rangle_{\mu}
$$
for $P,Q \in L_{\mu}(\K)$ and $g \in \GL(\K)$. \\

Now, let $E$ be a totally real field and $G_E = \mbox{Res}_{E/\Q} \mathrm{GL}_{3/E}$. Let $S_\inf = S_\inf(E)$ denote the set of archimedean places of $E$, i.e. the set of embeddings of $E$ into $\overline{\Q}$. Let $X^+(T_E) = X^+(T_3)^{ S_\inf}$, i.e. the set of tuples $\mu = (\mu_\t)_{\t \in S_\inf}$, with $\mu_\t =(m^+_\t,m^-_\t,v_\t) \in X^+(T_3)$ for all $\t \in S_\inf$. Let $E^g$ be the Galois closure of $E$ in $\overline{\Q}$ and suppose that $\K$ is a field extension of $E^g$. 

We define $L_{\mu}(\K):= \bigotimes_{\tau \in  S_\inf} L_{\mu_\tau}(\K)$. For each $\tau: E \to \overline{\Q}$, we can define a map $\tau: E \otimes_\Q \K \to \K$ by $ x \otimes k \mapsto \tau(x)k$ which induces a morphism $\tau: G_E(\K) \to \GL(\K)$. Then, we get an action of $G_E(\K)$ on $L_{\mu}(\K)$, given for $g \in G_E(\K)$ and for a pure tensor $P= \otimes_\tau P_\t \in L_{\mu}(\K)$, by:
$$
\rho_{\mu}(g)(P) = \bigotimes_{\t \in  S_\inf}  \rho_{\mu_\t}(\t(g))(P_\t).
$$
Thus, we get a representation $(\rho_{\mu},L_{\mu})$ of $\GL(\K)$. All the irreducible algebraic representations of $\GL(\K)$ are of this form, for $\mu \in X^+(T_E)$. By abuse, $\mu$ will sometimes be refered to as the highest weight of $L_{\mu}$. Moreover we have a perfect $\GL(\K)$-equivariant pairing:
$$
\langle \cdot, \cdot \rangle_{\mu}: L_{\mu}(\K) \x L_{\mu^\vee}(\K) \to \K
$$
obtained as a tensor product over $ S_\inf$ from the pairing (\ref{pairing_coefficients}). \\

\subsubsection{Integral structure and $p$-smallness.} We now suppose that $\K$ is some finite extension of $\Q_p$ and that $\O$ is its valuation ring. All the above definitions make sense over $\O$, and we get a representation $\rho_{\mu}$ of the group scheme $G_E = \mathrm{Res}_{\O/\Z_p} \mathrm{GL}_{3/\O}$ on the module $L_{\mu}(\O)$ consisting of polynomials in $L_{\mu}(\K)$ with coefficients in $\O$. We see from the expression of $\langle \cdot, \cdot \rangle_{\mu}$ that the $\O$-dual of $L_{\mu}(\O)$ in $L_{\mu^\vee}(\K)$ may not be $L_{\mu^\vee}(\O)$. To avoid this problem, we will say that $\mu = (\mu_\tau)_{\t \in S_\inf}$ is $p$\textit{-small} if for each $\tau \in  S_\inf$: 
$$
p > \mathrm{max}(m^+_\tau, m^-_\t)
$$
where $\mu_\tau = (m^+_\t,m^-_\t,v_\t)$. For a given weight $\mu$, this condition excludes a finite number of small primes. Then, as long as $\mu$ is $p$-small $\langle \cdot, \cdot \rangle_{\mu}$ restricts to a pairing defined on $\O$:
$$
\langle \cdot, \cdot \rangle_{\mu}: L_{\mu}(\O) \x L_{\mu^\vee}(\O) \to \O.
$$
It follows from (\ref{pairing_coefficients}), and from the perfectness of the pairing on $\K$, that this pairing is perfect on $\O$. See Section 1 of Polo-Tilouine in \cite{CSV} for a more theoretical explanation of these facts (be careful though that $p$-small there means $p \geq m^+_\tau + m^-_\t + 2$, for $\t \in S_\inf$).

Let $\K$ be a $p$-adic field $\K$, and $\O$ be its valuation ring. In this subsection, we consider the algebraic group $G_E = \mathrm{Res}_{E/\Q}(\mathrm{GL}_{n/E})$ for any integer $n \geq 1$ and any number field $E$. When $n=3$ and $E$ is totally real, which is our main case of concern in this thesis,  we have introduced in the precedent subsection a precise model $L_{\mu}$ for the algebraic representation of $G_E$ of highest weight $\mu$ (as explained $\mu \in X^+(T_E)$ is not strictly speaking a dominant weight but some convenient avatar of it, but since $X^+(T_E)$ is in bijection with the set of dominant weights for $G_E$, we call it the highest weight of $L_{\mu}$ by abuse). For a general $n \neq 3$ and any number field $E$, one can similarly define, for each dominant weight $\mu$ of $G_E$, a representation $L_{\mu}(\O)$ of $G_E(\O)$ which extends over $\K$ to $L_{\mu}(\K)$, the irreducible algebraic representation of $G_E(\K)$ of highest weight $\mu$ (see for example \cite{Hi98}). \\

\subsubsection{Adelic variety and sheaves}
\label{variety_sheaves}

Recall that $G_E = \mathrm{Res}_{E/\Q}(\mathrm{GL}_{3/E})$. For any open compact subgroup $K_f$ of $G_E(\A_{f})$, the adelic variety of level $K_f$ for $G_E$ is defined as:
$$
Y(K_f):= G_E(\Q) \bs G_E(\A) / K_f K_\inf
$$
where $K_\inf = \SO \R_+^\x$. Let $\mu \in X^+(T_E)$ be a dominant weight and $A$ denote $\O$, $\K$ or $\C$. Let $\L_\mu(A)$ be the locally constant sheaf of $A$-modules on $Y(K_f)$ attached to the representation $L_\mu(A)$. We have an inclusion of sheaves $\L_\mu(\O) \subset \L_\mu(\K) \subset \L_\mu(\C)$. \\

\subsubsection{Cohomology groups}
\label{cohomology_groups}

Let $A$ denote the fields $\K$ or $\C$. We consider the Betti cohomology groups $H^q(Y(K_f), \L_{\mu}(A))$ and the Betti cohomology with compact support $H^5_c(Y(K_f), \L_{\mu}(A))$. There is a natural map:
\begin{equation}
\label{std_to_compact}
i_A: H^\bullet_c(Y(K_f), \L_{\mu}(A)) \to H^\bullet(Y(K_f), \L_{\mu}(A))
\end{equation}
The interior cohomology $H^5_!(Y(K_f), \L_{\mu}(A))$ is defined to be the image of $i_A$. If $A$ is the ring $\O$, and if $?$ denotes $\emptyset, c \mbox{ or }!$ then $H_?^\bullet(Y(K_f), \L_{\mu}(\O))$ denotes the $\O$-torsion free part of the corresponding cohomology groups with coefficients in $\O$.

The cuspidal cohomology $H^\bullet_{cusp}(Y(K_f), \L_{\mu}(\C))$ is defined to be the following $(\g,K_\inf)$-cohomology group (or relative Lie algebra cohomology, see \cite[Chapter I]{BW00}):
$$
H^\bullet_{cusp}(Y(K_f), \L_{\mu}(\C)) = H^\bullet(\g,K_\inf ;\mathcal{A}_{cusp}(G_E(\Q)\bs G_E(\A)/K_f)\otimes L_{\mu}(\C))
$$
where $\mathcal{A}_{cusp}(G_E(\Q)\bs G_E(\A)/K_f, \w)$ is the space of $K_f$-fixed cusp forms on $G_E(\A)$. A priori, the cuspidal cohomology is just contained in $H^\bullet(Y(K_f), \L_{\mu}(\C))$. In fact, there is an injection:
$$
H^\bullet_{cusp}(Y(K_f), \L_{\mu}(\C)) \inj H^\bullet_!(Y(K_f),\L_{\mu}(\C))
$$
More precisely, there exists a canonical map (see for exemple \cite[2.1]{Hi99}):
$$
s: H^q_{cusp}(Y(K_f), \L_{\mu}(\C)) \inj H^q_c(Y(K_f),\L_{\mu}(\C))
$$
which a section of $i_\C: H^q_c(Y(K_f),\L_{\mu}(\C)) \to H^q(Y(K_f),\L_{\mu}(\C))$. Hence, $H^q_{cusp}(Y(K_f), \L_{\mu}(\C))$ can be viewed as a subgroup of $H^q_{c}(Y(K_f), \L_{\mu}(\C))$. 
We define the cuspidal cohomology groups with coefficients in $\K$ to be:
$$
H^\bullet_{cusp}(Y(K_f), \L_{\mu}(\K)) = H^\bullet_{cusp}(Y(K_f), \L_{\mu}(\C)) \cap H^\bullet_c(Y(K_f), \L_{\mu}(\K))
$$
It follows from \cite[Théorème 3.19]{Clozel90} that:
$$
H^\bullet_{cusp}(Y(K_f), \L_{\mu}(\K)) \otimes_\K \C = H^\bullet_{cusp}(Y(K_f), \L_{\mu}(\C))
$$
We then define the cuspidal cohomology groups with coefficients in $\O$ to be:
$$
H^\bullet_{cusp}(Y(K_f), \L_{\mu}(\O)) = H^\bullet_{cusp}(Y(K_f), \L_{\mu}(\C)) \cap H^\bullet_c(Y(K_f), \L_{\mu}(\O))
$$
In particular $H^\bullet_{cusp}(Y(K_f), \L_{\mu}(\O))$ is torsion free. \\

\subsubsection{Hecke algebras}
\label{hecke_corr}

Let $A$ denote $\O$, $\K$ or $\C$. Let $K_f = \prod_w K_w$ be some open compact subgroup of $G_E(\A_f)$ and let:
$$
S_{K_f}~=~\{ w \mbox{ finite place of }E\mbox{, s.t. } K_{w} \neq \GLn(\O_w)\}.
$$
This is a finite set. Let $S_p = \{ w \mbox{ finite place of E, s.t. } w~\mid~p \}$. We suppose that $p$ is outside the level of $K_f$, i.e that $S_p \cap S_{K_f} =\varnothing$, and we put $S :=  S_{K_f} \sqcup S_p$. 

Let $w \notin S$, and let $\varpi_w$ be an uniformizer of $E_w$. We define the Hecke operators $T_{w,i}$ at $w$ (for $i = 1,2,3$) as the characteristic functions of the double coset  $\GL(\O_w)\mathrm{diag}(\varpi_w I_i, I_{n-i})\GL(\O_w) \subset \GL(E_w)$. We will write $S_w$ for $T_{w,3}$. The spherical abstract Hecke algebra outside of $S$ with coefficients in $A$ is defined as the following tensor product:
$$
\H(K_f;A) = \bigotimes_{w \notin S} A[T_{w,1},T_{w,2},{S_{w}}^\pm]
$$

There is a well known action of $\H(K_f;A)$ (see for instance \cite[\S3.2.3]{thesis}) on the various cohomology groups $H^\bullet_?(Y(K_f), \L_\mu(A))$ (for $? = \emptyset, c, !$ or $cusp$) described in \S\ref{cohomology_groups}. Moreover, all the comparison maps described in \S\ref{cohomology_groups} are equivariant for this action. Finally, since we don't consider Hecke operators at $p$, the maps $H^\bullet_?(Y(K_f), \L_{\mu}(\O)) \to H^\bullet_?(Y(K_f), \L_{\mu}(\K))$ are $\H(K_f;\O)$-equivariant. We denote by $h(K_f;\O)$ the Hecke algebra acting faithfully on the cohomology:
$$
h(K_f,\O) := \mathrm{Im}(\H(K_f,\O) \to \End_\O H^\bullet(Y(K_f),M_\mu(\O)))
$$
It is a finite commutative $\O$-algebra. Thus, $h(K_f;\O)$ is semi-local and we have a decomposition $h(K_f;\O) \simeq \prod_{\m} h(K_f;\O)_\m$, where $\m$ ranges through maximal ideals of $h(K_f;\O)$ (there are finitely many of them), and $h(K_f;\O)_\m$ be the completion of $h(K_f;\O)$ at $\m$. In this paper, whenever such a maximal ideal $\m \subset h(K_f;\O)$ is fixed, we denote by ${\TT} = h(K_f;\O)_\m/(\O-{tors})$ the torsion-free quotient of $h(K_f;\O)_\m$. \\

\subsection{Cohomological cuspidal automorphic representations}

From now, we assume that $E$ is a totally real number field.

\subsubsection{Cohomological automorphic representations} Let $K_3 := \R_{>0} \cdot \SO \subset \GL(\R)$, and $K_\inf= \prod_{v\mid \inf} K_3 \subset \GL(E_\inf)$. Let $\g_{3}$ be the Lie algebra of $\GL(\R)$, and let $\g = \oplus_{v\mid \inf}  \g_{3}$ be the Lie algebra of $\GL(F_\inf)$. We say that a cuspidal automorphic representation $\Pi$ of $G_E$ is cohomological of cohomological weight $\mu \in X^+(T_E)$ if the following relative Lie algebra cohomology group (see \cite[Chapter I]{BW00}):
$$
H^q(\g,K_\inf ; \Pi_\inf \otimes M_\mu(\C))
$$
is non trivial for some $q \geq 0$. In that case, it is non-zero if and only if $b\leq q \leq t$, where $b=2t$ and $t= 3d$ are called the bottom and top degree of the cuspidal range (see \cite[Lemme 3.14]{Clozel90}). We denote by $\mathrm{Coh}(G_E,\mu)$ the set of such representations. For each open compact subgroup $K_f$ of $G_E(\A_f)$, we also define $\mathrm{Coh}(G_E,\mu,K_f)$ to be the subset of $\mathrm{Coh}(G_E,\mu)$ consisting of cohomological representations $\Pi$ such that $\Pi_f$ has non-zero $K_f$-fixed vectors. By definition, the property of being cohomological for a representation $\Pi$ only depends on its archimedean part $\Pi_\infty$. The cohomological weight of a cohomological representation $\Pi$ is unique \cite[Section 3.5]{Clozel90} and can be expressed from the Langlands parameter of $\Pi_\infty$ (see for example \cite[Proposition 4.1]{HN20}). \\

The cohomological weight $\mu$ of a cohomological representation $\Pi$ is necessarily \textit{pure} \cite[Lemme de pureté 4.9]{Clozel90}, in the sens of the following definition:
\begin{definition}[Pure weights]
Let $\mu \in X^+(T_E)$ be a highest weight for $G_E$. Write $\mu= (\mu_\tau)_{\tau \in  S_\inf}$, with $\mu_\tau = (m^+_\tau,m^-_\tau,v_\tau) \in X^+(T_3)$ for each $\tau \in S_\inf$. We say that $\mu$ is \textit{pure} if:
\begin{itemize}
\item $v_\tau$ does not depend on $\tau \in  S_\inf$ ;
\item $m^+_\tau=m^-_\tau$ for each $\tau \in  S_\inf$.
\end{itemize}
\end{definition}

Let $\mu \in X^+(T_E)$ and let $\Pi \in \mathrm{Coh}(G_E,\mu)$ be a cohomological cuspidal automorphic representation. Write $\Pi_\inf = \bigotimes_{\tau \in  S_\inf} \Pi_\tau$ for the archimedean part of $\Pi$. For each archimedean place $\tau \in  S_\inf$, on has that:
$$
\Pi_\tau = \mathrm{Ind}_{P_{2,1}(\R)}^{\GL(\R)} (D_{\ell_\tau} \otimes \e_\tau) \otimes | \cdot |^{v_\tau}
$$
for some quadratic character $\e_\tau$ of $\R^\x$. Here $\ell_\tau= 2m_\tau + 3$ is the minimal $\mathrm{SO}(3)$-type of $\Pi_\tau$, and $D_{\ell_\tau}$ is the discrete series representation of $\mathrm{GL}_2(\R)$ of weight $\ell_\tau$ (see \cite[Theorem 2.1(3)]{Che22} or \cite[\S 4.1]{HN20}). Note that the central character $\w_{\Pi_\tau}$ of $\Pi_\tau$ satisfies:
$$
\w_{\Pi_\tau}|_{\{\pm1\}} = \e_\tau \cdot \mathrm{sgn}
$$
where $\mathrm{sgn} : \{ \pm1\} \to \{ \pm1\}$ is the non-trivial character. Moreover, for each archimedean place $\tau \in  S_\inf$, the minimal $\mathrm{SO}(3)$-type of $\Pi_\tau$ is $\ell_\tau = 2m_\tau + 3$ (see \S\ref{alg_irrep_SO3}). \\

\subsubsection{Hecke eigensystems}
 Let $\Pi = \otimes_w \Pi_w$ be a cuspidal automorphic representation of $\GL(\A_E)$ of level $K_f$ (i.e. with non-trivial $K_f$-fixed vectors), and let $\H(K_f;\C)$ be the abstract Hecke algebra defined in \S\ref{hecke_corr}. Then, the Hecke eigensystem $\l_\Pi$ associated with $\Pi$ is the morphism of $\C$-algebras $\l_\Pi: \H(K_f;\C) \to \C$ such that for all $w \notin S$ and $i\in \{1,2,3 \}$, $\l_\Pi(T_{w,i})$ is defined to be the complex scalar through which $T_{w,i}$ acts on the $1$-dimensional $\C$-vector space of $K_w$-fixed vectors in $\Pi_w$.

Let $h(K_f;\C)$ be the Hecke algebra acting faithfully on the cohomology with coefficients in $M_\mu(\C)$. Then, if $\Pi \in \mathrm{Coh}(G_E,\mu,K_f)$ is cohomological, the Hecke eigensystem $\l_\Pi : \H(K_f;\C) \to \C$ attached to $\Pi$ factorizes into a $\C$-algebra morphism $\l_\Pi : h(K_f;\C) \to \C$. Finally, let $\K$ be a $p$-adic field containing $E$ and the rationality field $\Q(\Pi)$ of $\Pi$ and let $\O$ be its valuation ring. Let $h(K_f;\O)$ be the cohomological Hecke algebra with coefficients in $\O$, defined in \ref{hecke_corr}. Then, the image of $h(K_f;\O)$ through $\l_\Pi$ is included in $\O$, yielding an integral Hecke eigensystem $\l_\Pi : h(K_f;\O) \to \O$. \\

\subsubsection{Galois representations} 
\label{galois_reps}
Let $\Pi \in \mathrm{Coh}(G,\mu,K_f)$ be a cohomological cuspidal automorphic representation of $\GLn(\A_E)$. Let $\K$ be some sufficiently large $p$-adic field and $\O$ be its valuation ring.  Let $S$ be the finite set of places of $E$ containing the places where $\Pi$ is ramified and the places dividing $p$. The Galois representation associated with $\Pi$ according to the Langlands philosophy has been constructed by \cite[Theorem A]{HLTT16} (see also \cite[Corollary V.4.2]{Scholze15}). This is the following theorem:
\begin{theorem}
There exists a unique continuous semisimple Galois representation $\rho_{\Pi} : \Gal(\overline{E}/E) \to \GLn(\K)$ which is unramified outside of $S$, and such that for all $w \notin S$, the characteristic polynomial of $\rho_\Pi(\Frob_{w})$ is:
$$
\sum_{i=0}^n (-1)^i q_w^{i(i-1)/2} \l_\Pi(T_{w,i}) X^{n-i}
$$
where $\l_\Pi : \H(K_f;\C) \to \C$ is the Hecke eigensystem associated with $\Pi$.
\end{theorem}
One can find a $\Gal(\overline{E}/E)$-invariant $\O$-lattice in the representation space of $\rho_{\Pi}$. We then obtain a Galois representation $\rho_{\Pi} : \Gal(\overline{E}/E) \to \GLn(\O)$. If the residual representation $\overline{\rho_{\Pi}}$ is irreducible then there is a unique homothety classe of such an $\O$-lattice, and $\rho_{\Pi} : \Gal(\overline{E}/E) \to \GLn(\O)$ is unique. \\

\subsubsection{Structure of the cuspidal cohomology}

We now give the structure of the cuspidal cohomology. Let $q \in \Z_{\geq 0}$ and $\mu$ be some dominant weight for $G_E$. The cuspidal cohomology admits the following direct sum decomposition:
\begin{equation}
\label{deco_coho_cusp}
H^\bullet_{cusp}(Y(K_f), \L_{\mu}(\C)) = \bigoplus_{\Pi \in \mathrm{Coh}(G_E,\mu,K_f)} H^\bullet(\g, K_\inf  ; \Pi_{\inf} \otimes L_{\mu}(\C)) \otimes \Pi_f^{K_f}.
\end{equation}

In particular, it follows from \cite[Section VII]{BW00} and \cite[Section 4.5]{Clozel90} that the cuspidal cohomology group $H^q_{cusp}(Y(K_f), \L_{\mu}(\C))$ doesn't vanish only when $\mu$ is pure and $q$ belongs to the so-called Borel-Wallach interval $[b_E,t_E]$, for $b_E = 2d$ and $t_E = 3d$. In particular, as already said, the set $\mathrm{Coh}(G_E,\mu)$ is empty unless the weight $\mu$ is pure, and $\mathrm{Coh}(G_E,\mu,K_f)$ is a finite set. This decomposition justifies the nomenclature: a cuspidal automorphic representation is cohomological of cohomological weight $\mu$ if and only if it appears as a subspace of the cuspidal cohomology with coefficients in $\L_{\mu}(\C)$. A cohomological cuspidal automorphic representation $\Pi$ occurs in the cuspidal cohomology of degree $q$ for all degree $q \in [b_E,t_E]$, with multiplicity equal to the dimension of $H^q(\g, K_\inf  ; \Pi_{\inf} \otimes L_{\mu}(\C))$. Finally, recall that the cuspidal cohomology of level $K_f$ is endowed with an action of $\H(K_f,\C)$. Then, the above decomposition is a decomposition of $\H(K_f,\C)$-modules. In particular, $H^\bullet_{cusp}(Y(K_f), \L_{\mu}(\C))$ is a semi-simple $\H(K_f,\C)$-module.

\subsection{Eichler-Shimura maps and middle-degree automorphic periods}

\subsubsection{Generators of the $(\g,K_\inf)$-cohomology}
\label{chen_generators}
In this paragraph, we construct some explicit elements in the $(\g,K_\inf)$-cohomology groups appearing on the right hand side of (\ref{deco_coho_cusp}). Let $\Pi \in \mathrm{Coh}(G,\mu,K_f)$ and let $\Pi_\inf$ be its archimedean part. Let $q \in \N$ be a non-negative integer. From the Kunneth formula, we have the following direct sum decomposition into tensor products of local cohomology groups:
$$
\begin{aligned}
H^q(\g, K_\inf ; \Pi_\inf \otimes L_{\mu}(\C)) & = \bigoplus_{\sum_\tau  q_\tau = q } \left( \bigotimes_{\tau \in S_\inf} H^{q_\tau}(\g_3, K_3 ; \Pi_{\tau} \otimes L(\mu_\tau ;\C)) \right)
\end{aligned}
$$
where the direct sum is indexed by tuples $(q_\tau)_{\tau \in  S_\inf}$ of non-negative integers such that $\sum_{\tau \in  S_\inf}  q_\tau~=~q$. Let $\tau$ be an archimedean place of $E$. From a general result of Clozel (see \cite[Lemme  3.14]{Clozel90}), we know that:
$$
H^{q_\tau}(\g_3, K_3; \Pi_{\tau} \otimes L_{\mu_\tau}(\C))=\left\{
    \begin{array}{ll}
        \C & \mbox{if } q_\tau = 2,3 \\
        0 & \mbox{otherwise}
    \end{array}
\right.
$$
where $\g_3 = \glr$ and $K_3 = \mathrm{SO}(3)\R_+^\x$. In particular, together with the Kunneth formula, this gives that:
\begin{equation}
\label{g-K_dim}
\dim_\C H^q(\g, K_\inf ; \Pi_\inf \otimes L_{\mu}(\C)) = \binom{t-b}{q-b}
\end{equation}
{} \\

For the sake of brevity, the local cohomology groups $H^{q_\tau}(\g_3, K_3; \Pi_{\tau} \otimes L_{\mu_\tau}(\C))$ are denoted by $H^{q_\tau}_\tau$ in the following. Let $J \subset  S_\inf$ be a subset of the archimedean places. We define:
$$
H^q(\g, K_\inf  ; \Pi_\inf \otimes L_{\mu}(\C))_{J}:= \bigotimes_{\tau \in J} H^2_\tau  \otimes \bigotimes_{\tau \notin J} H^3_\tau
$$
It is a $1$-dimensional $\C$-vector space. A generator of this space can be constructed by choosing generators $[\Pi_v]_i$ of $H^i_\tau$ for $i = 2, 3$ at each archimedean place $\tau$. Chen \cite[Lemma 2.2]{Che22} has constructed explicit non-zero elements (thus generators)  $[\Pi_\tau]_i$ in $H^i_\tau$ and proven that they yields the correct archimedean local factors (the so-called $\G$-factors) for the adjoint $L$-function. Using these generators, we then define the following generator of $H^q(\g, K_\inf ; \Pi_{\inf} \otimes L_{\mu}(\C))_J$:
$$
[\Pi_\inf]_J = \bigotimes_{\tau \in J} [\Pi_\tau]_2 \otimes \bigotimes_{v \notin J} [\Pi_\tau]_3
$$
for $J \subset  S_\inf$. In particular, when $E$ is a real quadratic field, we obtain the following two elements:
$$
[\Pi_\inf]_{\{ \tau\}} = [\Pi_\tau]_2 \otimes [\Pi_{\s\tau}]_3, \quad \mbox{ and } \quad [\Pi_\inf]_{\{ \s\tau\}} = [\Pi_\tau]_3 \otimes [\Pi_{\s\tau}]_2,
$$
which are respective generators of $H^5_{\{ \tau\}}$ and $H^5_{\{ \s\tau\}}$. \\

\subsubsection{Eichler-Shimura maps}
\label{eichler-shimura_maps}

For each $J\subset  S_\inf$, we also denote by $H^q_{cusp}(Y_E(K_f),\L_{\mu}(\C))_J$ the subspace of $H^q_{{cusp}}(Y_E(K_f),\L_{\mu}(\C))$ defined as the direct sum:
$$
\bigoplus_{\Pi \in \mathrm{Coh}(G,\mu,K_f)}H^q(\g, K_\inf ; \Pi_\inf \otimes L_{\mu}(\C))_{J} \otimes \Pi_f^{K_f}.
$$

Let $\mu \in X^+(T_E)$ be a pure weight. Let $K_f$ be some open compact subgroup of $G_E(\A_{f})$. The space $S_{\mu}(K_f)$ of automorphic cusp forms of weight $\mu$ and level $K_f$ is defined by:
$$
S_{\mu}(K_f) := \bigoplus_{\Pi} \Pi_f^{K_f},
$$
where $\Pi$ runs through $\mathrm{Coh}(G_E,\mu,K_f)$, i.e. through the cuspidal automorphic representations of $\GL(\A_E)$ of level $K_f$ whose archimedean part is given by $\Pi_\inf = \bigotimes_{\tau \in  S_\inf} \Pi_\tau$, with:
$$
\Pi_\tau = \mathrm{Ind}_{P_{2,1}(\R)}^{\GL(\R)} (D_{\ell_\tau} \otimes \e_\tau) \otimes | \cdot |^{v_\tau}
$$
for some quadratic character $\e_\tau$ of $\R^\x$. Let $J \subset  S_\inf$ be a subset of the archimedean places. Then, the Eichler-Shimura map of type $J$: 
$$
\d_J^q :S_{\mu}(K_f) \inj H^{b + | J |}_{{cusp}}(Y_E(K_f),\L_{\mu}(\C))
$$
is defined by sending an element $\phi_f \in \Pi_f^{K_f}$ to $\phi_f \otimes [\Pi_\inf]_J \in H^{b + | J |}_{{cusp}}(Y_E(K_f),\L_{\mu}(\C))_J$, where $[\Pi_\inf]_J$ is the generator of $H^q(\g, C_\inf ; \Pi_{\inf} \otimes L_{\mu}(\C))_J$ specified in the preceding paragraph. Thus, $\d_J^q$ is an injective $\C$-linear map whose image is $H^q_{{cusp}}(Y_E(K_f),\L_{\mu}(\C))_J$. It is equivariant for the action of $\H(K_f,\C)$ on both side. Moreover, for each $\Pi \in \mathrm{Coh}(G,\mu,K_f)$, it induces an isomorphism:
$$
\d_J^q : \Pi_f^{K_f} \toeq H^q_{{cusp}}(Y_E(K_f),\L_{\mu}(\C))_{J}[\Pi_f].
$$

\mbox{} \\

\subsubsection{Eichler-Shimura maps for real quadratic fields} We here specify the above presentation to the special case where $E$ is a real quadratic field ($d=2$). We recall that in this case, we have fixed an embedding $\tau: E \inj \C$ and that $\s$ is the generator of the Galois group $\Gal(E/\Q)$, so that $ S_\inf$ can be identified with $\{\tau, \s\tau\}$.  In this situation, the cohomology is non-zero in degrees $q=4,5,6$. In middle degree $q=5$, Kunneth's formula is written:
$$
H^5(\g, K_\inf  ; \Pi_\inf \otimes L_{\mu}(\C)) = H^5(\g, K_\inf  ; \Pi_\inf \otimes L_{\mu}(\C))_{\{\tau\}} \oplus H^5(\g, K_\inf  ; \Pi_\inf \otimes L_{\mu}(\C))_{\{\s\tau\}}
$$
where $H^5(\g, K_\inf  ; \Pi_\inf \otimes L_{\mu}(\C))_{\{\tau\}} = H^{2}_{\tau} \otimes H^{3}_{\s\tau}$ and $H^5(\g, K_\inf  ; \Pi_\inf \otimes L_{\mu}(\C))_{\{\s\tau\}} = H^{3}_{\tau} \otimes H^{2}_{\s\tau}$. Consequently, there are two Eichler-Shimura maps:
$$
\d^5_J: S_{\mu}(K_f) \toeq H^5_{cusp}(Y_E(K_f),\L_{\mu}(\C))_J \subset H^5_{cusp}(Y_E(K_f),\L_{\mu}(\C))
$$
for $J = \{\tau \}, \{\s\tau \}$. \\

\subsubsection{The involution $\e$}
\label{e_involution} We now assume that $E$ is a real quadratic field, and $\s$ denotes its Galois involution. Let $\e: G_E(\A) \to G_E(\A)$ denote the conjugation-duality involution, defined by $\e:= g \mapsto {}^t\s(g)^{-1}$. We may sometimes write $g^\e:= \e(g)$. Assume that the open compact subgroup $K_f$ is invariant by $\e$. Then, since $G_E(\Q)$, $K_\inf$ and $K_f$ are invariant by $\e$, the involution $\e$ induces an involution of the adelic variety $Y(K_f)$, which we also denote $\e$. Let $\mu \in X^+(T_E)$ and let $\mu^\e:= (\mu_{\s\tau}^\vee, \mu_{\tau}^\vee) \in X^+(T_E)$. Assume that $\mu^\e = \mu$. Let $A$ denote $\O$, $\K$ or $\C$. We let $\e$ acts on $L_{\mu}(A)$ by $\e:  P_\tau \otimes P_{\s\tau} \mapsto \vee(P_{\s\tau}) \otimes \vee(P_{\tau})$, where $\vee : L_{\mu_\vs}(A) \to L_{\mu_\vs^\vee}(A)$ (for $\vs \in \{ \tau, \s\tau \}$) is defined by:
\begin{equation}
\label{dual_involution}
\vee: P(U,V,W; X,Y,Z) \mapsto P(X,Y,Z ; U,V,W)
\end{equation} \\

The involutive action of $\e$ on $Y(K_f)$ and on $L_{\mu}(A)$ induces a $A$-linear involutive map on the cohomology groups, which we also denote $\e$:
$$
\e : H^\bullet(Y(K_f), \L_{\mu}(A)) \to H^\bullet(Y(K_f), \L_{\mu}(A))
$$

As in paragraph~\ref{hecke_corr}, we can similarly construct involutions $\e$ on the compactly supported cohomology. One can then check that the arrow $\i_A$ from the compactly supported cohomology to the cohomology is equivariant for the action of $\e$ on both sides. Therefore, we obtain an $A$-linear map on the interior cohomology: 
$$
\e : H^\bullet_!(Y(K_f), \L_{\mu}(A)) \to H^\bullet_!(Y(K_f), \L_{\mu}(A)).
$$
Finally, the cuspidal cohomology $H_{cusp}^\bullet(Y(K_f), \L_{\mu}(A))$ is invariant by the action of $\e$. \\

 Let $S$ be the set of finite places associated with $K_f$ in \S\ref{hecke_corr} (i.e. containing the finite places of $E$ above $p$ and the places where $K_f$ is not hyperspecial). If $w$ is a finite place corresponding to some prime ideal $\mathfrak{q}$ of $E$, we denote by $w^\s$ the finite place of $w$ corresponding to the prime ideal $\s(\mathfrak{q})$, where $\s$ is the non-trivial Galois involution of $E$. For our choice of $K_f$, one has $\s(S) = S$. We get an involution of $\O$-algebra $\e : T \mapsto T^\e$ on $\H({K_f};\O)$, by sending a double coset $K_w x K_w$ ($w \notin S$, $x \in \GL(E_w)$) to $K_{w^\s} \e(x) K_{w^\s}$. One checks that $T_{w,1}^\e = S_{w^\s}^{-1} T_{w^\s,2}$, $T_{w,2}^\e = S_{w^\s}^{-1} T_{w^\s,1}$ and $S_w^{\e} = S_{w^\s}^{-1}$. One can also check that the action of $\e$ is semi-linear with respect to the Hecke action, i.e. that for all $T \in \H({K_f};\O)$, we have:
$$
T \circ \e = \e \circ T^\e
$$
in $\End_\O(H^\bullet(Y(K_f), \L_{\mu}(\O)))$. In particular, if $\Si \in \mathrm{Coh}(\mathrm{GL}_{n/E},\mu,K_f)$ is a cohomological cuspidal automorphic representation,  then $\e$ sends the $[\Si]$-isotypic part of $H_{cusp}^\bullet(Y(K_f), \L_{\mu}(\C))$ to the $[\Si^\e]$-isotypic part.

\subsubsection{The $\e$-periods}
\label{e_periods}
We now consider a cohomological cuspidal automorphic representation $\Pi$ which is conjugate self-dual, i.e. $\Pi^\e \simeq \Pi$. Let $\mathfrak{n}(\Pi)$ be the mirahoric level of $\Pi$. Assume that $\mathfrak{n}(\Pi)$ is only divisible by split primes and choose a subset $R$ of the places dividing $\mathfrak{n}$ containing exactly one representative in each orbit for the action of $\s$ (see \S\ref{global_mirahoric_theory}). Let $K_f$ to be the mixed mirahoric subgroup $K_1^*(\mathfrak{n}(\Pi))$ of level $\mathfrak{n}(\Pi)$ and type $R$, and let $\phi_f \in \Pi_f$ be the mixed essential vector $\phi_{\Pi}^*$ defined in \S\ref{global_mirahoric_theory}. Recall that it is a $K_f$-fixed form. \\

Let us write $H^5(A) = H^5_{cusp}(Y_E(K_f),\L_{\mu}(A))[\Pi]$ for the $\Pi$-isotypic component of $H^5_{cusp}(Y_E(K_f),\L_{\mu}(A))$, if $A$ is $\O$ or $\C$. It follows from (\ref{deco_coho_cusp}), from (\ref{g-K_dim}), and the choice of $K_f$ that $H^5(\C)$ is a 2-dimensional $\C$-vector space, and is a direct sum of the two $1$-dimensional subspaces $H^5_{\{\t\}}$ and $H^5_{\{\s\t\}}$. Since $\Pi$ is conjugate self-dual, $H^5(\C)$ is invariant by $\e$. Moreover, $\e$ exchanges $H_{\{\tau\}}^5$ and $H_{\{\s\tau\}}^5$, and thus is not trivial on $H^5(\C)$. Hence, since the $\O$-lattice $H^5(\O) \subset H^5(\C)$ is invariant by $\e$, the two eigenspaces corresponding to the $\pm1$-eigenvalues for $\e$:
$$
H^5(\C)[\pm] = H^5_{cusp}(Y_E(K_f),\L_{\mu}(\C))[\Pi; \pm]
$$
are $1$-dimensional vector spaces over $\C$, endowed with the $\O$-integral structure $H^5(\O)[\pm]$. We then define two maps (one for $+$, one for $-$):
$$
\d^\pm_\e : \Pi_f^{K_f} \to H^5(\C)[\pm]
$$
by $\d^\pm_\e : = \d_{\{\tau\}}^5 \pm \e \circ \d_{\{\tau\}}^5$. Let $\xi_\e^\pm$ be an $\O$-base of $H^5(\O)[\pm]$. Then the two $\e$-periods $\Om_5(\Pi, \e, +)$ and $\Om_5(\Pi, \e,-)$ of $\Pi$ are defined to be the complex numbers $\Om_5(\Pi, \e, \pm) \in \C^\x$ such that:
$$
\d^\pm_\e(\phi_f) = \Om_5(\Pi, \e, \pm) \cdot \xi_\e^\pm
$$
for our specific choice of $K_f$-fixed form $\phi_f$ in $\Pi_f$. Since the definition of $\Om_5(\Pi, \e, \pm)$ depends on the choice of the $\O$-base $\xi_\e^\pm$, the periods are only defined up to multiplication by an element of $\O^\x$.  These two automorphic periods are called the $\e$-periods of $\Pi$. \\

\section{An adjoint $L$-value formula for $\GL(E)$}
\label{part_adjoint}

The goal of this section is to establish the adjoint $L$-value formula of \ref{adjoint_L_value}, that relates the special $L$-value $L(\Pi,\Ad,1)$ to the $\e$-periods of $\Pi$, for a cohomological cuspidal representation $\Pi$ of $\GL(E)$ which is conjugate self-dual.

\subsection{Normalization of Haar measures}
\label{Haar_measure_GLn}

For each place $w$ of $E$, the local Haar measure on $N_n(E_w)$ is fixed to be:
$$
du_w = \prod_{1\leq i<j\leq n} du_{ij}, \quad u_v = (u_{ij}) \in N_n(E_w)
$$ 
where $du_{ij}$ is the Haar measure on $E_w$ (see \S\ref{measures}). The global Haar measure on $N_n(\A_E)$ is then $du = \prod_w du_w$. This is the Tamagawa measure of $N_n(\A_E)$. Finally, on $N_{n-1}(E_w)$, the additive Haar measure is normalized so that $\vol(N_{n-1}(\O_w)) = 1$. \\

The normalisation of our Haar measures on $\GLn$ at finite places will depend on the choice of some open compact subgroup $K_f \subset \GLn(\A_{E,f})$. It will always be clear from the context what this subgroup is. Write $K_f = \prod_{w} K_w$, and let $w$ be a finite place of $E$. The Haar measure $dg_w$ on $\GLn(E_w)$ is normalized so that:
$$
\mathrm{vol}(K_w, dg_w) = 1.
$$ 
The Haar measure $dg_w$ on $\mathrm{GL}_{n-1}(E_w)$ is normalized so that:
$$
\mathrm{vol}(\mathrm{GL}_{n-1}(\O_w), dg_w) = 1.
$$ 

If $w$ is a real archimedean place, the Haar mesure $dg_w$ on $\GLn(E_w) = \GLn(\R)$ is normalized using the Iwasawa decomposition $g_w = nak$, with $n \in N_n(\R)$, $a \in (\R^\x)^n$ and $k \in \mathrm{O}(n)$ by:
$$
dg_v = \d^{-1}_{B_n(\R)}(a)du d^\x a dk
$$
where $du$ is the Haar measure on $N_n(\R)$, $d^\x a =  \prod_{i=1}^n d^\x a_i$ and $dk$ is the probability measure on $\mathrm{O}(n)$. The modulus character $\d_{B_n(\R)}$ is defined by $\d_{B_n(\R)}(a) = \prod_{i=1}^n |a_i |^{n-2i+1}$. Let $dg= \prod_v dg_v$ be the global Haar measure on $\GLn(\A_F)$. Since it will always be clear if our integrals are local or global we also write $dg$ for $dg_v$.

\subsection{The Petersson product as an adjoint $L$-value}

In this subsection we temporarily adopt general notation. Let $E$ be a number field and $n \geq 1$ is any integer. We fix $\psi_E$ to be the additive character $\psi_E := \psi_\Q \circ \mathrm{Tr}_{E/\Q} : E\bs \A_E \to \C^\x$. Write $\psi_E = \bigotimes_w \psi_w$ for its factorization as a tensor product of local additive characters. In this section we consider a cuspidal automorphic representation $\Pi$ of $\mathrm{GL}_n(\A_E)$. Let $\n$ be the mirahoric level of $\Pi$ and write $K_1(\n) = \prod_w K_w$. If $w$ is a finite place of $E$, the Haar measure $dg_v$ on $\GLn(E_w)$ is normalized so that $\vol(K_{w})= 1$. If $w$ is archimedean, the Haar measure $dg_w$ on $\GLn(E_w)$ is specified in \S\ref{measures}. The global Haar measure on $\GL(\A_E)$ is $dg = \prod_w dg_w$. The measure $dn = \prod_w dn_w$ on $N_n(\A_E)$ is the Tamagawa measure.

\subsubsection{Rankin-Selberg and adjoint $L$-functions}
\label{L-functions}

Let $\Pi$ be a cuspidal automorphic representation of $\GLn(\A_E)$ and let $\Pi^\vee$ be the contragredient of $\Pi$. Let $w$ be a finite place of $E$ and let $q_w$ be the cardinality of the residuefield of $E_w$. Let us write the standard local $L$-factors associated with $\Pi_w$ and $\Pi^\vee_w$ as (see \cite{ZFSA}): 
\begin{equation}
\label{standard_L_factor}
L(\Pi_w,s) = \prod_{i=1}^{r} (1 - \a_i q_w^{-s})^{-1} \quad \mbox{ and } \quad L(\Pi_w^\vee,s) = \prod_{j=1}^{r'} (1 - \b_j q_w^{-s})^{-1}
\end{equation}
with $\a_i, \b_i \in \C^*$. One has that $r \leq n$ and that $\Pi$ is ramified if and only if $r<n$ (see \cite[Section 3]{Jacquet79}). Then, the \textit{imprimitive} Rankin-Selberg $L$-factor of $\Pi_w$ is defined by:
$$
L^{imp}(\Pi_w \x \Pi_w^\vee,s) = \prod_{i=1}^r \prod_{j=1}^r (1 - \a_i \b_j q_w^{-s})^{-1}
$$

One can also consider the \textit{primitive} Rankin-Selberg local $L$-factor $L(\Pi_w \x \Pi_w^\vee,s)$, which is defined as the greatest common divisor in $\C[q_w^{\pm s}]$ of some local zeta integrals (see \cite{J-PS-S83} for details).
These two definitions are related by the following lemma (see \cite[Corollary 3.3]{Jo22}, the last statement is du to \cite[Section 2]{J-S81-I}):
\begin{lemma}
\label{comparison_local_factors}
There exists a polynomial $P \in \C[X]$ satisfying $P(0) = 1$ such that:
$$
L^{imp}(\Pi_w \x \Pi_w^\vee,s) = P(q_w^{-s}) \cdot L(\Pi_w \x \Pi_w^\vee,s)
$$
Moreover, $P=1$ when $\Pi$ is unramified at $w$.
\end{lemma}
Thus, when $\Pi$ is unramified at $w$, the primitive and the imprimitive local $L$-factors at $w$ are equal. Moreover, since $L(\Pi_w \x \Pi_w^\vee,s)$ has no pole at $s=1$, the above lemma implies that $L^{imp}(\Pi_w \x \Pi^\vee_w,s)$ has no pole at $s=1$ either. \\

The global \textit{imprimitive} Rankin-Selberg $L$-function of $\Pi$ is defined to be the product of the local $L$-factors:
$$
L^{imp}(\Pi \x \Pi^\vee,s) = \prod_{w} L^{imp}(\Pi_w \x \Pi^\vee_w,s)
$$
The following theorem follows from the work of Jacquet, Piatetskii-Shapiro and Shalika (see \cite{J-S81-I} and \cite{J-PS-S83}):
\begin{theorem}
\label{J-PS-S-83}
The function $L^{imp}(\Pi \x \Pi^\vee,s)$ converges absolutely, uniformly on compact subsets, in some right half-plane. It has analytic continuation as a meromorphic function to the right half plane $Re(s) > 1-\e$, for some small $\e>0$. It has an unique pole in $Re(s) > 1-\e$, which is located at $s = 1$ and is simple. \\
\end{theorem}

The \textit{imprimitive} adjoint $L$-function $L^{imp}(\Pi,\Ad,s)$ of $\Pi$ is defined by the following formula:
$$
L^{imp}(\Pi\x\Pi^\vee,s) = \zeta_E(s) \cdot L^{imp}(\Pi,\Ad,s)
$$
where $\zeta_E$ is the Dedekind zeta of $E$. By definition, $L^{imp}(\Pi,\Ad,s)$ is a meromorphic function on the right half plane $Re(s) > 1-\e$, for some small $\e>0$. Moreover, it is holomorphic and nonzero at $s=1$. The point $s=1$ is critical (in the sens of \cite{De79}) if and only if $n=2$ and $E$ is totally real (see \cite[Proposition 3.4.1]{BR17}).

\subsubsection{The Jacquet-Shalika formula}

In this paragraph we present a formula, du to Jacquet and Shalika, expressing the Petersson product of two vectors of a cuspidal automorphic representation $\Pi$ of $\mathrm{GL}_n(\A_E)$ and its dual $\Pi^\vee$, as the product of the special value at $s=1$ of the adjoint $L$-function of $\Pi$ by some local ramified and archimedean factors. It is a generalization of a Shimura's formula \cite{Shi76} in the case of modular forms. More precisely, if $\phi \in \Pi$ and $\phi' \in \Pi^\vee$, we recall that the Petersson product of $\phi$ and $\phi'$ is defined by:
\begin{equation}
\label{petersson_product}
\langle \phi, \phi' \rangle = \int_{[Z_n\bs\mathrm{GL}_n]}\phi(g){\phi'(g)}dg
\end{equation}

Let $\W(\Pi, \psi_E)$ denote the Whittaker model of $\Pi$ with respect to $\psi_E$. It decomposes as a restricted tensor product of local Whittaker models:
$$
\W(\Pi, \psi)= \bigotimes_w  \W(\Pi_w, \psi_w)
$$
For each place $w$ of $E$, and each Whittaker functions $W_w  \in \W(\Pi_w, \psi_w)$ and $W'_w  \in \W(\Pi_w^\vee, \psi_w^{-1} )$, we define the following bilinear pairing:
\begin{equation}
\label{pairing_whittaker}
\langle W_w,W_w' \rangle_w = \int_{N_{n-1}(E_w)\bs G_{n-1}(E_w)}W_{w}\left(\begin{array}{ll}
g & \\
& 1
\end{array}\right) {W'_{w}}\left(\begin{array}{ll}
 g & \\
& 1
\end{array}\right)dg
\end{equation}
The above pairing is well-defined (see \cite[Proposition 3.6]{Jo22}) and equivariant for the action of $\mathrm{GL}_{n}(E_w)$. We recall that $S_\Pi$ is the set of finite places where $\Pi$ is ramified. The Jacquet and Shalika's formula is stated in the following proposition:

\begin{prop}
\label{jacquet-shalika_formula}
Let $S = S_\inf \cup S_\Pi$. Let $\phi = \otimes_w \phi_w \in \Pi$ and $\phi' = \otimes_w \phi'_w \in \Pi^\vee$ decomposable vectors such that $\phi_w$ (resp. $\phi_w'$) is the local factor at $w$ of $\phi_{\Pi}^\circ$ (resp. $\phi_{\Pi^\vee}^\circ$) whenever $w \notin S$. Then, we have the following equality:
\begin{equation}
\langle \phi, \phi' \rangle =  n \cdot {N_{E/\Q}(\mathfrak{n})^n \cdot D_E^{-n(n+3)/4}}\cdot c_n(E_\inf) \cdot L^{S}(\Pi, \Ad,1) \cdot \prod_{w \in S} \langle W_{\phi,w},W_{\phi',w} \rangle_w
\end{equation}
where $\mathfrak{n}$ is the mirahoric level of $\Pi$, and $c_n(E_\inf) \in \R_{>0}$ is some computable constant depending on our choice of Haar measures at archimedean places.  In particular, when $n=3$ and $E$ is totally real of dimension $d$, one has $c_n(E_\inf) = (4\pi)^{-d}$.
\end{prop}

The above result computes all the local factor at places where $\Pi$ is unramified, even when $\psi_w$ is not unramified. Since it is slightly more precise than \cite[Proposition 3.1]{Zh14}, we quickly recall the proof here and refer to this reference for a more detailed presentation. Note that for the choice of Haar measures made in \cite{Zh14}, the constant $c_n(E_\inf)$ equals $1$.

\begin{proof}

Let $\Phi \in \mathcal{S}(\A_E^n)$ be a Schwartz–Bruhat function which is a tensor product $\Phi = \bigotimes_w \Phi_w$ of local Schwartz–Bruhat functions $\Phi_w \in \mathcal{S}(E_w^n)$ with:
\begin{itemize}
\item $\Phi_w$ is the characteristic function $\mathbf{1}_{\O_w^n}$ of $\O_w^n$ when $w\notin S_\Pi$ (i.e when $\Pi$ is unramified at $w$);
\item $\Phi_w$ is the characteristic function $\Phi_{w,c}$ of $(\wp_w^c)^{n-1} \x (1 + \wp^c_w)$ in $E_w^n$ when $w \in S_\Pi$, where $c = v_{\wp_w}(\n)$ is the mirahoric conductor of $\Pi_w$;
\item $\Phi_w$ is such that $\widehat{\Phi_w}(0) = 1$ when $w \in  S_\inf$. \\
\end{itemize}

We recall that the Fourier transform of $\Phi_w \in \mathcal{S}(E_w^n)$ is defined by:
$$
\widehat{\Phi_w}(x)= \int_{E_w^n} \Phi_w(y) \psi_w(\langle x,y \rangle) dy
$$
Consequently, when $w$ is a finite place:
$$
\widehat{\Phi_w}(0) = \left\{
    \begin{array}{ll}
        \mathrm{vol}(\O_w^n) = N(\mathfrak{d}_w)^{-n/2} & \mbox{if } w \notin S_\Pi \\
         N(\wp_w)^{-nc} \cdot \mathrm{vol}(\O_w)^n =  N(\wp_w)^{-nc} N(\mathfrak{d}_w)^{-n/2} & \mbox{if } w \in S_\Pi
    \end{array}
\right. 
$$
and thus $\widehat{\Phi}(0) = D_E^{-n/2} \cdot N_{E/\Q}(\n)^{-n}$. Jacquet and Shalika \cite[Section 4]{J-S81-I} have shown, using the Rankin-Selberg method, that (see \cite[Proposition 3.1]{Zh14} for details):

\begin{equation}
\label{Rankin-Selberg_residue}
\langle \phi, \phi' \rangle =  \frac{n}{\mathrm{vol}(E^\x \bs \A_E^1)\cdot  \widehat{\Phi}(0)} \cdot \mathrm{Res}_{s=1} \left( \prod_{w \notin S}\Psi(s,\Phi_w,W_{\phi,w},W_{\phi',w}) \right) \cdot  \prod_{w \in S } \Psi(1,\Phi_w,W_{\phi,w},W_{\phi',w})
\end{equation}
where the local zeta integrals are defined by:
$$
\Psi(s,\Phi_w,W_{\phi,w},W_{\phi',w}) = \int_{N_n(E_w) \bs \GLn(E_w)} \Phi_w((0,0,\dots,1)g) W_{\phi,w}(g) W_{\phi',w}(g) |\det(g)|_w^s dg
$$
for each place $w$ of $E$. It thus remains to compute the local factors $\Psi(s,\Phi_w,W_{\phi,w},W_{\phi',w})$ in (\ref{Rankin-Selberg_residue}). \\

\textit{Spherical places not dividing $\mathfrak{d}_E$.} Let $w$ be a finite place of $E$ where $\Pi$ is unramified. If $w$ does not divide the different $\mathfrak{d}_E$ of $E$, then the local component $\psi_w$ of our fixed additive character $\psi_E$ is unramified, i.e. $\psi_w(\O_w) = 1$ and $\psi_w(\varpi_w^{-1}) \neq 1$. In this case, $\vol(N_n(\O_w)) = 1$. Then $W_{\phi,w} = W_{\Pi_w}^\circ$ and $W_{\phi',w} = W_{\Pi_w^\vee}^\circ$ are the spherical vectors of $\Pi_w$ and $\Pi_w^\vee$ with respect to $\psi_w$ and $\psi_w^{-1}$, and we know from \cite[Lemma 2.3]{J-S81-I} that:
$$
\Psi(s,\Phi_w,W_{\phi,w},W_{\phi',w}) = L(\Pi_w \x \Pi_w^\vee,s)
$$

\textit{Spherical places dividing $\mathfrak{d}_E$.}  When $w$ divides $\mathfrak{d}$, $\psi_w$ has conductor $\mathfrak{d}_{w}^{-1}$. By assumption one has:
$$
W_{\phi,w}(g) = W_{\Pi_w}^\circ(\mathrm{diag}(d_w^{n-1},\dots, d_w,1)g)
$$
where $W_{\Pi_w}^\circ \in \W(\Pi_w,\psi_w^\circ)$ is the spherical vector with respect to the unramified additive character $\psi^\circ_w$ and $d_w \in E_w$ is the representative of $\mathfrak{d}_{w}^{-1}$ chosen in \S\ref{global_mirahoric_theory} so that $\psi_w(x) = \psi^\circ_w(d_wx) $. Then making change of variable and using \cite[Lemma 2.3]{J-S81-I}, we get:
\begin{equation}
\label{change_of_variable}
\begin{aligned}
\Psi(s,\Phi_w,W_{\phi,w},W_{\phi',w}) &= |d_w^{-1}|_w^{n(n-1)s/2} \Psi(s,\Phi_w,W_{\Pi_w}^\circ,W_{\Pi_w}^\circ) \\
&= |d_w^{-1}|_w^{n(n-1)(s-1/2)/2} L(\Pi_w \x \Pi_w^\vee,s)
\end{aligned}
\end{equation}
where the factor $|d_w^{-1}|_w^{-n(n-1)/4}$  appearing in the second equality is the volume of $N_n(\O_w)$ for our choice of measure. Gathering the above unramified local computations, we get:
$$
\Psi(s,\Phi_w,W_{\phi,w},W_{\phi',w}) =| \mathfrak{d}_E^S |^{n(n-1)(s-1/2)/2} L^S(\Pi\x \Pi^\vee,s) \prod_{w\in S} \Psi(s,\Phi_w,W_{\phi,w},W_{\phi',w})
$$

Substituting the above local computations in (\ref{Rankin-Selberg_residue}), we get:
$$
\langle \phi, \phi' \rangle =  \frac{n \cdot | \mathfrak{d}_E^S |^{n(n-1)/4} \cdot \mathrm{Res}_{s=1}\zeta_E(s)}{\mathrm{vol}(E^\x \bs \A_E^1)\cdot \widehat{\Phi}(0)} \cdot L^{S}(\Pi, \Ad,1) \cdot  \prod_{w \in S} \Psi(1,\Phi_w,W_{\phi,w},W_{\phi',w})
$$
With our choice of Haar measure on $\A_E$, one has that $\mathrm{vol}(E^\x \bs \A_E^1) = \mathrm{Res}_{s=1}\zeta_E(s)$ (see for exemple \cite[Theorem 4.11.3]{Leahy}) so the front constant simplifies. It remains to simplify the local factors at ramification and archimedean places. \\

\textit{Ramification and archimedean places.} If $w \in S$, there exists a constant $c_n(E_w) \in \R$, depending on our choice of Haar measures, such that (see \cite[(3.9)]{Zh14}):
$$
\Psi(1,\Phi_w,W_{\phi,w},W_{\phi',w})= c_n(E_w)  \cdot \widehat{\Phi}_w(0) \cdot \langle W_{\phi,w},W_{\phi',w} \rangle_w
$$
for all $W_{\phi,w} \in \W(\Pi_w,\psi_w)$, $W_{\phi',w} \in \W(\Pi_w^\vee,\psi_w^{-1})$ and $\Phi_w \in \mathcal{S}(E_w^n)$. This constant is equal to $1$ for the choice of Haar measures made in \cite[\S 2.1]{Zh14}, which is different from ours. The constant $c_n(\R) \in \R$ is computed in \cite[\S 14.11, Corollary]{AGBI}. When $n=3$, one has that $c_n(\R) = (4\pi)^{-1}$ (see \cite[Remark 5.2]{Che22}). If $w$ is non-archimedean, let $c>0$ be the mirahoric conductor of $\Pi_w$. Then one can see by comparing \cite[Theorem 3.2]{Jo22} and \cite[Theorem 3.7]{Jo22} that:
$$
c_n(E_w) = \left\{
    \begin{array}{ll}
        \widehat{\Phi_{w,c}}(0)^{-1} & \mbox{if } w  \mbox{ does not divide } \mathfrak{d}_E \\
        \widehat{\Phi_{w,c}}(0)^{-1} |\mathfrak{d}_w|_w^{n(n-1)/4} & \mbox{if } w  \mbox{ divides } \mathfrak{d}_E
    \end{array}
\right. 
$$
where $\Phi_{w,c}$ is the characteristic function of $(\wp_w^c)^{n-1} \x (1 + \wp^c_w)$ in $E_w^n$. Gathering all the above formulas, we get the announced formula, with $c_n(E_\inf) := \prod_{w\mid \inf} c_n(E_w)$.
\end{proof}

\subsubsection{Local factors at ramification places}

From now, assume that $E$ is a quadratic number field, that $\Pi$ is conjugate self-dual and is only ramified above split primes. Let $S_\Pi$ be the set of finite places where $\Pi$ is ramified. We choose a particular subset $R$ of $S_\Pi$ such that $S_\Pi = R \sqcup \s(R)$, where $\s$ is the Galois involution of $E$. Let $K_f$ be the mixed mirahoric subgroup $K_1^*(\mathfrak{n}(\Pi))$ of level $\mathfrak{n}(\Pi)$ and type $R$. We then refine \ref{jacquet-shalika_formula} by computing the ramified local factors when the test vectors $\phi_f \in \Pi_f$ and $\phi_f' \in \Pi_f^\vee$ are the mixed essential vectors $\phi_\Pi^*$ of $\Pi$ and $\phi_{\Pi^\vee}^*$ of $\Pi^\vee$, defined in \S\ref{global_mirahoric_theory}. \\

The local computations in this paragraph are valid for $\GLn$ and for any non-archimedean local field $L$ of characteristic zero. Thus, we keep general notation in the following lemma:
\begin{lemma}
\label{split_ramified_computations}
Let $L$ be a non-archimedean local field of characteristic zero and let $\pi$ be an irreductible admissible representation of $\GLn(L)$ which is generic. Let $\psi$ be a non-trivial unramified additive character of $L$. Let $\langle \cdot,\cdot \rangle : \W(\pi,\psi) \x \W(\pi^\vee,\psi^{-1}) \to \C$ be the pairing defined by:
$$
\langle W,W' \rangle = \int_{N_{n-1}(L)\bs \mathrm{GL}_{n-1}(L)}W\left(\begin{array}{ll}
g & \\
& 1
\end{array}\right) {W'}\left(\begin{array}{ll}
 g & \\
& 1
\end{array}\right)dg
$$
Suppose that $\pi$ is ramified. Let $W_\pi^\circ$ (resp. $W_{\pi^\vee}^\circ$) be the essential vector of $\pi$ (resp. $\pi^\vee$) with respect to $\psi$ (resp. $\psi^{-1}$). Then one has:
$$
\langle W_\pi^\circ,W_{\pi^\vee}^\circ \rangle = \langle (W_{\pi^{\vee}}^\circ)^\vee,(W_{\pi}^\circ)^\vee\rangle = L^{imp}(\pi\x\pi^\vee,1)
$$
\end{lemma}

\begin{proof}
Let $\O_L$ be the valuation ring of $L$, $\varpi$ be an uniformizer of $L$ and let $q = \#(\O_L/(\varpi))$. We recall that the Haar measures on $\mathrm{GL}_{n-1}(L)$ and $N_{n-1}(L)$ are normalized so that $\vol(\mathrm{GL}_{n-1}(\O_L)) = \vol(N_{n-1}(\O_L))=1$. We first prove $(i)$. Using \cite[Proposition 4.7]{AM17}, one checks that:
$$
\begin{aligned}
\langle (W_{\pi^{\vee}}^\circ)^\vee,(W_{\pi}^\circ)^\vee\rangle &= \e(\pi^{\vee},\psi^{-1},1/2)^{n-1}\e(\pi,\psi,1/2)^{n-1} \langle \pi(\varpi_{n-1}^c)(W_\pi^\circ) ,\pi^\vee(\varpi_{n-1}^c)(W_{\pi^\vee}^\circ)\rangle\\
& =\langle W_\pi^\circ,W_{\pi^\vee}^\circ \rangle
\end{aligned}
$$
where the second equality follows from the $\GLn(L)$-equivariance of $\langle \cdot,\cdot \rangle$ and the fact that $\e(\pi^{\vee},\psi^{-1},1/2) = \e(\pi,\psi,1/2)^{-1}$. Now, using the Iwasawa decomposition one has that:
$$
\begin{aligned}
\langle W^\circ_{\pi},W^\circ_{\pi^\vee}  \rangle  &=  \int_{T_{n-1}(L)} \int_{\mathrm{GL}_{n-1}(\O_L)} W^\circ_{\pi}\left(\begin{array}{ll}
tk & \\
& 1
\end{array}\right) W^\circ_{\pi^\vee} \left(\begin{array}{ll}
tk & \\
& 1
\end{array}\right) \d_{B_{n-1}(L)}^{-1}(t)dkdt \\
&= \int_{T_{n-1}(L)} W^\circ_{\pi}\left(\begin{array}{ll}
t & \\
& 1
\end{array}\right) W^\circ_{\pi^\vee} \left(\begin{array}{ll}
t & \\
& 1
\end{array}\right) \d_{B_{n-1}(L)}^{-1}(t)dt \\
&=  \sum_{t \in T_{n-1}(L)/T_{n-1}(\O_L)} W^\circ_{\pi}\left(\begin{array}{ll}
t & \\
& 1
\end{array}\right) W^\circ_{\pi^\vee} \left(\begin{array}{ll}
t & \\
& 1
\end{array}\right) \d_{B_{n-1}(L)}^{-1}(t)\\
&= \sum_{\l \in \Z^{n-1}} W^\circ_{\pi}\left(\begin{array}{ll}
\varpi^\l & \\
& 1
\end{array}\right) W^\circ_{\pi^\vee} \left(\begin{array}{ll}
\varpi^\l & \\
& 1
\end{array}\right) \d_{B_{n-1}(L)}^{-1}(\varpi^\l) \\
\end{aligned}
$$
where $\d_{B_m(L)}(\varpi^\mu) = q^{-\sum_{i=1}^m (m+1-2i) \mu_i}$ (for $m \geq 1$ and $\mu \in \Z^m$) is the modulus character of $B_m(L)$. The second equality holds because $W^\circ_{\pi}$ and $W^\circ_{\pi^\vee}$ are fixed by the mirabolic subgroup and because the volume of $\mathrm{GL}_{n-1}(\O_L)$ with respect to the Haar measure on $\mathrm{GL}_{n-1}(L)$ is equal to 1. Write the standard $L$-functions of $\pi$ and $\pi^\vee$ as:
$$
L(\pi,s) = \prod_{i=1}^{r} (1 - \a_i q^{-s})^{-1} \quad \mbox{ and } \quad L(\pi^\vee,s) = \prod_{j=1}^{r} (1 - \b_j q^{-s})^{-1}
$$
where $r<n$ and $\a_i, \b_i \in \C^\x$, as in (\ref{standard_L_factor}). Let $\a := (\a_i)_{i=1}^{r}$ and $\b := (\b_i)_{i=1}^{r}$. Then, using the explicit values of the essential vectors on $T_{n-1}(L)$, we obtain:
$$
\begin{aligned}
\langle W^\circ_{\pi},W^\circ_{\pi^\vee}  \rangle &= \sum_{\l} \d_{B_n(L)}\left(\begin{smallmatrix} \varpi^\l & 0\\ 0 & 1\end{smallmatrix}\right)  \d_{B_{n-1}(L)}^{-1}(\varpi^\l) s_\l(\a) s_\l(\b)\\
&= \sum_{\l}  q^{-\Tr\l}s_\l(\a) s_\l(\b) =  \sum_{\l} s_\l(q^{-1/2}\a) s_\l(q^{-1/2}\b) 
\end{aligned}
$$
where sums are indexed by $\l \in \Z^{n-1}$ such that $\l_1 \geq \dots \geq \l_{n-1} \geq 0$ and $\Tr\l:= \l_1+ \dots + \l_{n-1}$ is equal to the degree of $s_\l$. Then, by \cite[Formula (4.3) p.63]{Macdonald}, we have:
$$
\langle W^\circ_{\pi},W^\circ_{\pi^\vee}  \rangle = \prod_{i,j} ( 1 - \a_i \b_jq^{-1})^{-1} = L^{imp}(\pi \x \pi^\vee,1)
$$
See also \cite[Theorem 3.7]{Jo22} for a slightly different proof using Matringe's formula for essential vectors \cite[Formula (1)]{Matringe13}.
\end{proof}

Let $w$ be finite place of $E$ where $\Pi$ (and $\Pi^\vee$) is ramified. Then, the local components $\phi_w$ and $\phi_w'$ of $\phi_f$ and $\phi_f'$ at $w$ are either both the essential vectors $\phi^\circ_{\Pi_w}$ and $\phi^\circ_{\Pi_w^\vee}$ or both the transposed essential vector ${}^\vee\phi^\circ_{\Pi_w}$ and ${}^\vee\phi^\circ_{\Pi_w^\vee}$ with respect to $\psi_w$ (which is unramified, since $w$ is split over $\Q$). In both cases, it follows from \ref{split_ramified_computations} that:
$$
\langle W_{\phi,w}, W_{\phi,w'} \rangle_w =L^{imp}(\pi\x \pi^\vee,1)
$$
\newline

The above computations, together with \ref{jacquet-shalika_formula}, give the following proposition:

\begin{prop}
Assume that $\Pi$ is conjugate self-dual and that $\Pi$ is only ramified above split primes. Let $\phi_f \in \Pi_f$ and $\phi_f' \in \Pi_f^\vee$ be the mixed essential vectors of $\Pi$ and $\Pi^\vee$. Let $\phi = \phi_f \otimes \phi_\inf$ and $\phi' = \phi'_f \otimes \phi'_\inf$, for some decomposable archimedean form $\phi_\inf \in \Pi_\inf$ and $\phi_\inf' \in \Pi_\inf^\vee$. Then: \\
\begin{equation}
\label{full_jacquet-shalika_formula}
\langle \phi, \phi' \rangle =  n \cdot N_{E/\Q}(\mathfrak{n})^n \cdot D_E^{-n(n+3)/4} \cdot c_n(E_\inf) \cdot L^{imp}(\Pi, \Ad,1) \cdot  \prod_{w \in S_\inf}  \langle W_{\phi,w},W_{\phi',w} \rangle_w
\end{equation}
where $ c_n(E_\inf) \in \R^\x$ is some computable constant. In particular $c_3(E_\inf) = (4\pi)^{-1}$.
\end{prop}

\subsection{A cohomological interpretation of the Jacquet-Shalika formula}
\label{proof_adjoint}
From now we assume that $E$ is a real quadratic field. Let $\Pi$ be a cohomological cuspidal automorphic representation of $\GL(\A_E)$ which conjugate self-dual and only ramified above split places. Its cohomological weight is denoted $\mu = (\mu_\tau,\mu_{\s\tau}) \in X^+(T_E)$. Recall that $K_f$ denotes the mixed mirahoric subgroup $K_1^*(\mathfrak{n}(\Pi))$ of level $\mathfrak{n}(\Pi)$ and type $R$, and $\phi_f$ be the mixed essential vector $\phi^*_\Pi \in \Pi_f$ of $\Pi$.

\subsubsection{A Poincaré pairing}
\label{pairings}

Let $A$ be any ring and $M$, $N$ two $A$-modules. Let $[\cdot,\cdot]: M \x N \to$ be a pairing (i.e a bilinear map) between $M$ and $N$. From $[\cdot,\cdot]$ we can construct the two following $A$-linear maps:
$$
M \to \Hom_A(N,A) \quad \mbox{ and } \quad N \to \Hom_A(M,A)
$$
We recall that $[\cdot,\cdot]$ is said to be non-degenerate if these two maps are injective, and perfect if they are isomorphisms. If $A$ is a field and $M$ and $N$ have finite dimensions, these two notions coincide.

\paragraph{A pairing on the cuspidal cohomology.} Let first $K$ denote $\C$ or some sufficiently large $p$-adic field $\K$. Consider the pairings:
$$
[ \cdot,\cdot ]:  H^{5}(Y(K_f), \L_{\mu}(K)) \x H_{c}^{5} ( Y(K_f), \L_{\mu^\vee}(K)) \to K
$$
$$
[ \cdot,\cdot ]:  H^{5}_c(Y(K_f), \L_{\mu}(K)) \x H^{5} ( Y(K_f), \L_{\mu^\vee}(K)) \to K
$$
constructed using cup product and the perfect pairing $\langle \cdot,\cdot\rangle_{\mu}: \L_{\mu}(K)\otimes
\L_{\mu^\vee}(K) \to K$ described in paragraph~\ref{alg_irrep}. By Poincaré duality (see for exemple Theorem 4.8.9 of \cite{LAG1}), these pairings are perfect. Moreover, we have the following commutative diagram:
$$
\xymatrix{
H^{5}(Y(K_f), \L_{\mu}(K)) \x H_c^{5} ( Y(K_f), \L_{\mu^\vee}(K)) \ar[r] & K\\
    H^{5}_c (Y(K_f), \L_{\mu}(K)) \x H_c^{5} ( Y(K_f), \L_{\mu^\vee}(K)) \ar@<-50pt>@{=}[d] \ar@<50pt>[d] \ar@<-50pt>@{=}[u] \ar@<50pt>[u] \ar[r] & K \ar@{=}[d]  \ar@{=}[u] \\
     H^{5}_c(Y(K_f), \L_{\mu}(K)) \x H^{5} ( Y(K_f), \L_{\mu^\vee}(K)) \ar[r] & K \\
}
$$
so these two pairings induce a pairing on the interior cohomology, again denoted $[ \cdot,\cdot ]$:
\begin{equation}
\label{poincare_pairing}
[ \cdot,\cdot ] : H^{5}_! (Y(K_f), \L_{\mu}(K)) \x H_!^{5} ( Y(K_f), \L_{\mu^\vee}(K)) \to K
\end{equation}

It is not difficult to see that this pairing is perfect as its non-degeneracy follows from the non-degeneracy of the two first pairings. Let $\O$ be the integer ring of $\K$. We now describe how to define the above pairing on the torsion-free quotient of the cohomology groups with coefficients in $\O$. Let $\L_{\mu}(\O)^\vee \subset  \L_{\mu^\vee}(\K)$ be the dual lattice of $\L_{\mu}(\O)$ for $\langle \cdot,\cdot\rangle_{\mu}$. Recall from paragraph~\ref{alg_irrep}, that when $\mu$ is $p$-small, one has in fact $\L_{\mu}(\O)^\vee= \L_{\mu^\vee}(\O)$. As above, we have the following two pairings:
$$
[\cdot,\cdot]: H^{5}(Y(K_f), \L_{\mu}(\O)) \x H_{c}^{5} ( Y(K_f), \L_{\mu}(\O)^\vee) \to \O
$$
$$
[\cdot,\cdot]:  H^{5}_c(Y(K_f), \L_{\mu}(\O)) \x H^{5} (Y(K_f), \L_{\mu}(\O)^\vee) \to \O
$$
where the cohomology groups are the torsion-free part of the cohomology groups with coefficients in $\O$. These two pairing are perfect (see the second part of Theorem 4.8.9 of \cite{LAG1}, be careful that there non-degenerate mean perfect). The same diagram as above shows the existence of a pairing on the inner cohomology:
\begin{equation}
\label{integral_poincare_pairing}
[ \cdot,\cdot ] : H^{5}_! (Y(K_f), \L_{\mu}(\O)) \x H_!^{5} (Y(K_f), \L_{\mu}(\O)^\vee) \to \O
\end{equation}
An important property of this pairing is that it is Hecke-equivariant:
\begin{lemma} 
\label{poincare_equivariance}
The pairing $[ \cdot,\cdot ]$ is Hecke-equivariant, i.e. for all $T \in \H(K_f; \O)$:
$$
[ T\cdot x, y ] = [ x, T^\vee \cdot  y ], 
$$
for $x \in H^{5}_{!}(Y(K_f), \L_{\mu}(\O))$ and $y \in H_{!}^{5} ( Y(K_f), \L_{\mu}(\O)^\vee)$, where $T^\vee$ has been implicitly defined in \ref{e_involution}.
\end{lemma}

The pairing (\ref{integral_poincare_pairing}) coincides with the pairing on the cohomology groups with coefficient in $\K$ when extending the scalars to $\K$. Thus, this pairing is non-degenerate but may not be perfect. However, it is perfect when localized at some non-Eisenstein ideal, as we now explain. Let $h(K_f; \O)$ be the $\O$-algebra of Hecke correspondences on $H^{\bullet}(Y(K_f), \L_{\mu}(\O))$. We let $h(K_f; \O)$ act on $H_{!}^{5} ( Y(K_f), \L_{\mu}(\O)^\vee)$ by:
$$
t(x) = ([\, \vee \,] \circ t \circ [\, \vee \,]^{-1})(x)
$$
where $[\, \vee \,]: H^{5}_{!}(Y(K_f), \L_{\mu}(\O)) \to H^{5}_{!}(Y(K_f^\vee), \L_{\mu}(\O)^\vee)$ has been implicitly defined in \ref{e_involution}. Be careful that $t$ on the right-hand side acts on the cohomology of level $K_f^\vee$. Then we see from \ref{poincare_equivariance} that the pairing $[\cdot, \cdot ]$ is equivariant with respect to this action of $h(K_f; \O)$. Let $\m \subset h(K_f; \O)$ be some non-Eisenstein maximal ideal. By localization, we obtain the following pairing:
\begin{equation}
\label{localized_poincare_pairing}
[ \cdot,\cdot ]: H^{5}_! (Y(K_f), \L_{\mu}(\O))_{\m} \x H_!^{5} (Y(K_f), \L_{\mu}(\O)^\vee)_{\m} \to \O
\end{equation}
As explained in \cite[\S 4.2.4]{BR17}, this pairing is perfect if the localized boundary cohomology $H^5(\partial(Y(K_f)),\L_{\mu}(\O))_{\m}$ has no torsion. However, it follows from \cite[Theorem 4.2]{NT16} that $H^5(\partial(Y(K_f)),\L_{\mu}(\O))_{\m} = 0$ if $\m$ is non-Eisenstein. We record this discussion in the following lemma:

\begin{lemma}
Let $\m \subset h(K_f;\O)$ be a maximal ideal. Then the pairing:
$$
[ \cdot,\cdot ]: H^{5}_! (Y(K_f), \L_{\mu}(\O))_{\m} \x H_!^{5} (Y(K_f), \L_{\mu}(\O)^\vee)_{\m} \to \O
$$
is $\TT$-equivariant. Moreover, if $\m$ is non-Eisenstein, it is perfect.
\end{lemma}

Before going into the details of archimedean computations, we record the following consequence of the above formula:
\begin{lemma}
Suppose that $\mu^\e = \mu$ and $K_f^\e = K_f$ so that $\e$ acts on the inner cohomology groups. Then the pairing $[ \cdot , \cdot ]$ defined on the cuspidal cohomology is anti-equivariant with respect to $\e$, i.e:
$$
[ \e \cdot x, y ] = - [ x, \e \cdot  y ], 
$$
where $x \in H^{5}_{!}(Y(K_f), \L_{\mu}(\O))$ and $y \in H_{!}^{5} (Y(K_f), \L_{\mu}(\O)^\vee)$.
\end{lemma}

\subsubsection{A pairing on the $(\g,K_\inf)$-cohomology}
In this paragraph, we define a pairing:
\begin{equation}
\label{gK_pairing}
B : H^5(\g, K_\inf  ; \W(\Pi_\inf,\psi_\inf) \otimes L_{\mu}(\C)) \x H^5(\g, K_\inf  ; \W(\Pi_\inf^\vee,\psi_\inf^{-1}) \otimes L_{\mu^\vee}(\C)) \to \C
\end{equation}
 on the $(\g,K_\inf)$-cohomology, which is the counterpart of the Poincaré pairing $[ \cdot, \cdot ]$. Recall that we have the following expression for this cohomology group
(see \cite[II, Proposition 3.2]{BW00}):
$$
H^5(\g, K_\inf  ; \W(\Pi_\inf,\psi_\inf) \otimes L_{\mu}(\C)) =\left(\bigwedge^5 \p_\C^* \otimes \W(\Pi_\inf,\psi_\inf) \otimes L_{\mu}(\C) \right)^{K_\inf}
$$
Then $B$ is defined as a tensor product on the right hand side. First, consider the pairing $s: \bigwedge^5 \p_{\C}^* \x \bigwedge^5  \p_{\C}^* \to \C$ defined by:
$$
X_1^* \wedge \dots \wedge X_5^* \wedge Y_1^* \wedge \dots \wedge Y_5^* =  s(X_1^* \wedge \dots \wedge X_5^*,Y_1^* \wedge \dots \wedge Y_5^*) \times \bigwedge_{i=1}^5 X_{i,\s}^* \wedge  \bigwedge_{i=1}^5 X_{i,\s\tau}^*
$$
where $(X_{i,\s})$ (resp. $(X_{i,\s\tau})$) is some ordered basis of $\p_{\tau,\C} = \p_{3,\C}$ (resp. of $\p_{\s\tau,\C} = \p_{3,\C}$). Then $B$ is defined by tensor product:
$$
B(\a \otimes W_1 \otimes P, \b \otimes W_2 \otimes Q) = s(\a,\b) \cdot  \langle W_1, W_2\rangle_\inf \cdot \langle P, Q \rangle_{\mu}
$$
where $\langle \cdot, \cdot \rangle_\inf:\W(\Pi_\inf,\psi_\inf) \x \W(\Pi_\inf^\vee,\psi_\inf^{-1}) \to \C$ and  $\langle \cdot, \cdot \rangle_{\mu}: L_{\mu}(\C) \x L_{\mu^\vee}(\C) \to \C$ have been defined respectively in \ref{pairing_whittaker} and \ref{pairing_coefficients}. \\

Recall from \S\ref{eichler-shimura_maps} that we have the following decomposition by the Kunneth formula:
$$
H^5(\g, K_\inf ;  \Pi_\inf \otimes L_{\mu}(\C)) = H^5(\Pi_\inf \otimes L_{\mu}(\C))_{\{\tau\}} \oplus H^5(\Pi_\inf \otimes L_{\mu}(\C))_{\{\s\tau\}}
$$
The one checks that the summands $H^5(\Pi_\inf \otimes L_{\mu}(\C))_{\{\tau\}}$ and $H^5(\Pi_\inf \otimes L_{\mu}(\C))_{\{\s\tau\}}$ are respectively orthogonal to $H^5(\Pi_\inf^\vee \otimes L_{\mu^\vee}(\C))_{\{\tau\}}$ and $H^5(\Pi_\inf^\vee \otimes L_{\mu^\vee}(\C))_{\{\s\tau\}}$ for  the pairing $B$. \\

We conclude this paragraph with the following lemma, concerning the anti-equivariant property of $B$ with respect to the involutive action of  $\e$:
\begin{lemma}
\label{B_epsilon}
Suppose that $\mu^\e = \mu$ and $\Pi_\inf^\e = \Pi_\inf$, then $B$ is anti-equivariant with respect to the conjugaison-duality involution $\e$, i.e:
$$
B(\e(c), \e(c')) = - B(c, c')
$$
for all $c \in H^5(\g, K_\inf ;  \W(\Pi_\inf,\psi_\inf) \otimes L_{\mu}(\C))$ and $c' \in H^5(\g, K_\inf ;  \W(\Pi_\inf^\vee,\psi_\inf^{-1}) \otimes L_{\mu^\vee}(\C))$
\end{lemma}

\subsubsection{A cohomological interpretation of the Jacquet-Shalika formula}
\label{coho_interpretation_petersson}

In this paragraph, we give a cohomological interpretation of Jacquet-Shalika formula, using the following comparison isomorphism:
$$
\delta:  \Pi_f^{K_f} \otimes H^5(\g,K_\inf ; \W(\Pi_\inf,\psi_\inf) \otimes L_{\mu}(\C)) \toeq H^5_{cusp}(Y_E(K_f),\L_{\mu}(\C))[\Pi_f]
$$
used in \S\ref{eichler-shimura_maps} to construct the Eichler-Shimura maps.
Let $\mathfrak{X} = \sum_{i\in I} \w_i \otimes W_i \otimes P_i$ and $\mathfrak{X}' = \sum_{j\in J} \w_j \otimes W_j' \otimes P_j'$ be some elements respectively in $H^5(\g,K_\inf ; \W(\Pi_\inf,\psi_\inf) \otimes L_{\mu}(\C))$ and $H^5(\g,K_\inf ; \W(\Pi_\inf^\vee,\psi_\inf^{-1}) \otimes L_{\mu^\vee}(\C))$. Then, if $\phi_f \in \Pi_f$ and $\phi_f' \in \Pi_f^\vee$ are some $K_f$-fixed vectors, one has that:
$$
[ \d(\phi_f \otimes \mathfrak{X}),\d(\phi_f' \otimes \mathfrak{X}') ] =\int_{Y_E(K_f)} \underbrace{\sum_{i\in I} \sum_{j\in J} s(\w_i,\eta_j) \cdot \langle P_i, P_j' \rangle_\mu \cdot \big( \phi_i \wedge \phi_j' \big)}_{=: \,  \w \in H^{6}_c(Y_E(K_f),\C)}
$$
where $\phi_i \in \Pi$ (resp. $\phi_j'\in \Pi^\vee$) is the form $\phi_f \otimes \phi_{i}$ (resp. $\phi_f' \otimes \phi_{j}'$), for $\phi_{i} \in \Pi_\inf$ (resp. $\phi_{j}' \in \Pi_\inf^\vee$) corresponding to the Whittaker function $W_i$ (resp. $W_j'$). Then, similarly to \cite[paragraph 3.3.3]{BR17}, one has: 
$$
\begin{aligned}
\int_{Y_E(K_f)} \w dg &= \overbrace{\vol(K_\inf)^{-1}}^{=4} \int_{Z_\inf^\circ \GL(E) \bs \GL(\A_E) / K_f} \w dg \\
&= 4 \cdot \int_{Z_\inf \GL(E) \bs \GL(\A_E) / K_f} \left( \sum_{c \in Z_\inf /Z_\inf^\circ} c \cdot \w \right) dg \\
&= 16 \cdot \int_{Z_\inf \GL(E) \bs \GL(\A_E) / K_f} \w dg \\
&= 16 \cdot h(K_f) \int_{Z_3(\A_E) \GL(E) \bs \GL(\A_E) / K_f} \w dg \\
&= 16 \cdot h(K_f) \cdot \vol(K_f, dg)^{-1} \cdot \int_{Z_3(\A_E) \GL(E) \bs \GL(\A_E)} \w dg
\end{aligned}
$$
The third equality, follows from the fact that $Z_\inf /Z_\inf^\circ$ acts via the trivial character $1$ on $\w$. The Haar measure on $\GL(\A_E)$ is normalized so that the volume of $K_f$ is $1$. Finally, we get:
\begin{equation}
\label{petersson_coho_delta}
[ \d(\phi_f \otimes \mathfrak{X}),\d(\phi_f' \otimes \mathfrak{X}') ]= 16 h(K_f) \sum_{i\in I} \sum_{j \in J} s(\w_i,\w_j) \cdot  \langle P_i,P_j' \rangle_{\mu} \cdot \langle \phi_i,\phi_j' \rangle
\end{equation}
where $\phi_i \in \Pi$ (resp. $\phi_j'\in \Pi^\vee$) is the form $\phi_f \otimes \phi_{i}$ (resp. $\phi_f' \otimes \phi_{j}'$), for $\phi_{i} \in \Pi_\inf$ (resp. $\phi_{j}' \in \Pi_\inf^\vee$) corresponding to the Whittaker function $W_i$ (resp. $W_j'$). Note that this formula is true for any $\Pi \in \mathrm{Coh}(G_E,\mu,K_f)$. In addition:
$$
h(K_f) := \mathrm{vol}(Z(E)\bs Z(\A_{E,f}) / K_f \cap Z(\A_{E,f})) = \mathrm{vol}(E^\x \bs \A_{E,f}^\x / U(\mathfrak{n})) = \#\mathrm{Cl}_E(\mathfrak{n})
$$
where $U(\mathfrak{n}) = \prod_{w} U_w(\n_w) \subset \A_{E,f}$ and $U_w(\mathfrak{n}_w) = \{\a \in \O_w, \a \cong 1 \mod \mathfrak{n}_w\}$. Then $h(K_f)$ is simply the cardinal $h_E(\mathfrak{n})$ of $\mathrm{Cl}_E(\mathfrak{n})$ the wide ray class group of level $\mathfrak{n}$. Suppose now that $\phi_f$, $\phi_f'$ and each $W_i$, $W_j'$ are pure tensors. Using formula~(\ref{full_jacquet-shalika_formula}), we obtain that:
$$
[ \d(\phi_f \otimes \mathfrak{X}),\d(\phi_f' \otimes \mathfrak{X}') ] = C_2  \cdot {L^{imp}(\Pi,\Ad,1)} \x B(\mathfrak{X},\mathfrak{X}')
$$
where $C_2 = 16 h_E(\mathfrak{n}) \cdot C_1$, with $C_1$ is the front constant of formula~(\ref{full_jacquet-shalika_formula}). Recall that $\Pi$ is supposed to be conjugate self-dual. As a particular case of the above formula, we have:
\begin{equation}
\label{petersson_coho}
[ \d_\e^\pm(\phi_f),\d_\e^\pm(\phi_f')] = C_2  \cdot {L^{imp}(\Pi,\Ad,1)} \x B([\Pi_\inf]_\e^\pm,[\Pi_\inf^\vee]_\e^\mp)
\end{equation}
where $\d_\e^\pm$ are the Eichler-Shimura maps defined in paragraph~\ref{e_periods} using the explicit elements $[\Pi_\inf]_\e^\pm$ and $[\Pi_\inf^\vee]_\e^\pm$. \\

\subsubsection{Archimedean computations} 
\label{archimedean_adjoint}

We now compute the explicit value of $B([\Pi_\inf]^\pm_\e,[\Pi_\inf^\vee]^\mp_\e)$ for the explicit choice of $[\Pi_\inf]^\pm_\e$ and $[\Pi_\inf^\vee]^\pm_\e$ giving the Eichler-Shimura isomorphisms $\d_\e^\pm$ defined in paragraph~\ref{eichler-shimura_maps}. We write $\Pi_\inf = \Pi_\tau \otimes \Pi_{\s\tau}$ for the archimedean part of $\Pi$. Since $\Pi$ is conjugate self-dual, we have that $\Pi_{\s\tau} = \Pi_\s^\vee$. The maps $\d_\e^\pm$ corresponds to $[\Pi_\inf]^\pm_\e = [\Pi_\inf]_{\{ \tau\}} \pm \e([\Pi_\inf]_{\{\tau\}})$. Using the equivariance properties of $B$ with respect to $\e$, as well as the explicit action of $\e$ on the Chen's generators, the same kind of calculations as in the proof of \cite[Lemme 5.3]{Che22}) gives:
$$
B([\Pi_\inf]^\pm_\e,[\Pi_\inf^\vee]^\mp_\e) \sim \pi^2 \cdot L(\Pi_\inf \x \Pi_\inf^\vee,1) 
$$

Finally, gathering all the previous computations, we get:

\begin{prop}
\label{pairing_computation}
Let $\phi_f \in \Pi_f$ (resp. $\phi_f' \in \Pi_f^\vee$) be the mixed essential vector of $\Pi$ (resp. of $\Pi^\vee$). Assume that $p \nmid 6N_{E/\Q}(\n)h_E(\n)D_E$. Then:
$$
[ \d^\pm_\e(\phi_f),\d^\mp_\e(\phi_f') ] \sim\Lambda^{imp}(\Pi,\Ad,1)
$$
\end{prop}

\subsection{Congruence numbers and perfect pairings}
\label{congruence_numbers}

\subsubsection{Congruence modules} Let $\TT$ be local $\O$-algebra which is finite, flat and reduced. Let $\l: \TT \to \O$ be an augmentation, i.e. a surjective map of $\O$-algebras. We can therefore decompose $\TT_\K \cong \K \x \SS_\K$ such that $\l_\K:= \l \otimes_\O \K$ corresponds to the projection on $\K$. We denote $e_\l$ the idempotent of $\TT_\K$ corresponding to $(1,0)$ through the above isomorphism.

Now, let $M$ be a $\TT$-module, which is finite flat over $\O$. We define:
$$
M^\l = e_\l \cdot M \quad \mbox{ and } \quad M_\l = e_\l \cdot M_\K \cap M
$$
where $M_\K:= M \otimes_\O \K$. Note that $M_\lambda = M[\lambda] := \{ m \in M \mid t \cdot m = \lambda(t) m, \, \forall t \in \TT \}$, so we will sometimes switch between these two notations in the following. Moreover, we define the \textit{$\l$-rank} of $M$ to be $\mathrm{rank}_\l(M) := \mathrm{dim}_\K(M[\lambda] \otimes_\O \K)$. One checks that $M_\lambda \subset M^\lambda$. We then define the \textit{congruence module} of $\lambda$ on $M$ to be:
$$
C_\lambda(M) := M^\lambda /M_\lambda.
$$
Its Fitting ideal $\eta_\lambda(M) := \mathrm{Fitt}_\O(C_\lambda(M))$ is called the \textit{congruence number} of $\lambda$ on $M$. When $M = \TT$, $C_\lambda(M)$ and $\eta_\lambda(M)$ are denoted respectively by $C_\lambda$ and $\eta_\lambda$, and are simply called the {congruence module} and {congruence number} of $\lambda$. In that case $C_\lambda \simeq \O/\eta_\l$ is in fact a ring. More generally, for a $\TT$-module $M$ of $\l$-rank $1$, one has that $C_\lambda(M) \simeq \O/\eta_\l(M)$ and that $\eta_\l \subset \eta_\l(M)$.

\subsubsection{Hecke modules with a semi-linear involution}
\label{hmod_involution}

We retain the notation from the previous paragraph. Suppose now that $\TT$ is endowed with an involution $\i : \TT \to\TT$ of $\O$-algebra. If $T\in \TT$ we will write $T^\i$ for $\i(T)$. Let $M$ be a $\TT$-module as above. We will say that $M$ is endowed with a semi-linear involution if there exists a non-trivial $\O$-linear involution, also denoted $\i$, which is semi-linear with respect to the $\TT$-action, i.e. such that $\forall T\in \TT$:
$$
\i \circ T = T^\i \circ \i 
$$
in $\End_\O(M)$. Let $\l: \TT \to \O$ be a $\O$-algebra morphism. Suppose that $\l$ is invariant with respect to the $\i$-action on $\TT$, i.e. that $\l(T^\i) = \l(T)$, for all $T\in \TT$. In this case, the $\i$-action on $M$ commutes with $e_\l$. Consequently, the congruence module:
$$
C_\l(M):= M^{\lambda}/M_{\lambda}
$$
inherits an involution $\O$-linear action of $\i$. Let $C_\l(M)[\pm]$ denote the eigenspace associated to the eigenvalue $\pm 1$ of $\i$. Then, let $\eta_\l(M)[\pm]$ be the Fitting ideal of the $\O$-module $C_\l(M)[\pm]$. Since $C_\l(M) = C_\l(M)[+] \oplus C_\l(M)[-]$, we have that:
$$
\eta_\l(M) = \eta_\l(M)[+]\x \eta_\l(M)[-]
$$

Suppose now that the $\l$-rank of $M$ is $2$ and that the action of $\i$ on $M[\lambda]$ is not trivial, so that $M_\l[\pm]$ and $M^\l[\pm]$ are $\O$-modules of rank 1. In this case, one checks that:
$$
C_\l(M)[\pm] = \O/\eta_\l(M)[\pm] \quad \mbox{ et } \quad \eta_\l(M)[\pm] \,\, | \,\,\eta_\l
$$

We will need the following lemma, which is a variant of \cite[Proposition 2.3]{TU22} adapted to the situation of Hecke modules with a semi-linear involution:
\begin{lemma}
\label{pairing_lemma_2}
Let $M$ and $N$ be two $\TT$-modules, finite flat over $\O$, endowed with a semi-linear involution $\i$, and let:
$$
\langle\cdot,\cdot\rangle: M \x N \to\O
$$
be a perfect pairing which is $\TT$-equivariant and $\i$-anti-equivariant. Then, if $\l:\TT \to \O$ is $\i$-invariant, we have that $\eta_\l(M)[\pm] = \eta_\l(N)[\mp]$. Moreover, assume that $M$ an and $N$ are of $\l$-rank $2$ and that the $\i$-action on $M_\l$ and $N_\l$ is non-trivial. Then, if $m_\pm$ and $n_\mp$ are respective $\O$-basis of $M_\l[\pm]$ and $N_\l[\mp]$, one has:
$$
\langle m_\pm,n_\mp\rangle = \eta_\l(M)[\pm]
$$
\end{lemma}

The proof is easily adapted from the proof of \cite[Proposition 2.3]{TU22} (see \cite[Lemma 3.6]{thesis}).

\subsection{The adjoint $L$-value formula}
Let $p$ be an odd prime number. Let $\K$ be some sufficiently large $p$-adic field, $\O$ its valuation ring, and $\wp$ its prime ideal. Let $h(K_f;\O)$ be the spherical Hecke algebra of level $K_f$ acting faithfully on the cohomology. We recall that $\Pi$ is a conjugate self-dual cohomological automorphic cuspidal representation of $\GL(\A_E)$ of mirahoric level $\n$ and cohomological weight $\mu_E = (\mu,\mu^\vee) \in X^+(T_E)$. Let $K_f = K_1(\n)$ be the mirahoric subgroup of level $\n$. Let $\m_\Pi$ denote the maximal ideal of $h(K_f;\O)$ corresponding to $\Pi$ and let $\TT:= h(K_f;\O)_{\m_{\Pi}}/(\O-\mathrm{tors})$. Let $\l_\Pi: \TT \to \O$ be the Hecke eigensystem associated with $\Pi$, and let $\eta_{\l_\Pi}(M)$ be the congruence number of $\l_\Pi$ on the $\TT$-module:
$$
M = H_{cusp}^5(Y_E(K_f), \L_{\mu_E}( \O))_{\m_\Pi}
$$
When $\Pi$ is $\s$-invariant, $M$ is endowed with a semi-linear action of the Galois involution $\s$, and with a semi-linear action of the conjugation-duality involution $\e$ when $\Pi$ is $\e$-invariant. We prove the following theorem:

\begin{theorem}
\label{adjoint_L_value}
Let $\Pi$ be conjugate self-dual cohomological automorphic cuspidal representation of $\GL(\A_E)$, which is only ramified above split primes. Assume that the Galois representation associated with $\Pi$ is residually absolutely irreducible. Assume that the cohomological weight $\mu_E \in X^+(T_E)$ of $\Pi$ is $p$-small, and that $p$ doesn't divide $6N_{E/\Q}(\mathfrak{n}) h_E(\mathfrak{n})D_E$. Then:
$$
\eta_{\l_\Pi}(M)[\pm] \quad \sim \quad \frac{\Lambda^{imp}(\Pi,\Ad,1)}{\Om_5(\Pi,\e,\pm)\cdot \Om_5(\Pi^\vee,\e,\mp)}
$$
where $\eta_{\l_\Pi}(M)[\pm]$ is the $\pm$-part for the action of $\e$.
\end{theorem}

\begin{proof}[Proof of \ref{adjoint_L_value}] To prove the equalities of  \ref{adjoint_L_value}, we now apply the formalism of congruence numbers introduced in the last paragraphs to the localized Hecke algebra $\TT = h(K_f; \O)_{\m_\Pi}/(\O-\mathrm{tors})$, the $\TT$-modules $M = H^5_{cusp}(Y(K_f), \L_{\mu}(\O))_{\m_\Pi}$ and $N = H^5_{cusp}(Y(K_f), \L_{\mu}(\O)^\vee)_{\m_\Pi}$, and the perfect pairing $[\cdot,\cdot]: M \x N \to \O$ defined in (\ref{localized_poincare_pairing}). We recall that the action of $t \in \TT$ on $N$ is given by $[\, \e \,] \circ t \circ [\, \e \,]^{-1}$ so that the pairing $[\cdot,\cdot]$ is equivariant for the action of $\TT$ (see the discussion just after the \ref{poincare_equivariance}). Note that since the action of $h(K_f ; \O)$ on $H^5_{cusp}(Y(K_f), \L_{\mu}(\O)^\vee)_{\m_\Pi}$ is twisted by $\vee$, the $\Pi^\vee$-isotypic part is then localized at $\m_{\Pi}$. In fact, the $\l_\Pi$-isotypic part of $N$ for this action is the $\l_{\Pi^\vee}$-istoypic part for the standard  action of the spherical Hecke algebra on $H^5_{cusp}(Y(K_f),\L_{\mu^\vee}(\C))$:
$$
M_{\l_\Pi} = H^5(Y(K_f), \L_{\mu}(\O))_{\m_\Pi}[\Pi_f] \quad \mbox{ et } \quad N_{\l_\Pi} = H^5(Y(K_f), \L_{\mu}(\O)^\vee)_{\m_\Pi}[\Pi_f^\vee].
$$

Moreover, by the definition of the $\e$-periods, $\d^\pm_\e(W_{\phi_f})/\Om_5(\Pi, \e, \pm)$ and $\d^\pm_\e(W_{\phi'_f})/\Om_5(\Pi^\vee,\e, \pm)$ are respective basis of $M_{\l_\Pi}[\pm]$ and $N_{\l_\Pi}[\pm]$. Then \ref{pairing_lemma_2} yields: 
$$
\eta_{\l_\Pi}(M)[\pm] \sim \frac{[\d_\pm(W_{\phi_f}),\d_\mp(W_{\phi'_f}) ]}{\Om_5(\Pi, \e, \pm)\cdot \Om_5(\Pi^\vee, \e, \mp)}
$$
The theorem then follows from \ref{pairing_computation}.
\end{proof}

\subsubsection{Relation with top and bottom degree periods}
\label{middle_vs_extremal}
As explained in the introduction, a cohomological cuspidal automorphic representation $\Pi$ of $\GLn(\A_E)$ (for $E$ a general number field) is associated with some top and bottom degree periods. In particular, when $n=3$ an $E$ is a real quadratic field, one can define two $p$-integrally normalized periods attached to $\Pi$, the top degree period $\Om_6(\Pi) \in \C^\x / \O^\x$ and the bottom degree period $\Om_4(\Pi) \in \C^\x / \O^\x$, defined respectively within the top and bottom degrees of the cuspidal cohomology of $\GL(E)$. We refer to \cite[\S3.2.5]{BR17} for a precise definition. Then, combining \ref{adjoint_L_value} with Balasubramanyam-Raghuram's formula \cite[Theorem A]{BR17}, we can deduce an integral relation between the middle-degree periods and the extremal-degree periods of $\Pi$, under Calegari-Geraghty setting.

More precisely, let $\tilde{\TT}_E$ denote the \textit{full} Hecke algebra of $\GL(E)$ acting on the cohomology, localized at the maximal ideal $\m_\Pi$. By \textit{full}, we mean here that $\tilde{\TT}_E$ is defined by acting on the full cohomology $\tilde{H}_E$ of $Y_E(K_f)$ with coefficients in $\O$, not only on its $\O$-torsion-free part, as in \S\ref{hecke_corr}. In particular $\tilde{\TT}_E$ may contain $\O$-torsion, and $\TT_E = \tilde{\TT}_E /(\O-tors)$ is simply the torsion-free quotient of $\tilde{\TT}_E$. We now assume that the following three conjectures hold :
\begin{itemize}
\item $\mathrm{(Gal_{\m_\Pi})}$ The Galois representation $\rho_{\m_\Pi}$ with coefficients in $\tilde{\TT}_E$ attached to $\Pi$ exists
\item $\mathrm{(LGC_{\m_\Pi})}$ The Galois representation $\rho_{\m_\Pi}$ satisfies local-global compatibilities at minimal, Fontaine-Laffaille, and Taylor-Wiles places 
\item $\mathrm{(Van_{\m_\Pi})}$ The residual cohomology groups $H^q(Y_E(K_f),\L_{\mu}(\FF))_{\m_\Pi}$ vanishes unless $q \in [4,6]$
\end{itemize}
Moreover, we will say that $\rho_{\m_\Pi}$ satisfy $\mathrm{(CG)}$ if it satsifies the following three conditions:
\begin{itemize}
\item $\rho_{\m_\Pi}$ is $\n$-minimal;
\item $p-3 > 2m$, where $\mu = (m,m,0) \in X^+(T_3)$ (this ensures that $\rho_{\m_\Pi}$ is Fontaine-Laffaille above $p$);
\item the residual representation $\overline{\rho}_{\m_\Pi}$ has enormous image.
\end{itemize}

Under these conditions, and conditionally on the three conjectures $\mathrm{(Gal_{\m_\Pi})}$, $\mathrm{(LGC_{\m_\Pi})}$ and $\mathrm{(Van_{\m_\Pi})}$, the Calegari-Geraghty theory \cite{CG18} implies that the full cohomology is free over the full Hecke algebra $\tilde{\TT}_E$. We refer to the author's thesis \cite[\S3.2.4-5]{thesis} for a more detailed presentation of Calegari-Geraghty theory.  Be careful that $H^\bullet$ and $\TT$ denote the full cohomology and Hecke algebra there, while their torsion-free parts are denoted by $\bar{H}^\bullet$ and $\bar{\TT}$. Then, under Calegari-Geraghty setting, we obtain the following corollary:

\begin{corollaire} 
\label{relation_middle_top_bottom}
Assume $p >2$. Let $\Pi$ be a cohomological cuspidal automorphic representation of $\GL(\A_E)$ which is conjugate self-dual, and is only ramified above split primes. Assume that conjectures $\mathrm{(Gal_{\m_\Pi})}$, $\mathrm{(LGC_{\m_\Pi})}$ and $\mathrm{(Van_{\m_\Pi})}$ hold, and that $\rho_{\m_\Pi}$ satisfies $\mathrm{(CG)}$. Then:
$$
\Om_5(\Pi,\e,\pm) \cdot \Om_5(\Pi^\vee,\e,\mp) \sim \Om_4(\Pi) \cdot \Om_6(\Pi^\vee) \cdot  \nu_{\Pi}^\pm
$$
where $\nu_{\Pi}^\pm := \eta_{\Pi} \cdot \eta_{\Pi}(M)[\pm]^{-1} \in \O$.
\end{corollaire}

It is worth noting that there is no need to exclude specific primes (except in $\mathrm{(CG)}$ for primes which are small with respect to $\mu$), as the constants in the two adjoint $L$-value formulas cancel each other out. \\

\section{A divisibility of automorphic periods for the real quadratic base change}
\label{part_CBC}

After recalling the relevant properties of the base change for $\GL$, we prove \ref{thmB} (see \ref{CBC_divisibility}). We then deduce the divisibility of \ref{main_div_BC} (see \ref{period_divisibility}) between the periods of a cuspidal automorphic representation of $\GL(\Q)$ and the periods of its base change to $\GL(E)$. In paragraph \S\ref{non_self_dual}, we present the results for non self-dual representations. \\

Let $\psi_F := \psi_\Q \circ \mathrm{Tr}_{F/\Q} : F \bs \A_F \to \C^\x$ and $\psi_E : =  E \bs \A_E \to \C^\x$ defined by:
$$
\psi_E(x) = \psi_F(\mathrm{Tr}_{E/F}(x)) =   \psi_F(x+ \s(x)), \quad x \in \A_E
$$
with $\s$ being the non-trivial element in $\mathrm{Gal}(E/F)$. We write $\psi_E = \bigotimes_w \psi_{E,w}$ for its factorization as a tensor product of local additive characters.

\subsection{Quadratic base change}
\label{CBC}

We temporarily adopt general notation. Let $n\geq 1$ be any integer and $E/F$ a general quadratic extension of number fields. Let $\pi$ be a cuspidal automorphic representation of $\GLn(\A_F)$. In this subsection, we recall some facts about the base change automorphic lifting from $\GLn(\A_F)$ to $\GLn(\A_E)$.

\subsubsection{Local base change}
For a place $v$ of $F$, we note $K= F_v$. Suppose first that $v$ is split in $E$, so that $E_v = K\x K$. Then the local base change of $\pi_v$ is defined to be:
$$
\mathrm{BC}(\pi_v) = \pi_v \otimes \pi_v
$$

Suppose now that $v$ is non-split, so that $L := E_v$ is a field. Let $\L_K$ (resp. $\L_L$) be the Weil-Deligne group of $K$ (resp. $L$). Let $\phi_\pi : \L_K \to \GLn(\C)$ be the $n$-dimensional representation of $\L_K$ associated to $\pi$ by the local Langlands correspondence for $\GLn(K)$. Then the base change $\mathrm{BC}(\pi_v)$ of $\pi_v$ is defined to be the unique irreducible admissible representation of $\GLn(L)$ corresponding to the restriction $\phi_\pi |_{\L_L}$ through the local Langlands correspondence for $\GLn(L)$.

When $\pi_v$ is unramified, the local base change has a simple expression in term of the Satake parameters. Let $ S_{\pi_v} = \mathrm{diag}(\a_1,\dots,\a_n) \in \GLn(\C)$ be (a diagonal representative of) the Satake parameters of $\pi_v$. Then the Satake parameter of $\mathrm{BC}(\pi_v)$ is simply given by (the conjugacy class of):
$$
S_{\pi_v}^{f_v} =\mathrm{diag}(\a_1^{f_v},\dots,\a_n^{f_v}) \in \GLn(\C)
$$
where $f_v$ is the residual degree of $E_w/F_v$.

\subsubsection{Global base change.}

We say that an automorphic representation $\Pi$ of $\GLn(\A_E)$ is a (strong) base change of $\pi$ if for all places $v$ of $F$, we have that:
$$
\Pi_v \simeq \mathrm{BC}(\pi_v)
$$
The following theorem is due to Arthur and Clozel (see \cite[Theorem 4.2 and Theorem 5.1]{SABC}):

\begin{theorem}[Existence of cuspidal base change]
Let $\pi$ be a cuspidal automorphic representation of $\GLn(\A_F)$. If $\pi \not\simeq \pi \otimes \chi_{E/F}$, then $\pi$ admits a unique strong base change $\Pi$. Moreover, $\Pi$ is $\s$-invariant and cuspidal.
\end{theorem}

\subsubsection{Factorization of adjoint $L$-functions}
\label{twisted_adjoint_L-func}
Let  $\chi_{E/F}$ be the quadratic character of $\A_F^\x/F^\x$ associated with the quadratic extension $E/F$ by class field theory. The \textit{primitive} twisted adjoint $L$-function $L(\pi,\Ad \otimes \chi_{E/F},s)$ of $\pi$ is defined by the following formula:
$$
L(\pi\otimes\pi^\vee \otimes \chi_{E/F},s) = L(\chi,s) \cdot L(\pi,\Ad \otimes \chi_{E/F},s)
$$
where $L(\chi_{E/F},\cdot)$ is the Dirichlet $L$-function of $\chi_{E/F}$, and $L(\pi\otimes\pi^\vee \otimes \chi_{E/F},s)$ is the \textit{primitive} twisted Rankin-Selberg $L$-function of $\pi$, which can be defined as in \S\ref{L-functions}. Assume that $\pi \not\simeq \pi \otimes \chi_{E/F}$, so that $\mathrm{BC}(\pi)$ is cuspidal. We then have the following factorization of primitive adjoint $L$-functions:
$$
L(\mathrm{BC}(\pi),\Ad,s) = L(\pi,\Ad,s) \cdot L(\pi,\Ad \otimes \chi_{E/F},s)
$$
which is also valid for the completed $\Lambda$-functions, which are obtained by adding the $\Gamma$-factors at archimedean places. \\

One can define the \textit{imprimitive} twisted Rankin-Selberg $L$-function $L^{imp}(\pi\otimes\pi^\vee \otimes \chi_{E/F},s)$ of $\pi$ analogously to the imprimitive Rankin-Selberg $L$-function of $\pi$ in \S\ref{L-functions}. Then, the \textit{imprimitive} twisted adjoint $L$-function $L^{imp}(\pi,\Ad \otimes \chi_{E/F},s)$ of $\pi$ is defined by the following formula:
$$
L^{imp}(\pi\otimes\pi^\vee \otimes \chi_{E/F},s) = L(\chi_{E/F},s) \cdot L^{imp}(\pi,\Ad \otimes \chi_{E/F},s)
$$
By definition, $L^{imp}(\pi,\Ad \otimes \chi_{E/F},s)$ is a meromorphic function on the right half plane $Re(s) > 1-\e$, for some small $\e>0$. Moreover, it is holomorphic and nonzero at $s=1$. It only differs from $L(\pi,\Ad \otimes \chi_{E/F},s)$ by a finite product of factors at places where $\pi$ is ramified. \\

Assume that $\pi$ is only ramified at places of $F$ which are split in $E$, and that $\pi \not\simeq \pi \otimes \chi_{E/F}$, so that $\mathrm{BC}(\pi)$ is cuspidal. Then, as for the primitive $L$-functions, one has the following factorization of adjoint imprimitive $L$-functions:
\begin{equation}
\label{decomposition_ad}
L^{imp}(\mathrm{BC}(\pi),\Ad,s) = L^{imp}(\pi,\Ad,s) \cdot L^{imp}(\pi,\Ad \otimes \chi_{E/F},s)
\end{equation}
which is also valid for the completed imprimitive $\Lambda^{imp}$-functions.\\

\subsubsection{Base change for Hecke algebras}
\label{hecke_CBC}

In this paragraph, we explain how the base change is described from the perspective of Hecke algebras. More precisely we construct a morphism $\theta: \TT_{E} \to \TT_{F}$ between the localized cohomological Hecke algebras, which corresponds to the base change transfer. Under some conditions, we prove that this morphism is surjective.  \\

Let $\pi$ be a cuspidal automorphic representation of $\GLn(\A_F)$ such that $\pi \not\simeq \pi \otimes \chi_{E/F}$ and let $\Pi$ be its strong base change to $E$, which is a cuspidal automorphic representation of $\GLn(\A_E)$. Let $k_f$ and $K_f$ be the mirahoric level of $\pi$ and $\Pi$. Let $\H_{F} = \H_{F}(k_f;\O)$ and $\H_{E} = \H_{E}(K_f;\O)$ be the abstract spherical Hecke algebras associated with $k_f$ and $K_f$, with coefficients in $\O$, defined in \S\ref{hecke_corr}. Let:
$$
\theta : \H_{E} \to \H_{F}
$$
be the morphism defined by the following formulas ($v$ is a finite place of $F$ where $\pi$ is unramified and $w$ is a place of $E$ dividing $v$):
\begin{itemize}
\item If $v$ is split or ramified in $E$ : $\theta(T_{w,i}) :=T_{v,i}$;
\item If $v$ is inert in $E$:
$$
\theta(T_{w,i}) :=  \sum_{k=\max(0,2i-n)}^{\min(2i,n)} (-q_v)^{(i-k)^2} T_{v,k} T_{v,2i-k}.
$$
where $q_v = \#(\O_v/ \wp_v)$ and $T_{v,0} = 1$. \\
\end{itemize}
This morphism describes the base change automorphic transfer at the level of Hecke algebras, as we have the following lemma:
\begin{lemma}
\label{hecke_base_change}
Let $\rho$ be a cuspidal  automorphic representation of $\GLn(\A_F)$ of level $k_f$ and let $\Si = \mathrm{BC}(\rho)$. Let $\theta_\C : \H_E(K_f,\C) \to \H_F(k_f,\C)$ denote the scalar extension of $\theta$ to $\C$. Then for all $T \in \H_E(K_f,\C)$, we have:
$$
\l_\Si(T) = \l_\rho(\theta_\C(T))
$$
where $\l_\Si : \H_E(K_f,\C) \to \C$ and $\l_\rho : \H_F(k_f,\C) \to \C$ are the Hecke eigensystems associated with $\Si$ and $\rho$.
\end{lemma}

\begin{proof}
Let $S_\Si$ be the set of finite places where $\Si$ is ramified, $S_p(E)$ the set of places of $E$ dividing $p$ and $S_E = S_\Si \cup S_p(E)$. Since $\H_{E}(K_f;\C)$ is generated over $\C$ by the $T_{w,i}$, for $w \notin S_E$ and $0\leq i \leq n$, we only need to prove the formula for these Hecke operators. Let $w \notin S_E$ and let $v$ be the place of $F$ below $w$. Let   $0\leq i \leq n$. Let $\a = \diag(\a_1,\dots,\a_n)$ be a diagonal representative of the Satake parameter of $\rho_v$. We recall that:
$$
\l_\rho(T_{v,i}) = q_v^{i(n-i)/2} \s_i(\a_1,\dots,\a_n)
$$
where $q_v = \# \O_v / (\varpi_v)$ and $\s_i(X_1,\dots,X_n)$ is the $i$-th symmetric polynomial (see \cite[Proposition 7.2 p.53]{LALF}). The Satake parameters of $\Si_w$ is given by the conjugacy class of:
$$
\b:= \diag(\b_1,\dots,\b_n) = \diag(\a_1^{f_v},\dots,\a_n^{f_v})
$$
When $v$ is split or ramified, $f_v =1$ and $q_w = q_v$. Thus:
$$
\l_\rho(T_{w,i}) = q_w^{i(n-i)/2} \s_i(\b_1,\dots,\b_n) = q_v^{i(n-i)/2} \s_i(\a_1,\dots,\a_n) = \l_\rho(T_{v,i}) = \l_\rho(\theta(T_{w,i}))
$$
When $v$ is inert, $f_v =2$ and $q_w = q_v^2$. Thus:
$$
\l_\rho(T_{w,i}) = q_w^{i(n-i)/2} \s_i(\a_1^2,\dots,\a_n^2)
$$
In this case, the lemma then follows from the definition of $\theta$ at inert places and from the following formula:
$$
(-1)^i \s_i(X_1^2,\dots,X_n^2) =  \sum_{k=\max(0,2i-n)}^{\min(2i,n)} (-1)^k \s_k(X_1,\dots,X_n)\s_{2i-k}(X_1,\dots,X_n)
$$
which can be proven using the identity $\prod_{i=1}^n(X^2-X_i^2) = \prod_{i=1}^n(X-X_i)(X+X_i)$.
\end{proof}

From now we assume that $\pi$, and thus $\Pi$, are cohomological. Let $\mu_F \in X^+(T_F)$ and $\mu_E = (\mu_F,\mu_F) \in X^+(T_E)$ be the cohomological weights of $\pi$ and $\Pi$.  Let  $h_{F}$ and $h_{E}$ be the cohomological Hecke algebras acting faithfully on the cuspidal cohomology of $Y_F(k_f)$ and $Y_E(K_f)$ with coefficients in $\L_{\mu_F}(\O)$ and $\L_{\mu_E}(\O)$, defined in \S\ref{hecke_corr}. The morphism $\theta$ induces a morphism between these cohomological Hecke algebras. To see this, we consider the following two decompositions of the cuspidal cohomology:
$$
H^\bullet_{cusp}(Y_E(K_f), \L_{\mu_E}(\C)) = \bigoplus_{\Si \in \mathrm{Coh}(G_E,\mu_E,K_f)} H^\bullet(\g, K_\inf  ; \Si_{\inf} \otimes L_{\mu_E}(\C)) \otimes \Si_f^{K_f}
$$
and:
$$
H^\bullet_{cusp}(Y_F(k_f), \L_{\mu_F}(\C)) =  \bigoplus_{\rho \in \mathrm{Coh}(G_{F},\mu_F,k_f)}  H^\bullet(\g_n, K_n ; \rho_{\inf} \otimes L_{\mu_F}(\C)) \otimes \rho_f^{k_f}
$$
These two decompositions are equivariant with respect to the action of $\H_{E}$ and $\H_{F}$, acting on the left-hand sides by Hecke correspondences and on the right-hand sides via Hecke operators. Thus, the Hecke action on the cuspidal cohomology groups is completely determined by the Hecke action on finite parts of cuspidal automorphic representations. Thus, if $T$ and $T'$ are two operators in $\H_E$ defining the same correspondence on $H^\bullet_{cusp}(Y_E(K_f), \L_{\mu_E}(\C))$, it follows from \ref{hecke_base_change} that $\theta(T)$ and $\theta(T')$ define the same correspondence on $H^\bullet_{cusp}(Y_F(k_f), \L_{\mu_F}(\C))$. Since $h_F$ acts faithfully on the latter, this shows that $\theta$ induces a morphism on the cohomological Hecke algebras, also denoted $\theta$:
$$
\theta : h_{E} \to h_{F}
$$

Let $\l_\pi: h_{F} \to \O$ be the Hecke eigensystem associated with $\pi$. We denote by $\overline{\lambda_\pi}$ its reduction modulo $\varpi$ and by $\m_\pi$ the kernel of $\overline{\lambda_\pi}$. It is a maximal ideal of $h_{F}$, and we denote by $\TT_{F}$ the localization of $h_F$ at $\m_\pi$. Similarly, let $
\l_{\Pi}: h_{E} \to \O$ be the eigensystem associated with $\Pi =\mathrm{BC}(\pi)$. Thus, from \ref{hecke_base_change}, we have $\lambda_{\Pi} = \lambda_\pi \circ \theta$. Let $\m_{\Pi} = \Ker \overline{\lambda_{\Pi}}$, and let $\TT_{E}$ be the localization of $h_E$ at $\m_{\Pi}$. Since $t \in \m_{\Pi} \Leftrightarrow \theta(t) \in \m_\pi$, $\theta$ induces, by localization, a morphism of $\O$-algebras, which we still denote $\theta$:
$$
\theta : \TT_{E} \to \TT_{F}
$$
We then have the following lemma:
\begin{lemma}
\label{BC_surjectivity}
Assume that $F=\Q$. Assume that $p$ is unramified in $E$, that $\pi$ is spherical at primes which are ramified in $E$, and that the Galois representation $\rho_{\pi}$ associated with $\pi$ is residually absolutely irreducible. Then the morphism $\theta :  \TT_{E} \to \TT_{F}$ is surjective.
\end{lemma}

\begin{proof} This proof was communicated to us by J. Tilouine, although any potential mistakes or inaccuracies in the following proof are the fault of the author. Let $S_\Q = S_\pi \cup \{p\}$, where $S_\pi$ the set of finite places where $\pi$ is ramified, and let $q \notin S_\Q$ be a prime split in $E$. It follows from the defining formula of $\theta$ that (the image in $\TT_\Q$ of) $T_{q,i}$ is in $\theta(\TT_E)$, which is a closed complete local sub-$\O$-algebra of $\TT_\Q$. One knows that there exists a Galois representation $\rho: \Gal(\overline{\Q}/\Q) \to \GLn(\TT_\Q)$, with coefficients in $\TT_\Q$, such that if $q \notin S_\Q$ is a finite place of $\Q$, then the characteristic polynomial of $\rho(\Frob_{q})$ is:
$$
\sum_{i=0}^n (-1)^i q^{i(i-1)/2} T_{q,i} X^{n-i}
$$
This representation can be obtained from the Scholze's Galois representation \cite[Corollary 5.4.4]{Scholze15}. Alternatively, one can construct this representation following the section 2.2 of \cite{Carayol} by putting together the Galois representations $\rho_\pi$ attached to each $\pi \in \mathrm{Coh}(G_\Q,\mu,k_f)$ by \cite{HLTT16}. By a theorem of Carayol, it follows from the existence of $\rho$ and the hypothesis that it is residually absolutely irreducible, that $\TT_\Q$ is generated by the traces, i.e. by the $\Tr(\rho(g))$, for $g \in \Gal(\overline{\Q}/\Q)$. Since $S_\Q$ does not contain any prime ramified in $E$, the subfield $H$ of $\overline{\Q}$ fixed by ${\Ker \rho}$ is linearly disjoint from $E$. Thus, by Chebotarev's density theorem, the set of $\Frob_q$, for $q\notin S_\Q$ split in $E$, is dense in $\Gal(H/{\Q})$, hence its image is dense in the image of $\rho$. Since $\TT_\Q$ is generated by the traces, and since $\Tr(\rho(\Frob_q)) = -T_{q,1} \in \theta(\TT_E)$ when $q \notin S_\Q$ is split in $E$, this proves that $\theta(\TT_E) = \TT_\Q$ and thus that $\theta$ is surjective.
\end{proof}

\subsection{Jacquet-Ye period and Jacquet's conjecture}

We still assume that $E/F$ is a quadratic extension of number fields.

\subsubsection{Quasi-split groups and Jacquet-Ye periods}

 Let $X$ be the symmetric space of Hermitian matrices of rank $n$ with respect to $E/F$, whose $A$-points are given for any $F$-algebra $A$, by:
$$
X(A) = \{ x \in  \GLn(E\otimes_F A), \,{}^t x^\s = x \}
$$

For every $x \in X(F)$, we define the unitary group determined by $x$ to be the $F$-algebraic group $U_x$ whose $A$-points are given, for any $F$-algebra $A$, by:
$$
U_x(A) =\{g \in \GLn(E\otimes_F A), {}^t g^\s \cdot x \cdot g = x \}
$$
Let $Z_U$ be the center of $U_x$ whose $A$-points are given by:
$$
Z_U(A) =\{z \in (E\otimes_\Q A)^\x, N_{E/F}(z):=zz^{\s} = 1 \}
$$
Let $x \in X(F)$ and let $\Pi$ be a cuspidal automorphic representation of $\GLn(\A_E)$. We say that $\Pi$ is distinguished by $U_x$ (or $U_x$-distinguished) if the so-called \textit{Jacquet-Ye period} defined by: 
$$
\P_{U_x}(\phi) = \int_{U_x(F) \bs U_x(\A_F)} \phi(h)dh
$$
is non-zero for some $\phi \in \Pi$.

\subsubsection{Jacquet's conjecture}When $n=3$ and if every archidean places of $F$ is split in $E$, Jacquet \cite[Theorem 2]{J05a} has shown that $\Pi$ is a quadratic base change from $G_F$ if and only if $\Pi$ is $U_x$-distinguished for some $x \in X(F)$. In fact, Jacquet has also proven \cite[Theorem 1]{J10} that this is equivalent for $\Pi$ to be $U_{E/F}$-distinguished. Feigon, Offen and Lapid \cite[Theorem 0.1]{FLO12} have latter generalized the results of Jacquet to any $n$ and any $x \in X(F)$, and removing the splitting assumption at archimedean places. We record in the following theorem some of their results that we will need later:
\begin{theorem}[Feigon-Lapid-Offen]
\label{Jacquet_conjecture}
Suppose that $n$ is odd. Let $\Pi$ be a cuspidal automorphic representation of $\mathrm{GL}_n(\A_E)$, and let $x \in X(F)$. Then $\Pi$ is a base change from $\mathrm{GL}_n(\A_E)$ if and only if $\Pi$ is $U_x$-distinguished. 
\end{theorem}
When $n$ is even, one has to add a local condition at inert places of $F$ where $U_x$ is not quasi-split. These local conditions do not appear when $n$ is odd since in that case all unitary groups over $p$-adic fields are quasi-split. \\

When $\Pi$ is distinguished by the unitary group $U_1$ determined by the identity matrix, Jacquet \cite[Theorem 1]{J01} has been able to compute (for $n=3$ and $E/F$ split at archimedean places) the global period $\P_{U_1}$  as a product of local periods. His results has also been completely generalized by Feigon, Offen and Lapid \cite[Theorem 10.2]{FLO12}. In order to state their result we need to introduce some notation. Suppose that $\Pi$ is the base change of a cuspidal automorphic representation $\pi$ of $\GLn(\A_F)$. Let $S_F$ be a finite set of places of $F$ containing the archimedean places, the even places and the places ramified in $E$, as well as the places where $\pi$ is ramified. Let $\W(\Pi,\psi_E)$ be the Whittaker model of $\Pi$ with respect to $\psi_E$. It decomposes as a restricted tensor product $\bigotimes_w \W(\Pi_w,\psi_{E,w})$, where $\psi_E = \otimes_w \psi_{E,w}$. Then (\cite[Theorem 10.2, Corollary 7.2]{FLO12}):

\begin{theorem}
\label{Jacquet_decomposition}
Let $\pi$ be a cuspidal automorphic representation of $\GLn(\A_F)$ and $\Pi$ its base change to $\GLn(\A_E)$. Let $x \in X(F)$. Then, for all place $v$ of  $F$, there exists a non-zero linear form $\P_{x,v} :\W(\Pi_v, \psi_{E,v}) \to \C$ which is $U_{x}(F_v)$-invariant, such that if $\phi= \otimes_w \phi_w \in \Pi$ is a pure tensor, then there exists some finite set $S_\phi \supset S_F$ such that:
$$
\P_{U_{x}}(\phi)= 2 \cdot \w_{BC} \cdot L^{S_\phi}(\chi_{E/F},1) \cdot L^{S_\phi}(\pi,\Ad \otimes \chi_{E/F},1) \times \prod_{v \in S_\phi}{\P_{x,v}(W_{\phi_v})}
$$
where $\w_{BC} \in \C^\x$ is some constant depending on the choice of the Haar measures. Moreover, suppose that $v$ is split in $E$ and write $x =(h,{}^th)$, with $h \in \GLn(F)$. Then, the local period $\P_{x,v}$ is explicitly given by:
$$
\P_{x,v}(W_1 \otimes W_2) =  \int_{N_{n-1}(F_v)\bs \mathrm{GL}_{n-1}(F_v)} W_1\left( \left(\begin{array}{ll}
g & 0\\
0 & 1
\end{array}\right) h \right)W_2 \left(w \left(\begin{array}{cc}
{}^tg^{-1} & 0 \\
0 & 1
\end{array}\right) \right) dg
$$
\end{theorem}

In the following we will only consider the unitary group $U_1$ associated to the identity. Thus, to simplify the notation, we write $U$ for $U_1$ and $\P_v$ for $\P_{1,v}$. \\

\subsection{A cohomological interpretation of the Feigon-Offen-Lapid formula} 
\label{proof_CBC}
From now, $E$ is a real quadratic field. Let $\pi$ be a self-dual cuspidal automorphic representation of $\GL(\A_\Q)$, which is cohomological of weight $\mu = (m,m,0) \in X^+(T_3)$. Assume that $\pi$ is only ramified at primes that splits in $E$. Let $\Pi$ be its base change to $\GL(\A_E)$. Then $\Pi$ is a cohomological cuspidal automorphic representation of $\GL(\A_E)$, of cohomological weight $\mu_E = (\mu,\mu)$, which is self-conjugate : $\Pi^\s \simeq \Pi$. Moreover, since $\pi$ is self-dual, $\Pi$ is conjugate self-dual .Let $S_\Pi$ be the set of finite places where $\Pi$ is ramified. We choose a particular subset $R$ of $S_\Pi$ such that $S_\Pi = R \sqcup \s(R)$, where $\s$ is the Galois involution of $E$. Let $K_f$ denote the mixed mirahoric subgroup $K_1^*(\mathfrak{n})$ of level $\mathfrak{n}$ and type $R$, where $\n := \mathfrak{n}(\Pi)$ is the mirahoric level of $\Pi$. Finally, let $\phi_f$ be the mixed essential vector $\phi_\Pi^*$ of $\Pi$, defined in \S\ref{global_mirahoric_theory}. The Haar measures on $U$ are normalized as in \cite[\S 1.2]{FLO12} except at $v= \inf$ and $v \mid N$ (in both cases $v$ is split in $E$), where the local measure $dg_v$ on $U_E(F_v) = \GL(F_v)$ are normalized as specified in \S\ref{Haar_measure_GLn}.  \\

\subsubsection{Some linear forms}
\label{linear_forms_CBC}

We first define some linear forms on the different cohomological objects under consideration. Let:
$$
Y_U(K_f) = U(\Q) \bs U(\A)/K_{f,U}K_3
$$
be the adelic variety of level $K_{f,U} = K_f \cap U(\A_f)$ associated with $U$. It is a $5$-dimensionnal subvariety of $Y_E(K_f)$. The restriction to $U(\O)$ of the representation $\rho$ of $G_E(\O)$ on $L_{\mu_E}(\O)$ decomposes as a tensor product:
$$
L_{\mu_E}(\O)|_{U} = L_{\mu}(\O)|_{U} \otimes L_{\mu}(\O)|_{U}^\vee
$$
Here $L_{\mu}(\O)|_U$ and $L_{\mu}(\O)|_{U}^\vee$ are the representations of $U(\O)$ whose underlying space is both $L_{\mu}(\O)$ and on which $g \in U(\O)$ acts respectively by $\rho_{\mu}(\tau(g))$ and $\rho_{\mu}({}^t\tau(g)^{-1})$. Suppose that $\mu_E$ is $p$-small, i.e. that $m<p$ where $\mu = (m,m,v)$. Then we can consider the $U(\O)$-equivariant map onto the trivial representation defined by:
$$
L_{\mu_E}(\O)|_{U} \to \O, \quad P \otimes Q \mapsto \langle P , \vee(Q) \rangle_{\mu}
$$
where $\vee : L_{\mu}(\O) \to L_{\mu^\vee}(\O)$ has been defined in paragraph~\ref{alg_irrep}.  We denote by $\Pd$ the $\O$-linear form obtained by composing the following map:
$$
\begin{aligned}
\Pd: H^5_{cusp}(Y_E(K_f),\L_{\mu_E}(\O)) & \inj H^5_{c}(Y_E(K_f),\L_{\mu_E}(\O)) \\
&  \to H^5_{c}(Y_U(K_f),\L_{\mu_E}(\O)|_{U}) \\
&\to H^5_{c}(Y_U(K_f),\O) \\
&\to \O
\end{aligned}
$$
where the third map is obtained by functoriality from the above $U(\O)$-equivariant map. $\Pd$ is invariant by the action of the conjugation-duality involution $\e$ on the cuspidal cohomology described in paragraph~\ref{e_involution}.

\paragraph{A linear form on the $(\g,K_\inf)$-cohomology.} We now define a linear form:
$$
\mathfrak{P} : H^5(\g,K_\inf; \W(\Pi_\inf,\psi_\inf) \otimes L_{\mu_E}(\C)) \to \C
$$
which is the counterpart on the $(\g,K_\inf)$-cohomology of the above linear form. It is constructed as a tensor product of the three following linear forms:

\begin{itemize}
\item The linear form $\mathfrak{P}_{\mu_E}: L_{\mu_E}(\C) \to \C$ defined by $\mathfrak{P}_{\mu_E}(P_\tau \otimes P_{\s\tau}) = \langle P_\tau, \vee(P_{\s\tau})\rangle_{\mu}$ (recall $\mu_E =(\mu, \mu)$). The pairing $\langle \cdot,\cdot \rangle_{\mu} : L_{\mu}(\C) \x L_{\mu^\vee}(\C) \to \C$ and the map $\vee : L_{\mu}(\C) \to L_{\mu^\vee}(\C)$ have been defined in paragraph~\ref{alg_irrep}.
\item The linear form $\mbox{\textcalligra{p}} : \bigwedge^5 \mathfrak{p}_{\inf,\C}^* \to \C$ obtained by functoriality of the exterior product from the map:
$$
X_\tau + Y_{\s\tau} \in \p_{\C}^*  \mapsto X_\tau - Y_{\s\tau} \in \p_{3,\C}^*,
$$ 
where $\p_{\C}^* = \p_{\tau,\C}^* \oplus \p_{\s\tau,\C}^* = \p_{3,\C}^* \oplus \p_{3,\C}^*$, and by identifying $\bigwedge^5\p_{3,\C}^*$ with $\C$.
\item The linear form $\P_\inf: \W(\Pi_\inf,\psi_\inf) \to \C$ defined in \ref{Jacquet_decomposition} by: 
$$
\P_\inf\left(W_{\tau} \otimes W_{\s\tau}\right)=\int_{N_{2}(\R) \backslash \mathrm{GL}_2(\R)} W_{\tau}\left(\begin{array}{ll}
g & \\
& 1
\end{array}\right) W_{\s\tau}\left( w \left(\begin{array}{ll}
{}^tg^{-1} & \\
& 1
\end{array}\right)\right) dg  .
$$
where $\inf$ denotes the archimedean place of $\Q$.
\end{itemize}

Thus, the linear form $\mathfrak{P}$ is simply defined on pure tensors by:
$$
\mathfrak{P}(W \otimes P \otimes \w) = \P_\inf(W) \x \mathfrak{P}_{\mu}(P) \x \mbox{\textcalligra{p}}(\w)
$$
It is obviously invariant for the action of $\e$ on the $(\g,K_\inf)$-cohomology. \\

\subsubsection{A cohomological interpretation of the Feigon-Offen-Lapid formula}
\label{coho_CBC}

In this paragraph only, $\Pi$ is \textit{any} cohomological cuspidal representation in $\mathrm{Coh}(G_E,\mu_E,K_f)$. Write $\w_\Pi =\w_{\Pi_f} \otimes \w_{\Pi_\inf}$ for its central character. Suppose that $\w_{\Pi_f}$ is trivial on $Z_U(\A_f)$. We give a cohomological interpretation of the Jacquet-Ye period, using the Eichler-Shimura map:
$$
\d_\e^+:\Pi_f^{K_f}\to H_{cusp}^5(Y_E(K_f),\L_{\mu_E}(\C))[\Pi_f;\e =+]
$$
defined in paragraph~\S\ref{e_periods}. Recall that this map is associated with the choice of an explicit element $[\Pi_\inf]_\e^+ \in H^5(\g,K_\inf ; \W(\Pi_\inf,\psi_\inf) \otimes L_{\mu_E}(\C))$. Let $\phi_f \in \Pi_f$ fixed by $K_f$. As in \cite[\S4.6.2]{TR_SBC}, one can see that:
$$
\Pd(\d_\e^+(\phi_f)) = 0
$$
if $\w_{\Pi_\inf}(-1,-1) = -1$. Suppose that $\w_{\Pi_\inf}(-1,-1) = 1$. This is the case in particular when $\Pi$ is a base change from $G_\Q$. Since $\w_{\Pi_\inf}|_{Z_U(\R)^\circ}$ is trivial (it is determined by the cohomological weight $\mu_E$ which is self-conjugate), thus the central character $\w_{\Pi}$ of $\Pi$ is trivial on $Z_U(\A)$. Then one has that:
$$
\P_U(\phi) = \mathrm{vol}(Z_U(\Q) \bs Z_U(\A)) \int_{Z_U(\A)U(\Q) \bs U(\A)} \phi(h)dh
$$
for all $\phi \in \Pi$. Recall that $Z_U(\Q) \bs Z_U(\A) \subset E \bs \A_E^1$ is compact with:
\begin{equation}
\label{vol_ZU}
\vol(Z_U(\Q) \bs Z_U(\A)) = 2 \cdot \frac{\vol(E \bs \A_E^1)}{\vol(\Q \bs \A_\Q^1)} = 2 \cdot  \mathrm{Res}_{s=1} \zeta_E(s)
\end{equation}
for our choices of Haar measures. Write $[\Pi_\inf]_\e^+ = \sum_{i\in I} \w_i \otimes W_i \otimes P_i$. Then, similarly to \cite[paragraph 3.3.3]{BR17}, one has that:
$$
\Pd(\d_\e^+(\phi_f)) = \frac{4h(K_{f,U})}{\mathrm{vol}(Z_U(\Q) \bs Z_U(\A))} \cdot \sum_{i \in I} \,\mbox{\textcalligra{p}}(\w_i) \times  \mathfrak{P}_{\mu_E}(P_i) \times \P_U(\phi_i)
$$
where $\phi_i \in \Pi$ is the form $\phi_f \otimes \phi_{i,\inf}$, with $\phi_{i,\inf} \in \Pi_\inf$ corresponding the Whittaker function $ W_i \in \W(\Pi_\inf,\psi_\inf)$, and:
$$
h(K_{f,U}) := \vol(Z_U(\Q) \bs Z_U(\A_f)/Z_U(\A_f) \cap K_{f,U})
$$
Recall that $K_f = K^*_1(\n)$ is the mixed mirahoric of level $\n$. Thus $K_{f,U}$ is the mirahoric subgroupe $K_1(N)$ of level $N = \mathfrak{n} \cap \Z$ in $U(\A_f)$ (see \cite[\S3.2.2]{TR_SBC} for its definition). We write $h_U(N)$ for $h(K_{f,U})$. Note that $h_U(N)$ divides $h_E(\n)$. The above cohomological interpretation formula, together with (\ref{vol_ZU}), \ref{Jacquet_conjecture} and \ref{Jacquet_decomposition}, has for direct consequence the following proposition:
\begin{prop}
\label{lf_vanishing_CBC}
Let $\Pi \in \mathrm{Coh}(G_E,\mu_E,K_f)$ such that $\w_{\Pi_f}$ is trivial on $Z_U(\A_f)$.
\begin{itemize}
\item If $\Pi$ is not a base change from $G$, then $\Pd$ vanishes on $H^5_{cusp}(Y_E(K_f),\L_{\mu_E}(\C))[\Pi_f]$
\item If $\Pi$ is the base change of some cuspidal automorphic representation $\pi$ of $\GL(\A)$, and if $\Pi$ is $\e$-dual, then for all $K_f$-fixed form $\phi_f \in \Pi_f$ which is a pure tensor:
\begin{equation}
\label{coho_jacquet_decomposition}
\Pd(\d_\e^+(\phi_f))= 4 \cdot {\w}_{BC} \cdot h_U(N) \cdot L^{imp}(\pi, \Ad \otimes \chi_E,1) \x \prod_{v \in S_{\phi_f}} \P_v^\natural(W_{\phi_f,v}) \x \mathfrak{P}([\Pi_\inf]_\e^+)
\end{equation}
where $S_{\phi_f}$ is some sufficiently large finite set of non-archimedean places of $\Q$ which contains the even places and the places which ramifies in $E$, as well as the places where $\pi$ is ramified, and: 
$$
\P_v^\natural(W_v) :=  L^{imp}(\pi_v \otimes \pi_v^\vee \otimes \chi_E,1)^{-1} \cdot \P_v(W_v) \\
$$
for $W_v \in \W(\Pi_v,\psi_{E,v}) \to \C$. The constant $\w_{BC}$ is the defined in \ref{Jacquet_decomposition}.
\end{itemize}
\end{prop}

\subsection{Local computations}
Let $\phi_f = \phi_\Pi^*\in \Pi_f$ be the mixed mirahoric vector for the mixed mirahoric subgroup $K_f = K_1^*(\n)$. Then the set $S_{\phi_f}$ in (\ref{coho_jacquet_decomposition}) can be chosen to be $S_\Q$, the finite set of places of $\Q$ containing the archimedean place, the even place, the places which are ramified in $E$, as well as the places where $\pi$ is ramified. We then compute the local factors in (\ref{coho_jacquet_decomposition}) for this specific choice of newform $\phi_f$.

\subsubsection{Local factors at ramification places}
\label{ram_CBC}

Let $v \in S_\pi$ which is split in $E$, and let $w \in S$ and $w' = w^\s$ the two places of $E$ above $v$. We now compute the local factor $\P_v(W_{\phi_f,v})$ for the newform $\phi_f = \phi_\Pi^*$ of $\Pi_f$. Since $\P_v(W_1 \otimes W_2) = \langle W_1,W^\vee_2 \rangle_v$, it follows from the definition of $\phi_\Pi^*$ and from \ref{split_ramified_computations} that:
$$
\P_v(W_{\phi_f,v}) = \langle W_{\Pi_w}^\circ,(( W_{\Pi_w^\vee}^\circ)^\vee)^\vee \rangle_v = \langle W_{\Pi_w}^\circ,W_{\Pi_w^\vee}^\circ \rangle_v = L^{imp}(\Pi_w \otimes \Pi_w^\vee,1) = L^{imp}(\pi_v \otimes \pi_v^\vee \otimes \chi_E,1)
$$

\subsubsection{Archimedean computations}
\label{archimedean_CBC}

In this paragraph, we compute the archimedean factor of (\ref{coho_jacquet_decomposition}). Recall that the cohomological type of $\pi$ is $\mu=(m,m,0) \in X^+(T_3)$ and that the minimal $\SO$-type of $\pi$ is given by:
$$
\ell = 2m +3
$$
The archimedean part of $\Pi$ is of the form $\Pi_\inf = \pi_\inf \otimes \pi_\inf$. In this paragraph we compute the value:
$$
\mathfrak{P}([\Pi_\inf]^{+}_\e) = \mathfrak{P}([\Pi_\inf]_{\{\tau\}})
$$
where $[\Pi_\inf]^{+}_\e = [\Pi_\inf]_{\{\tau\}} + \e([\Pi_\inf]_{\{\tau\}})$ and $[\Pi_\inf]_{\{\tau\}} = [\pi_\inf]_2 \otimes [\pi_\inf]_3$ is the Chen's generator defined in paragraph~\S\ref{chen_generators}. To compute the above archimedean value, we imitate the proof of \cite[Lemma 5.3]{Che22}. In fact, by carefully looking at the calculation there, one sees that it only uses the $\SO$-equivariance property of the involved pairings and is actually true for any $\SO$-equivariant pairings. Consequently, since the three linear forms $\P_\inf$, $\mathfrak{P}_{\mu}$ and $\mbox{\textcalligra{p}}$ are $\SO$-equivariant, we get:
$$
\mathfrak{P}([\Pi_\inf]_{\pm}) = (-1)^{v} 4 \frac{(2\ell+1)!}{(\ell!)^2} \x \mbox{\textcalligra{p}}(\Om_2,\Om_3^-) \x \mathfrak{P}_{\mu}(Q_{\mu} \otimes Q_{\mu}^-) \x  \P_\inf(W_{\pi_\inf,0} \otimes W_{\pi_\inf,0}),
$$
where $\Om_2 = X_{-1}^*\wedge X_{-2}^*$ and $\Om_3^- = X_{0}^*\wedge X_{1}^*\wedge X_{2}^*$,
$$
Q_{\mu} = \rho_{\mu} \left(\left(\begin{array}{ccc}1 & 0 & 1 \\ \sqrt{-1} & 0 & -\sqrt{-1} \\ 0 & 1 & 0\end{array}\right)\right) (P_{\mu}^{+}) \quad \mbox{ and } \quad Q_{\mu}^- = \rho_{\mu} \left(\left(\begin{array}{ccc}-1 & 0 & -1 \\ \sqrt{-1} & 0 & -\sqrt{-1} \\ 0 & 1 & 0\end{array}\right)\right) (P_{\mu}^{+}),
$$
and $W_{\pi_\inf,0} \in W(\pi_\inf,\psi_\inf)$ is the image of $\mathbf{v}_0$ throught the Miyazaki's embedding (see \cite[\S3.4.3]{thesis} for details). One has that $\mbox{\textcalligra{p}}(\Om_2,\Om_3^-)= 4 \sqrt{-1}$. Moreover, one can check that:
$$
\mathfrak{P}_{\mu}(Q_{\mu} \otimes Q_{\mu}^-) = \left\langle \rho_{\mu}\left(\left(\begin{array}{ccc}-2 & 0 & 0 \\ 0 & 1 & 0 \\ 0 & 0 & -2 \end{array}\right)\right)(P_{\mu}^+),\i(P_{\mu}^+)\right\rangle = 4^v (-2)^{m^+-m^-} = 1
$$
where $(m^+,m^-,v)= \mu = (m,m,0)$.
To compute the last factor of the above formula, we recall from \cite[Lemma 3.3]{thesis} that:
$$
{W^\vee_{\pi_\inf,j}} = - W_{\pi^\vee_\inf,j}
$$
for all $-\ell \leq j \leq \ell$. Hence, we are reduced to \cite[Lemma 5.4]{Che22}:
$$
\P_\inf(W_{\pi_\inf,0} \otimes W_{\pi_\inf,0}) = - \langle W_{\pi_\inf,0}, W_{\pi_\inf^\vee,0} \rangle_\inf = - \frac{2^{\ell+4}\pi}{2\ell +2} \binom{2\ell+1}{\ell}^{-1} \cdot \Gamma(\pi_\inf \x \pi_\inf^\vee,1)
$$
Consequently:
$$
\mathfrak{P}([\Pi_\inf]_{\pm}) = - (-1)^{v+1/2}  2^{\ell + 7}\pi \cdot \Gamma(\pi_\inf \x \pi_\inf^\vee,1) 
$$ \\

Finaly, similarly than in the proof of \ref{jacquet-shalika_formula}, one can see that if $p$ does not divide $6N$, then $\w_{BC} \sim \pi^{-1}$ for our choices of Haar measures. Thus, gathering all the above local computations, we get:

\begin{prop}
\label{linear_formula}
Let $\pi$ and $\Pi = \mathrm{BC}(\pi)$ be as above. Let $\phi_f \in \Pi_f$ be the mixed essential vector of $\Pi$. Assume that $p \nmid 6N_{E/\Q}(\n)h_E(\n)D_E$. Then:
$$
\Pd(\d_\e^+(\phi_f)) \sim \frac{\Lambda^{imp}(\Pi, \Ad \otimes \chi_E,1)}{u}
$$
with $u := \P_2^\natural(W_{\phi_f,2})^{-1} \x \prod_{v\in S_{ram}} \P_v^\natural(W_{\phi_f,v})^{-1} \in \C^\x$. \\
\end{prop} 

In the above proposition, $S_{ram}$ is the set of rational primes which are ramified in $E$. As explained in the introduction, the factor $u \in \C^\x$ is expected to be $1$ but is hard to compute. In the following paragraph, we prove that at least this factor is an algebraic number (our result is a little bit more precise, see \ref{algebraicity_u2ram}). \\

\subsection{Algebraicity of uncomputed local factors}
\label{paragraph_algebraicity}

In this paragraph, we prove that the uncomputed factor $u \in \C^\x$ of \ref{linear_formula} is an algebraic number. Our method is inspired by \cite[\S 2]{Che25}. The results presented in this paragraph hold for any extension $E/F$ of number fields and for any integer $n \geq 2$.\\

Let $v$ be a finite place of $F$ and let $K=F_v$ and $L=E_v$. In order to simplify the notation, we first adopt local notation unless otherwise stated, and write $\pi = \pi_v$, $\Pi = \Pi_v$ and $\P_x = \P_{x,v}$, as well as $\psi_K = \psi_{F,v}$ and $\psi_L = \psi_{E,v} := \psi_K \circ \mathrm{Tr}_{L/K}$. Finally, we write $G_K=\GLn(K)$ and $G_L = \GLn(L)$. \\

Let $\vs \in \mathrm{Aut}(\C)$ be an automorphism of $\C$, and let ${}^\vs \pi$ and ${}^\vs \Pi$ be the $\vs$-conjugate of the representations $\pi$ and $\Pi$ (see for instance the proof of \cite[Lemma 3.2]{TR_SBC} for a precise definition of these representations). Then one has that ${}^\vs \Pi = \mathrm{BC}({}^\vs \pi)$. Let $\Q^{ab}$ be the maximal abelian extension of $\Q$ in $\C$. Suppose $\vs \in \mathrm{Aut}(\C/\Q^{ab})$, i.e. that $\vs$ fixes $\Q^{ab}$. Since $\psi_K$ takes its values in $\Q^{ab}$, we have that $\vs(\psi_K(k)) = \psi_K(k)$ for all $k \in K$. The representation ${}^\vs \pi$ is generic and the map $t_{\vs} : W \mapsto {}^\vs W$ defined by:
$$
{}^\vs W : g \mapsto \vs(W(g))
$$
is a $\GLn(K)$-equivariant isomorphism $\W(\pi,\psi_K) \toeq \W({}^\vs \pi,\psi_K)$ which is $\s$-semi-linear.  The representation ${}^\vs \Pi$ is also generic, and one similarly defines a map $t_{\vs}: \W(\Pi,\psi_L) \toeq \W({}^\vs \Pi,\psi_L)$. \\

Let $ {}^\vs \P_x : \W({}^\vs \Pi, \psi_L) \to \C$ be the linear form defined in \ref{Jacquet_decomposition} applied to ${}^\vs \Pi = \mathrm{BC}({}^\vs \pi)$, for $x \in X(F)$. We then prove the following algebraicity result:

\begin{lemma}
\label{equivariance_Px}
Let $\vs \in \mathrm{Aut}(\C/\Q^{ab})$. Let $W  \in \W(\Pi, \psi_L)$ and ${}^\vs W = t_\vs(W) \in \W({}^\vs \Pi, \psi_L)$. Then:
$$
 {}^\vs \P_x({}^\vs W) =  \vs(\P_x(W))
$$
\end{lemma}

\begin{proof}[Proof of \ref{equivariance_Px}] Let us first recall some notation from \cite[\S2]{FLO12}. Let $\mathcal{D} =(\pi,\hat{\pi},B)$, where $\hat{\pi}$ is some irreducible admissible representation of $G_K$, and $B : \pi \x \hat{\pi} \to \C$ is a non-degenerate $G_K$-invariant bilinear form. We denote by $\Lambda_{\mathcal{D}} : \pi^\vee \to \hat{\pi}$ the identification induced by $B$. For some linear forms $\ell \in \pi^*$ and $\hat{\ell} \in \hat{\pi}^*$, the Bessel distribution (with respect to $\mathcal{D}$, $\ell$ and $\hat{\ell}$) is simply defined by:
$$
\mathfrak{B}_{\mathcal{D}}^{\ell, \hat{\ell}}(f) := \hat{\ell}(\Lambda_{\mathcal{D}}(\ell \circ \pi(f)), \quad f \in \mathcal{S}(G_K),
$$
where $\mathcal{S}(G_K)$ is the space of Schwartz functions on $G_K$. In particular, one can take $\mathcal{D}$ to be:
$$
\mathfrak{W}(\pi) = (\W(\pi,\psi_K), \W(\pi^\vee,\psi_K^{-1}), \langle \cdot,\cdot \rangle),
$$
with $\langle \cdot,\cdot \rangle$ being the Whittaker pairing defined by (\ref{pairing_whittaker}). For any $g \in G_K$ we denote by $\d_g(\pi)$ the linear form on $\W(\pi,\psi_K)$ given by the evaluation at $g$. Let $w := \mathrm{antidiag}(1,\dots,1)$. Let $\Pi = \mathrm{BC}(\pi)$ and let $x \in X(F)$. The linear form $\P_x : \W(\Pi,\psi_L) \to \C$ appearing in \ref{Jacquet_decomposition} is characterized by the following local Bessel identity:
\begin{equation}
\label{LHS}
\mathfrak{B}_{\mathfrak{W}(\Pi)}^{\P_{x}, \d_e(\Pi^\vee)}(f_L) = \mathfrak{B}_{\mathfrak{W}(\pi)}^{\d_{w}(\pi),\d_{e}(\pi^\vee)}(f_K)
\end{equation}
for every Schwartz functions $f_L \in \mathcal{S}(G_L)$ and $f_K \in \mathcal{S}(G_K)$ having $x$-matching orbital (see \cite[\S 3.3]{FLO12}). Let $\vs \in \mathrm{Aut}(\C/\Q^{ab})$ and let apply $\vs$ to the above Bessel identity. Write:
$$
\mathfrak{W}({}^\vs \pi) = (\W({}^\vs\pi,\psi_K), \W({}^\vs\pi^\vee,\psi_K^{-1}), \langle \cdot,\cdot \rangle_\vs)
$$
where the pairing $\langle \cdot,\cdot \rangle_\vs : \W({}^\vs \pi,\psi_K) \x \W({}^\vs \pi^\vee,\psi_K^{-1}) \to \C$ is defined by formula (\ref{pairing_whittaker}). On the right hand side, one easily checks that:
$$
\vs\left( \mathfrak{B}_{\mathfrak{W}(\pi)}^{\d_{w}(\pi),\d_{e}(\pi^\vee)}(f_K) \right) =  \mathfrak{B}_{\mathfrak{W}({}^\vs \pi)}^{\d_{w}({}^\vs \pi),\d_{e}({}^\vs \pi^\vee)}(f_K)
$$
for any Schwartz function $f_K \in \mathcal{S}(G_K)$. The equality follows from \cite[Lemma 2.1(5)]{FLO12} and from the following formula:
$$
\langle {}^\vs W,{}^\vs W' \rangle_\vs = \vs(\langle W, W' \rangle)
$$
for all $W \in \W(\pi,\psi_K)$ and $W' \in \W(\pi^\vee,\psi_K^{-1})$. Similarly, on the left hand side, one can check that:
$$
\vs\left( \mathfrak{B}_{\mathfrak{W}(\Pi)}^{\P_x,\d_{e}(\Pi^\vee)}(f_L) \right) =  \mathfrak{B}_{\mathfrak{W}({}^\vs \Pi)}^{\vs \circ \P_x \circ t_\vs^{-1},\d_{e}({}^\vs \Pi^\vee)}(f_L)
$$
for any Schwartz function $f_L \in \mathcal{S}(G_L)$. Then, the lemma follows from the fact that the linear form ${}^\vs \P_x : \W({}^\vs \Pi, \psi_L) \to \C$ is characterized by the following local Bessel identity:
$$
\mathfrak{B}_{\mathfrak{W}({}^\vs \Pi)}^{{}^\vs \P_{x}, \d_e({}^\vs \Pi^\vee)}(f_L) = \mathfrak{B}_{\mathfrak{W}({}^\vs \pi)}^{\d_{w}({}^\vs \pi),\d_{e}({}^\vs \pi^\vee)}(f_K)
$$
for $f_L \in \mathcal{S}(G_L)$ and $f_K \in \mathcal{S}(G_K)$ having $x$-matching orbital. 

\end{proof}

We now go back to global notation, i.e. $\pi$ and $\Pi = \mathrm{BC}(\pi)$ now denote the cohomological cuspidal automorphic representations of \ref{linear_formula}.  Let $\Q(\Pi_f) := \C^{S(\Pi_f)}$ with $S(\Pi_f) = \{ \vs \in \mathrm{Aut}(\C), {}^\vs \Pi_f = \Pi_f \}$ be the rationality field of $\Pi_f$. Since $\Pi$ is a cohomological cuspidal automorphic representation, we know from \cite{Clozel90} that $\Q(\Pi_f) $ is a number field. Let us write $u(\Pi_f)$ for the constant $u$ of \ref{CBC_divisibility} defined in \S\ref{final_proof_CBC}, in order to emphasize its dependence on the finite part $\Pi_f$ of $\Pi$. As a corollary of the previous lemma, we get that $u(\Pi_f)$ is an algebraic number: 
\begin{corollaire}
\label{algebraicity_u2ram}
For every $\vs \in \mathrm{Aut}(\C/\Q^{ab})$, one has that $\vs(u(\Pi_f)) = u({}^\vs \Pi_f)$. Thus, it follows that:
$$
u(\Pi_f) \in \Q(\Pi_f) \cdot \Q^{ab}
$$

\end{corollaire}

To conclude, by following the same method and through a more careful analysis, it should even be possible  to prove that the constant $u$ belongs to the rationality field $\Q(\Pi_f)$ of $\Pi_f$. We did not have the opportunity to pursue this direction. \\

\subsection{Relative congruences numbers and linear forms}
\label{relative_congruence_numbers}

\subsubsection{Relative congruence modules}
\label{sss_transfer}

Let $\TT'$ and $\TT$ be two local, finite and flat $\O$-algebras such that $\TT'_\K$ and $\TT_\K$ are semisimple $\K$-algebras. Suppose that there exists a surjective $\O$-algebras morphism $\theta: \TT' \to \TT$. After tensoring by $\K$, we obtain a decomposition of $\TT'_\K$:
$$
 \TT'_\K \simeq \TT_\K \times \TT_\K^\#
$$
as a product of $\TT_\K$ and some semisimple $\K$-algebra $\TT_\K^\#$, such that $\theta_\K :  \TT'_\K \to \T_\K$ corresponds to the first projection : $\mathbf{pr}_1 \circ \i_\K$. We denote $e_\theta$ (resp. $e_\#$) the idempotent of $\TT'_\K$ corresponding to $(1,0)$ (resp. to $(0,1)$) through the above isomorphism. As before, we consider an $\O$-algebra morphism $\l: \TT \to \O$ and denote by $\l' = \l \circ \theta$ its transfer to $\TT'$.

Let $M$ be a $\TT'$-module which is finite flat over $\O$, and let $\eta_{\l'}(M)$ be its congruence ideal. We define $M_{\TT} := M \cap e_\theta \cdot M_\K$. It is endowed with a $\TT$-module structure: an element $t \in \TT$ acts on $m\in M$ by $t \cdot m:= t' \cdot m$, where $t' \in \TT'$ is such that $\theta(t') = t$ (such an element always exists by the surjectivity of $\theta$). This action does not depend on the choice of $t'$. In fact, let $t_1',t_2' \in \TT'$ such that $\theta(t'_1) = \theta(t'_2) = t$, that is $t'_1 e_\theta = t'_2e_\theta$. Since $m$ belongs to $e_\theta \cdot M_\K$, we can write $m = e_\theta \cdot m_\K$ with $m_\K \in M_\K$ and therefore $t_1' \cdot m = t_2' \cdot m$. 

We define the \textit{relative congruence module} of $\l$ on $M$ by: 
$$
C_\l^\#(M) = M^{\l'} / (M_\TT)^\l
$$
Its Fitting ideal, denoted by $\eta_\l^\#(M)$, is called the \textit{relative congruence number} of $\l$ on $M$. By definition, we have that $(M_\TT)_\l = M_{\l'}$ and we obtain the following exact sequence:
$$
0 \to \frac{(M_\TT)^{\l}}{(M_\TT)_{\l}} \to \frac{M^{\l'}}{M_{\l'}} \to \frac{M^{\l'}}{(M_\TT)^\l} \to 0
$$

By multiplicativity of the Fitting ideals we obtain the following fondamental relation between the congruence module associated to $M$ and the congruence module associated to its pushforward $M_\TT$: 
\begin{equation}
\label{rel_mult}
\eta_{\l'}(M) = \eta_\l(M_\TT) \cdot \eta_\l^\#(M)
\end{equation}

\paragraph{The case $M = \TT'$.} In the special case where $M$ is $\TT'$ we simply call $\eta_\l^\#(M)$ the \textbf{relative congruence number} of $\l$, and denote it by $\eta_\l^\#$. Moreover, in this case $M_\TT^\l$ is a $\O$-module of rank $1$ and so $\eta_\l(M_\TT)$ divides $\eta_\l$. Hence (\ref{rel_mult}) gives the following divisibility:
$$
\eta_{\l'}\,\,  | \,\,  \eta_\l \cdot \eta_\l^\#
$$

\subsubsection{Relative congruence modules with an involution} We now switch to the relative setting and consider some surjective $\O$-algebra morphism $\theta: \TT' \to \TT$. Let $\l: \TT \to \O$ a Hecke eigensystem, and let $\l':= \l \circ \theta$ be its transfer to $\TT'$. Moreover we suppose that $\TT'$ is given with an involution denoted $\i$, and that $\l'$ is $\i$-invariant. Let $M$ be a $\TT'$-module of rank $2$ such that $M$ is given an action of $\i$ which is semi-linear. Then the action of $\i$ on $M$ commutes with $e_{\lambda'}$, and the three congruences modules:
$$
C_{\l'}(M):= M^{\lambda'}/M_{\lambda'}, \quad C_\l(M_\TT):= (M_\TT)^{\lambda}/(M_\TT)_{\lambda} \quad \mbox{ and } \quad C_{\l}^{\#}(M):= M^{\lambda'}/(M_\TT)^{\lambda}
$$
where $M_\TT = M \cap e_\theta \cdot M_\K$, are endowed with an action of $\i$ induced by that on $M$. We then assume that the action of $\i$ is non-trivial on $M_{\l'}$, so that it is not trivial on the three congruence modules either. Then we denote by $\eta_{\l'}(M)[\pm]$, $\eta_{\l}(M_\TT)[\pm]$ and $\eta_{\l}^{\#}(M)[\pm]$ the Fitting ideals of respectively $C_\l(M)[\pm]$, $C_\l(M_\TT)[\pm]$ and $C_{\l}^{\#}(M)[\pm]$. We have the following lemma:
\begin{lemma}
\label{cn_decomposition}
 Let $M$ be a $\TT'$-module, finite flat over $\O$, given with a semi-linear involution $\i$. Then:
$$
\eta_{\l'}(M)[\pm] = \eta_{\l}(M_\TT)[\pm] \cdot \eta_{\l}^{\#}(M)[\pm]
$$
\end{lemma}

\subsubsection{An useful lemma} Moreover we have the following lemma: 

\begin{lemma}
\label{lf_lemma}
Let $M$ be $\TT'$-module, finite flat over $\O$, given with a semi-linear involution $\i$. Assume that the $\l$-rank of $M$ is $2$. Let $\Ld \in \mathrm{Hom}_\O(M,\O)[\pm]$ be a $\i$-$\pm$-invariant linear form (i.e $\Ld(\i^{-1}(m)) = \pm \Ld(m)$ for all $m \in M$) such that  $M_\K^\# \subset \mathrm{Ker}(\Ld_\K)$ where $M_\K^\#:= e_\# M_\K$. Then for every $\d$ in $M_{\lambda'}[\pm]$ we have:

$$
\Ld(\d) \in \eta^{\#}_{\lambda'}(M^*)[\pm]
$$
\end{lemma}

The proof is easily adapted to the case of congruence modules with an involution from the proof of \cite[Proposition 2.9]{TU22} (see \cite[Lemma 3.8]{thesis}). \\

\subsection{The main divisibility}
\label{final_proof_CBC} 

Let $\pi$ be a cohomological automorphic cuspidal representation of $\GL(\A_\Q)$ of cohomological weight $\mu = (m,m,0) \in X^+(T_3)$, which is self-dual. Let $E$ be a real quadratic field and let $\Pi= \mathrm{BC}(\pi)$ be its strong base change to $\GL(\A_E)$. Assume that $\pi \not\simeq \pi \otimes \chi_E$ so that $\Pi$ is a cuspidal automorphic representation, and assume that $\pi$ is only ramified at primes which are split in  $E$. $\Pi$ is self-conjugate, and since $\pi$ is self-dual, $\Pi$ is also conjugate self-dual. Thus $\Pi$ is associated with two middle-degree $\e$-periods $\Om_5(\Pi,\e,\pm)$, as explained in \S\ref{e_periods}. Let $\mu_E = (\mu,\mu) = (\mu,\mu^\vee)$ be the cohomological weight of $\Pi$. Let $\mathfrak{n} =\mathfrak{n}(\Pi)$ be the mirahoric level of $\Pi$, and let $K_f$ denote the mixed mirahoric subgroup $K_f = K_1^*(\mathfrak{n})$ of level $\mathfrak{n}$.\\

Let $\l_\Pi : \TT_E \to \O$ be the Hecke-eigensystem associated with $\Pi$. Let $M$ denote the cuspidal cohomology of $G_E$ localized at the maximal ideal $\m_\Pi$ of $h(K_f;\O)$ corresponding to $\Pi$: 
$$
M = H_{cusp}^5(Y_E(K_f), \L_{\mu_E}(\O))_{\m_\Pi}
$$
It is a $\TT_E$-module. Since $\Pi$ is conjugate self-dual, $M$ is equipped with a semi-linear action of the involution $\e$. Let $\eta_{\l_\Pi}^\#(M^*)$ be the relative congruence numbers of $\Pi$ on the $\TT_E$-module $M^* := \Hom_\O(M,\O)$, for the base change map $\theta_{{BC}}: \TT_E \to \TT_\Q$ described in paragraph~\ref{hecke_CBC}. The main result of this section is the following theorem, establishing a divisibility between this congruence number and the imprimitive completed twisted adjoint $L$-function of $\pi$ (see \S\ref{twisted_adjoint_L-func}), normalized by the $\e$-periods associated to $\Pi$:

\begin{theorem}
\label{CBC_divisibility}
Let $\pi$ and $\Pi = \mathrm{BC}(\pi)$ be as above. Assume that the Galois representation associated with $\Pi$ is residually absolutely irreducible. Assume that the cohomological weight $\mu$ of $\pi$ is $p$-small and that $p$ doesn't divide $6N_{E/\Q}(\n)h_E(\n)D_E$. Then, there exists a nonzero algebraic number $u \in \bar{\Q}^\x$, depending only on the local components of $\Pi$ above $2$ and ramified primes in $E$, such that:
$$
\eta_{\l_\Pi}^\#(M^*)[+] \quad | \quad \frac{\Lambda^{imp}(\pi,\Ad \otimes \chi_{E},1)}{u \cdot \Om_5(\Pi,\e,+)}
$$
where $\eta_{\l_\Pi}^\#(M^*)[+]$ is the $+$-part for the action of $\e$. The constant $u$ is defined in \ref{linear_formula}.
\end{theorem}

\begin{proof}[Proof of \ref{CBC_divisibility}] We prove \ref{CBC_divisibility} using the congruence number formalism introduced in the last paragraphs. We consider the $\TT_E$-module $M = H^5_{cusp}(Y_E(K_f),\L(\mu_E ;\O))_{\m_\Pi}$, and the linear form $\Pd: M \to \O$ obtained by restricting to $M$ the linear form $\Pd$ defined in paragraph~\ref{linear_forms_CBC}. Since $\Pi$ is conjugate self-conjugate, we have that $t \in \m_\Pi \iff t^\e \in \m_\Pi$, and the action of $\e$ on $h(\mathrm{GL}_{n/E},\O)$ induces an action on $\TT_E:= h(\mathrm{GL}_{n/E},\O)_{\m_\Pi}$, and thus on $\TT_E$. $M$ also inherits from an action of $\e$ and this action is semi-linear, in the sense of \S\ref{hmod_involution}. We have seen that $\Pd$ is naturally $\e$-invariant. \ref{lf_vanishing_CBC} implies that $\Pd_\K$ vanishes on $e_\#M_\K$. Here $e_\# := 1 - e_{\theta}$, for the idempotent $e_{\theta}$ of $\TT_E$ associated with the base change transfer $\theta_{{BC}}: \TT_{E} \to \TT_{\Q}$ defined in \S\ref{hecke_CBC} (see \S \ref{sss_transfer} for a precise definition of $e_{\theta}$). Then \ref{lf_lemma} implies:
$$
\eta_{\l_\Pi}^\#(M^*)[+] \quad | \quad \Pd \left(\frac{\d^+_\e(\phi_f)}{\Om_5(\Pi,\e,+)}\right)
$$
because by definition of the $\e$-periods, the cohomology class $\d^+_\e(W_{\phi_f})/\Om_5(\Pi,\e,+)$ is an $\O$-base of $M_{\l_\Pi}[+]$. Since $p$ does not divide $N_{E/\Q}(\n) h_E(\n) D_E$, we know from (\ref{coho_jacquet_decomposition}), and the above ramified and archimedean computations, that:
$$
\Pd(\d_\e^+(\phi_f)) \sim \frac{\Lambda^{imp}(\Pi, \Ad \otimes \chi,1)}{u}
$$
where $u$ is the constant of \ref{linear_formula}, which is proven to be an algebraic number in \ref{algebraicity_u2ram}.
\end{proof}

\subsection{A divisibility of automorphic periods}
We keep the same notation as in the last paragraph. Let $N$ be the mirahoric level of $\pi$ and let $k_f = K_1(N) \subset \GL(\A_f)$ be the mirahoric subgroup of level $N$ for $G_\Q$. The cuspidal cohomology of $\GL(\Q)$ is concentrated in degrees $q=2,3$. Then, using the two Eichler-Shimura maps:
$$
\d_q: \W(\pi_f)^{k_f} \to H^{q}_{cusp}(Y_\Q(k_f), \L_{\mu}(\C))[\pi] =: H^{q}(\C)[\pi]
$$
described in paragraph~\ref{eichler-shimura_maps}, one can define two periods $\Om_2(\pi)$ (the bottom degree period) and $\Om_3(\pi)$ (the top degree period) associated with $\pi$. In fact, one knows that $H^{q}(\C)[\pi]$ are of $1$-dimensional $\C$-vector spaces given with an integral $\O$-structure $H^{q}(\O)[\pi]$. If we choose some $\O$-base $\xi_q$ of $H^{q}(\O)[\pi]$, the periods are then defined to be the complex numbers such that:
$$
\d_q(\phi_f) = \Om_q(\pi) \cdot \xi_q
$$
where $\phi_f = \phi_\pi^\circ \in \pi_f^{k_f}$ is the essential vector of $\pi$. Since this definition depends on the choice of $\xi_q$, the periods are only defined up to some $\O$-units. Then we have the following equality, proven up to some non-explicit archimedean factor by Balasubramanyam-Raghuram \cite{BR17} and completed by the archimedean computations of \cite[Theorem 5.5]{Che22}:
\begin{theorem}[Balasubramanyam-Raghuram, Chen]
\label{BR-Chen}
Let $\pi$ be a cohomological cuspidal automorphic representation of $\GL(\A_\Q)$ of mirahoric level $N$. Assume that the Galois representation associated to $\pi$ is residually absolutely irreducible. Assume that the cohomological weight $\mu = (m,m,v)$ of $\pi$ is $p$-small and that $p \nmid 6N\ph(N)$, where $\ph$ is the Euler's phi function. Then we have the following equality: 
$$
\eta_{\pi}(H^3) \sim \frac{\Lambda^{imp}(\pi,\Ad,1)}{\Om_2(\pi)\cdot \Om_3(\pi^\vee)}
$$
\end{theorem}

The formula in the above theorem involves the congruence number $\eta_{\pi}(H^3)$ of $\pi$ on the top-degree cohomology group $H^3 := {H}^3(Y_\Q(k_f), \L_{\mu}(\O))_{\m_\pi}$. It easily follows from the definitions that $\eta_{\pi}(H^3)$ divides the congruence number $\eta_\pi$ of $\pi$. In order to establish a period divisibility, we need this two congruence numbers to coincide. To this end, me make the following assumption: \\

\begin{assumption}[$H = \TT_\Q$]
The module $H^3 := {H}^3(Y_\Q(k_f), \L_{\mu}(\O))_{\m_\pi}$ is free over ${\TT}_\Q$. \\
\end{assumption}

Since cohomological cuspidal automorphic representations appear with multiplicity one in the top-degree cohomology group, the assumption ($H = \TT_\Q$) implies that $H^3$ is free of rank one over ${\TT}_\Q$. Consequently: 
\begin{equation}
\label{CG_equality}
\eta_{\pi}(H^3) \sim \eta_{\pi} 
\end{equation}

We now provide sufficient conditions under which the above assumption holds. Let $\tilde{\TT}_\Q$ be the \textit{full} Hecke algebra acting on the cohomology, localized at $\m_\pi$. By \textit{full}, we mean here that $\tilde{\TT}_\Q$ is defined by acting on the full cohomology $\tilde{H}_\Q$ of $Y_\Q(k_f)$ with coefficients in $\O$ (not only on its $\O$-torsion-free part, as in \S\ref{hecke_corr}). In particular $\tilde{\TT}_\Q$ may contain $\O$-torsion, and $\TT_\Q = \tilde{\TT}_\Q /(\O-tors)$ is simply the torsion-free quotient of $\tilde{\TT}_\Q$. Assume that $p$ is split in $E$. Then, the Galois representation $\rho_{\m_\pi} : \Gal(\overline{\Q}/\Q) \to\GL(\tilde{\TT}_\Q)$ associated with $\m_\pi$ is proven to exist (see \cite[Theorem 6.1.4]{CGH+20}). We then assume that $\rho_{\m_\pi}$ satisfy the following three conditions:
\begin{itemize}
\item $\rho_{\m_\pi}$ is $N$-minimal;
\item $p-3 > 2m$, where $\mu = (m,m,0) \in X^+(T_3)$ (this ensures that $\rho_{\m_\pi}$ is Fontaine-Laffaille at $p$);
\item the residual representation $\overline{\rho}_{\m_\pi}$ has enormous image.
\end{itemize}
Under these conditions, and conditionally on the conjecture that $\rho_{\m_\pi}$ satisfies local-global compatibilities at minimal, Fontaine-Laffaille, and Taylor-Wiles places, the Calegari-Geraghty theory \cite{CG18} implies that the assumption ($H = \TT_\Q$) holds. In fact, the Calegari-Geraghty theory proves the much stronger result that the full cohomology $\tilde{H}_\Q$ is free over the full Hecke algebra $\tilde{\TT}_\Q$. The interested reader is encouraged to consult the author's thesis \cite{thesis} — in particular Paragraphs 3.2.4 and 3.2.5, as well as the discussion below Theorem 7.2 — for a more detailed presentation of Calegari-Geraghty theory. Be careful that $H^\bullet$ and $\TT$ denote the full cohomology and Hecke algebra there, while their torsion-free parts are denoted by $\bar{H}^\bullet$ and $\bar{\TT}$. \\

We now consider $\Pi = \mathrm{BC}(\pi)$ to be the base change of $\pi$ to $G_E$. $\Pi$ is conjugate self-dual, so it can be associated with its $\e$-periods $\Om_5(\Pi,\e,\pm)$. We then prove the following divisibility between the periods of $\pi$ and the $\e$-periods of $\Pi$:

\begin{theorem}
\label{period_divisibility}
Let $\pi$ and $\Pi = \mathrm{BC}(\pi)$ be as above. Assume that the Galois representation associated to $\Pi$ is residually absolutely irreducible. Assume that the cohomological weight $\mu$ of $\pi$ is $p$-small, and that $p$ does not divide $6N_{E/\Q}(\n)h_E(\n) D_E$. Then, under assumption $(H = \TT_\Q)$, the following divisibility holds:
$$
\nu_\pi \cdot \Om_2(\pi) \cdot \Om_3(\pi)\, \mid \, \Om_5(\Pi,\e,-) \cdot u
$$
where $\nu_\pi := \eta_\pi \cdot \eta_\pi(M_{\TT_\Q})[+]^{-1} \in \O$ and $u \in \overline{\Q}^\x$ is the constant (conjecturally equal to $1$) of \ref{CBC_divisibility}.
\end{theorem}

\begin{proof}
We have the following decomposition of (imprimitive) completed $L$-functions:
$$
\Lambda^{imp}(\Pi,\Ad,s) = \Lambda^{imp}(\pi,\Ad,s) \cdot \Lambda^{imp}(\pi,\Ad \x \chi_{E},s) 
$$
From \ref{cn_decomposition}, we have the following relation of congruences numbers:
$$
\eta_{\l_\Pi}(H^5)[\pm] \, = \, \nu_\pi^{-1} \cdot \eta_{\pi} \cdot \eta_{\l_\Pi}^\#(H^5)[\pm]
$$
where $\nu_\pi := \eta_\pi \cdot \eta_\pi(M_{\TT_\Q})[+]^{-1}$ and $M_{\TT_\Q}$ is defined in \ref{sss_transfer}. The announced divisibility then follows from \ref{CBC_divisibility} and \ref{adjoint_L_value} combined with \ref{BR-Chen} and (\ref{CG_equality}).
\end{proof}

\subsection{Results for non self-dual representations}
\label{non_self_dual}
We keep the notation from the previous paragraph but we no longer assume that the representation $\pi$ is self-dual, nor that it is ramified only at split primes. Consequently its base change $\Pi = \mathrm{BC}(\pi)$ is no longer conjugate self-dual and the $\e$-periods of $\Pi$ are not even defined. However, as a base change, $\Pi$ is still is self-conjugate. In this case, following \cite[\S2.4.4]{TR_SBC} we can attached two middle-degree periods:
$$
\Om_5(\Pi,\s,\pm) \in \C^\x /\O^\x
$$
to the cohomological cuspidal automorphic representation $\Pi$. These automorphic periods are defined within the cuspidal cohomology group of middle degree $q=5$ and will be called the $\s$-periods of $\Pi$. \\

However, unlike what happens for the $\e$-periods, the newform $\phi_f \in \Pi_f$ used to define the $\s$-periods of $\Pi$, which is the essential vector $\phi_\Pi^\circ$ of $\Pi$, is not always a good test vector for the Jacquet-Ye integral period. More precisely, it annihilates the Jacquet-Ye period, unless $\Pi$ satisfies the following condition at ramification places: \\

\begin{quote}
For each prime $w$ where $\Pi$ is ramified, the standard $L$-function $L(\Pi_w,s)$ of $\Pi_w$ has degree $2$. \\
\end{quote}

For instance, this condition is satisfied if the mirahoric level $\n$ of $\Pi$ is squarefree and the representation  $\Pi$ at $\n$ is either in the principal series (induced from two unramified characters and one ramified character) or the partial Steinberg associated with the parabolic of type $(2,1)$.  \\

Nethertheless, under this quite restrictive assumption on the ramification of $\Pi$, the newform $\phi_f = \phi_\Pi^\circ \in \Pi_f$ becomes a good test vector for the Jacquet-Ye period, and we prove the following version of \ref{CBC_divisibility} when $\pi$ is non-necessarily self-dual:

\begin{theorem}
\label{CBC_divisibility_NSD}
Let $\pi$ and $\Pi= \mathrm{BC}(\pi)$ be as above. Assume that the Galois representation associated with $\Pi$ is residually absolutely irreducible. Suppose that the cohomological weight $\mu$ of $\pi$ is $p$-small and that $p \nmid 6N_{E/\Q}(\mathfrak{n}) h_E(\mathfrak{n})D_E$. Then, there exists a sign $\d \in \{ \pm\} $, and a nonzero algebraic number $u \in \bar{\Q}^\x$, depending only on the local components of $\Pi$ above $2$ and ramified primes in $E$, such that we have the following divisibility:
$$
\eta_{\Pi}^\#(M^*)[\d] \quad | \quad \frac{\Lambda^{imp}(\pi,\Ad \otimes \chi_{E},1)}{u \cdot \Om_5(\Pi,\s,\d)}
$$
where $\eta_{\Pi}^\#(M^*)[\d]$ is the $\d$-part for the action of $\s$,
\end{theorem}

The proof of this theorem is similar to that of \ref{CBC_divisibility}. Consequently, to avoid overloading the presentation, we only state the results here and refer to \cite[Section 7]{thesis} for a more complete discussion. The ramification condition on $\Pi$ is denoted \textbf{(Ram)} there. The sign $\d \in \{ \pm\} $ is the sign for which the Jacquet-Ye linear form $P$ of \S\ref{linear_forms_CBC} is $\d$-invariant for the action of the Galois involution $\s$ on $M = H^5_{cusp}(Y_E(K_f),\L_{\mu_E}(\C))_{\m_\Pi}$.

\begin{corollaire}
Let $\pi$ and $\Pi= \mathrm{BC}(\pi)$ be as above. Assume that the Galois representation associated with $\Pi$ is residually absolutely irreducible. Assume that the cohomological weight $\mu$ of $\pi$ is $p$-small, and that $p \nmid 6N_{E/\Q}(\n)h_E(\n) D_E$. Then, under assumption $(H = \TT_\Q)$, the following divisibility holds:
$$
\nu_\pi \cdot \Om_2(\pi) \cdot \Om_3(\pi^\vee) \mid  \Om_5(\Pi^\vee,\s,- \d) \cdot u 
$$
where $\d \in \{\pm1\}$ is the sign of \ref{CBC_divisibility_NSD}, $\nu_\pi := \eta_\pi \cdot \eta_\pi(M_{\TT_\Q})[+]^{-1} \in \O$ and $u \in \overline{\Q}^\x$ is the constant (conjecturally equal to $1$) of \ref{CBC_divisibility}.
\end{corollaire}

\bibliographystyle{alpha}
\bibliography{../../Notes/Bibliographie/biblio_report}

\end{document}